\documentclass[oneside,english]{amsart}
\usepackage[T1]{fontenc}
\usepackage[latin9]{inputenc}
\usepackage{babel}
\usepackage{array}
\usepackage{amsthm}
\usepackage{amstext,mathtools}
\usepackage{amssymb}
\usepackage{graphicx}
\usepackage[unicode=true,
 bookmarks=true,bookmarksnumbered=false,bookmarksopen=false,
 breaklinks=false,pdfborder={0 0 1},backref=false,colorlinks=false]
 {hyperref}
\hypersetup{pdftitle={Disjoint sparsity  for signal separation and hybrid inverse problems},
 pdfauthor={Giovanni S. Alberti and Habib Ammari}}

\makeatletter

\providecommand{\tabularnewline}{\\}

\theoremstyle{plain}
\newtheorem{thm}{\protect\theoremname}
  \theoremstyle{plain}
  \newtheorem{prop}{\protect\propositionname}
  \theoremstyle{definition}
  \newtheorem{defn}{\protect\definitionname}
  \theoremstyle{definition}
  \newtheorem{example}{\protect\examplename}
  \theoremstyle{plain}
  \newtheorem{lem}{\protect\lemmaname}
  \theoremstyle{remark}
  \newtheorem{rem}{\protect\remarkname}

\usepackage{tikz}

\usepackage{enumitem}
\setlist{leftmargin=20pt}

\usepackage{color}

\@ifundefined{showcaptionsetup}{}{%
 \PassOptionsToPackage{caption=false}{subfig}}
\usepackage{subfig}
\makeatother

  \providecommand{\definitionname}{Definition}
   \providecommand{\examplename}{Example}
  \providecommand{\lemmaname}{Lemma}
  \providecommand{\propositionname}{Proposition}
  \providecommand{\remarkname}{Remark}
\providecommand{\theoremname}{Theorem}

\begin{document}
\global\long\def\R{\mathbb{R}}

\global\long\def\N{\mathbb{N}}

\global\long\def\C{\mathbb{C}}

\global\long\def\Cl{C}

\global\long\def\tr{\mathrm{tr}}

\global\long\def\phi{\varphi}

\global\long\def\epsilon{\varepsilon}

\global\long\def\div{{\rm div}}

\global\long\def\ceil{{\rm ceil}}

\global\long\def\curl{{\rm curl}}

\global\long\def\supp{{\rm supp}\,}

\global\long\def\sp{{\rm span}}

\global\long\def\ii{\mathbf{i}}

\global\long\def\bo{\partial\Omega}

\global\long\def\e{{\rm e}}

\global\long\def\y{\tilde{y}}

\DeclarePairedDelimiter{\abs}{\lvert}{\rvert}
\DeclarePairedDelimiter{\norm}{\lVert}{\rVert}

\title[Disjoint sparsity  for hybrid imaging]{Disjoint sparsity for signal separation and applications to hybrid  inverse problems in medical imaging}
\author{Giovanni S.\ Alberti}
\address{Department of Mathematics and Applications,  \'Ecole Normale Sup\'erieure, 45 rue d'Ulm, 75005 Paris, France.} 
\email{giovanni.alberti@ens.fr}
\author{Habib Ammari}
\address{Department of Mathematics and et Applications,  \'Ecole Normale Sup\'erieure, 45 rue d'Ulm, 75005 Paris, France.} 
\email{habib.ammari@ens.fr}
\thanks{This work was supported by the ERC Advanced Grant Project MULTIMOD-267184.}
\date{August 2, 2015}

\begin{abstract}
The main focus of this work is the reconstruction of the signals $f$
and $g_{i}$, $i=1,\dots,N$, from the knowledge of their sums $h_{i}=f+g_{i}$,
under the assumption that $f$ and the $g_{i}$'s can be sparsely
represented with respect to two different dictionaries $A_{f}$ and
$A_{g}$. This generalizes the well-known ``morphological component
analysis'' to a multi-measurement setting. The main result of the
paper states that $f$ and the $g_{i}$'s can be uniquely and stably
reconstructed by finding sparse representations of $h_{i}$ for every
$i$ with respect to the concatenated dictionary $[A_{f},A_{g}]$,
provided that enough incoherent measurements $g_{i}$ are available.
The incoherence is measured in terms of their mutual disjoint sparsity.

This method finds applications in the reconstruction procedures of
several hybrid imaging inverse problems, where internal data are measured.
These measurements usually consist of the main unknown multiplied
by other unknown quantities, and so the disjoint sparsity approach
can be directly applied. As an example, we show how to apply the method
to the reconstruction in quantitative photoacoustic tomography, also
in the case when the Grüneisen parameter, the optical absorption and
the diffusion coefficient are all unknown.
\end{abstract}
\maketitle

\section{Introduction}

Hybrid, or coupled-physics, inverse problems have been extensively
studied over the last years, both from the mathematical and the experimental
points of view. A hybrid imaging modality consists in the combination
of two types of techniques, one exhibiting the high contrast of tissues
and a second one providing high resolution. Thus, the main drawbacks
of the standard imaging modalities can be overcome, at least theoretically.
Many combinations have been considered, such as optical and acoustic
waves (photoacoustic tomography \cite{2011-kuchment-kunyansky}),
electric currents and ultrasounds (ultrasound modulated EIT \cite{cap2008})
or microwaves and ultrasounds (thermoacoustic tomography \cite{2011-kuchment-kunyansky}).
The reader is referred to \cite{ammaribook, 2012-kuchment,bal2012_review,alberti-capdeboscq-2014,alberti-capdeboscq-2015, seobook}
for a review on the mathematical aspects related to hybrid imaging
problems.

In general, the inversion for such problems involves two steps. In
the first step, an inverse problem related to the high resolution
wave provides certain internal measurements. Such internal data are
usually functionals of the unknown parameters and of certain solutions
of partial differential equations (the unknowns are normally the coefficients
of the PDE). In the second step, the unknown parameter has to be reconstructed
from the knowledge of the internal measurements. This is sometimes referred
to as the quantitative step, since the information on the tissue properties
contained in the internal data is only qualitative. In this paper
we suppose that the first inversion has been performed, and focus
only on the second step.

The quantitative step is normally solved with PDE-based methods, by
combining the internal data with the PDE modeling the problem. Such
approach is sometimes very powerful in the reconstruction \cite{bal2010inverse,cap2011,bal2012_review,2014-albertigs,2013-albertigs}.
However, there may be difficulties in using these methods. First,
the PDE model may be accurate only in some circumstances but not in
others \cite{bal2012_review}. Second, even if the PDE model is accurate,
there may be too many unknowns to have unique reconstruction \cite{bal-kui-2011}.
Third, even in cases when the reconstruction is unique, this may require
the differentiation of the data \cite{cap2011,ammari2012quantitative,alberti2013multiple},
which is known to be an unstable process, or may require additional
assumptions to be satisfied \cite{cap2011,alberti2013multiple,bal2012inversediffusion,ammari2012quantitative,alberti-ammari-ruan-2014}.

The main focus of this paper is an alternative approach to such problem
based on the use of sparse representations, as it was first done by
Rosenthal et al.\ in quantitative photoacoustic tomography (QPAT)
\cite{2009-razansky}. The internal data in a domain $\Omega$ can
be often expressed as the product of the unknown(s) and an expression
involving the solutions of the PDE. (For example, in QPAT the internal
data have the form $H=\Gamma\mu u$, where $\Gamma$ is the Grüneisen
parameter, $\mu$ is the optical absorption and $u$ is the light
intensity.) Taking the logarithm, the inversion corresponds to recovering
two functions $f$ and $g$ from the knowledge of their sum
\[
h(x)=f(x)+g(x),\qquad x\in\Omega.
\]
This problem is, in general, clearly unsolvable. However, it is possible
to exploit the different levels of smoothness of $f$ and $g$. Indeed,
since $f$ represents a property of the medium, such as the log conductivity,
it is typically highly discontinuous. On the other hand, $g$ is an
expression involving the solutions of a PDE, and as such enjoys higher
regularity properties. As a consequence, $f$ and $g$ have different
features, and this can be used to separate them by using a sparse
representation approach.

Two signals $f,g\in\R^{n}$ can be reconstructed from the knowledge
of their sum $h=f+g$ provided that they have different characteristics.
More precisely, they need to be sparsely represented, i.e., with few
atoms, with respect to two incoherent dictionaries $A_{f}$ and $A_{g}$.
This method is usually called ``morphological component analysis''
(MCA), and was introduced by Starck et al.\ in \cite{2004-starck-MCA}
(see  \cite{2001-bofill-zibulevski,2001-donoho-huo,2003-donoho-elad,2005-georgiev-theis-cichocki,2006-MMCA,2009-bobin-moudden-fadili-starck,2010-donoho-kutyniok,2014-kutyniok,2014-toumi-caldarelli-torresani,2013-mccoy-tropp}
for related works and \cite{2012-kutyniok} for a survey on the topic. In particular, Donoho and Kutyniok \cite{2010-donoho-kutyniok,2014-kutyniok} first provided a theoretical foundation of geometric image separation into point and curve singularities by using tools from sparsity methodologies. If compared to these works, the novelty of this paper lies in the particular structure of the measurements, as we now describe.

In this work, motivated by hybrid imaging techniques, where multiple measurements
with the parameters fixed can be taken,  we extend this
method to a multi-measurement setting. In general terms, this corresponds
to the reconstruction of $f$ (and $g_{i}$) from the knowledge of
their sums
\[
h_{i}=f+g_{i},\qquad i=1,\dots,N.
\]
We prove that the MCA approach gives unique and stable reconstruction,
provided that enough incoherent measurements $g_{i}$'s are available.
The incoherence is measured in terms of their mutual disjoint sparsity.
In vague terms, the atoms from $A_{g}$ used to represent $g_{i}$
should change for different measurements (see Definition~\ref{def:complete}
for the precise conditions). Numerical simulations show that taking
several solutions to the relevant PDE yields the necessary incoherence. 

As an example, we discuss the inversion for QPAT, both in the simpler
case when $\Gamma=1$ and in the case with non-constant $\Gamma$
in the diffusive regime for light propagation. For the $\Gamma=1$
case, this method has the advantages of being very robust to noise
and of not requiring a particular model for light propagation, if
compared to the PDE-based approaches \cite{bal2010inverse,2010-bal-jollivet-jugnon}.
In the case when $\Gamma\neq1$ there is no uniqueness in the reconstruction,
even with multiple measurements \cite{bal-kui-2011}; if the parameters
are piecewise constant, uniqueness can be guaranteed, but the inversion
may be very sensible to noise \cite{2014-scherzer-naetar}. We propose
a combination of the disjoint sparsity signal separation method and
of the PDE method, which provides a satisfactory reconstruction, without
requiring piecewise constant parameters. Numerical simulations are
provided.

This work is structured as follows. In Section~\ref{sec:Sparse-representations-and}
we recall the basic notions related to sparse representations and
present the method of morphological component analysis. In Section~\ref{sec:Disjoint-sparsity-for},
the signal separation method based on multiple measurements and disjoint
sparsity is described in detail and the main reconstruction result
is proved. The numerical implementation of the method, together with
the possible choices for the dictionaries $A_{f}$ and $A_{g}$ in
hybrid imaging, are discussed in Section~\ref{sec:Numerical-implementation}.
Next, this method is applied to hybrid imaging in Section~\ref{sec:Applications-to-hybrid},
and several numerical simulations are provided. Some concluding
remarks are contained in Section~\ref{sec:Conclusions}. The proofs of the uncertainty principles stated in Section~\ref{sec:Disjoint-sparsity-for} are given in Appendix~\ref{sec:appendix}.

\section{\label{sec:Sparse-representations-and}Sparse representations and
morphological component analysis}

In this section, we introduce the basic notions related to sparse
representations and morphological components analysis.

\subsection{Introduction to sparse representations}

This presentation follows \cite{2009-donoho-review}. Let $n$ be
the dimension of the signal space, namely we shall consider signals
$f\in\R^{n}$. Since in this paper we shall consider only images,
we should think of $n$ as being the resolution of the image, namely
$n=d\times d$, where $d$ is the number of pixels in each row and
column. However, in this section we shall use the more general notation
$f\in\R^{n}$, and think of a signal as a column vector of length
$n$. We now discuss how a signal can be represented as a superposition
of given atoms in a fixed dictionary. More precisely, let $A\in\R^{n\times m}$
be a dictionary, namely a collection of $m$ atoms, that are the column
vectors of $A$. We assume $m\ge n$ and that $A$ has full rank.
Thus, it is always possible to express $f$ as a linear combination
of these atoms, i.e. to write
\begin{equation}
f=Ay\label{eq:f=00003DAy}
\end{equation}
for some vector of coefficients $y\in\R^{m}$. The most common situation
is when $m=n$ and $A$ is an orthonormal basis: in this case the
coefficient $y$ is uniquely determined. However, the situation we
are interested in is when $m>n$. In this case the dictionary $A$
is redundant, since $f$ can be represented in many different ways
as a combination of the atoms in $A$. In other words, the system
\eqref{eq:f=00003DAy} is underdetermined and has many solutions $y\in\R^{m}$. 

The key observation is that this non-uniqueness can be exploited by
selecting the best representation $y$. One way to measure the quality
of a representation $y$ is given by its sparsity, which can be quantified
by the number of non-zero coefficients of $y$
\[
\left\Vert y\right\Vert _{0}:=\#\{\alpha\in\{1,\dots,m\}:y(\alpha)\neq0\},
\]
where the symbol $\#$ denotes the cardinality of a set. The representation
\eqref{eq:f=00003DAy} is called \emph{sparse} if $\left\Vert y\right\Vert _{0}\ll m$.
From the theoretical point of view, the sparsest representation can
be found by minimizing the following problem
\begin{equation}
\min_{y\in\R^{m}}\left\Vert y\right\Vert _{0}\quad\text{subject to \ensuremath{Ay=f.}}\label{eq:sparse min}
\end{equation}
The practical search for the minimum poses highly non trivial difficulties,
and the description of the main issues goes beyond the scope of this
work. Algorithms such as Matching Pursuit \cite{1993-mallat-zhang}
or Basis Pursuit \cite{2001-basispursuit} can often be successfully
used to find the sparsest solution. More details will be given in
Section~\ref{sec:Numerical-implementation}.

The choice of the dictionary $A$ clearly plays a fundamental role
in this context. Indeed, a signal $f$ admits a sparse representation
with respect to a dictionary $A$ if $f$ can be written as combination
of few atoms in $A$. Therefore, the dictionary $A$ has to be chosen
to capture the main features of the signals we consider. Many choices
of dictionaries for images are available, and a detailed discussion
is presented in Section~\ref{sec:Numerical-implementation}.

In the presence of noise, it is helpful to consider a relaxation of
\eqref{eq:sparse min} and to allow a small error between the signal
$f$ and its representation in terms of the atoms of $A$. Thus, the
minimization problem becomes
\[
\min_{y\in\R^{m}}\left\Vert y\right\Vert _{0}\quad\text{subject to \ensuremath{\left\Vert Ay-f\right\Vert _{2}\le\epsilon,}}
\]
for some $\epsilon>0$, or equivalently
\[
\min_{y\in\R^{m}}\left\Vert y\right\Vert _{0}+\lambda\left\Vert Ay-f\right\Vert _{2},
\]
for a suitable Lagrange multiplier $\lambda>0$.

\subsection{Introduction to morphological component analysis (MCA)}

One of the relevant applications of sparse representations is related
to the separation of a signal into its constitutive components, provided
they have different features. We shall describe the method discussed
in \cite{2004-starck-MCA}. Suppose that a signal $h\in\R^{n}$ can
be written as a sum
\[
h=\tilde{f}+\tilde{g},
\]
with $\tilde{f},\tilde{g}\in\R^{n}$. The problem under consideration
is the reconstruction of $\tilde{f}$ and $\tilde{g}$ from the knowledge
of $h$, under the assumption that $\tilde{f}$ and $\tilde{g}$ have
distinctive characteristics. This assumption can be expressed in terms
of sparse representations. Namely, suppose that there exist two dictionaries
$A_{f}\in\R^{n\times m_{f}}$ and $A_{g}\in\R^{n\times m_{g}}$ such
that:
\begin{enumerate}
\item $\tilde{f}$ can be sparsely represented with respect to $A_{f}$
but not with respect to $A_{g}$;
\item and $\tilde{g}$ can be sparsely represented with respect to $A_{g}$
but not with respect to $A_{f}$.
\end{enumerate}
Under these heuristic conditions (which will be made precise below), a strategy to find $\tilde{f}$ and $\tilde{g}$
may be to find a sparse representation $y=\left[\begin{smallmatrix}y_{f}\\
y_{g}
\end{smallmatrix}\right]$ of $h$ with respect to the concatenated dictionary $A=[A_{f,}A_{g}]$
and then to write $f=A_{f}y_{f}$ and $g=A_{g}y_{g}$. As we have
seen before, the sparse representation $y$ is the minimum of the
minimization problem 
\[
\min_{y\in\R^{m_{f}+m_{g}}}\left\Vert y\right\Vert _{0}\quad\text{subject to \ensuremath{[A_{f,}A_{g}]\left[\begin{smallmatrix}y_{f}\\
y_{g}
\end{smallmatrix}\right]=h,}}
\]
or, in the presence of noise, of
\[
\min_{y\in\R^{m_{f}+m_{g}}}\left\Vert y\right\Vert _{0}\quad\text{subject to \ensuremath{\left\Vert [A_{f,}A_{g}]\left[\begin{smallmatrix}y_{f}\\
y_{g}
\end{smallmatrix}\right]-h\right\Vert _{2}\le\epsilon}}.
\]

Even though this procedure is successful in many practical cases \cite{2004-starck-MCA,2005-inpainting,2009-razansky},
a proof of the correct reconstruction, i.e.\ $f=\tilde{f}$ and $g=\tilde{g}$,
is only valid in an ideal situation, which we now describe. In the
case when $A_{f}$ and $A_{g}$ are both orthonormal sets, the proof
is based on the following uncertainty principle \cite[Theorem 1]{2002-elad-uncertainty}.
\begin{prop}
\label{prop:uncertainty}Let $\{a_1,\dots,a_{m_A}\}$ and $\{b_1,\dots,b_{m_B}\}$ be two orthonormal sets of vectors in $\R^n$. Take a vector $h\in\R^{n}\setminus\{0\}$
and suppose it has the following representations
\[
h=Ay_{A}=By_{B}
\]
with respect to  $A=[a_{1},\dots,a_{m_A}]$ and
$B=[b_{1},\dots,b_{m_B}]$. Then
\[
\left\Vert y_{A}\right\Vert _{0}+\left\Vert y_{B}\right\Vert _{0}\ge2/M,
\]
where
\begin{equation}\label{eq:mutual}
 M=\max_{i,j}|(a_{i},b_{j})_{2}|
\end{equation}
is the \emph{mutual coherence}.
\end{prop}
We provide a proof for completeness.
\begin{proof}
Let $S_A$ and $S_B$ denote the supports of $y_A$ and $y_B$, respectively. By assumption we have $y_B(j)=(h,b_j)_2=\sum_{i\in S_A} y_A(i) (a_i,b_j)_2$, whence by Cauchy-Schwartz inequality we obtain $|y_B(j)|\le \left\Vert y_A \right\Vert_2 \sqrt{\# S_A} M$. Thus, since $\left\Vert y_{A}\right\Vert _{0}=\# S_A$ and $\left\Vert y_{B}\right\Vert _{0}=\# S_B$, this implies 
\[
\left\Vert y_B \right\Vert_2\le\left\Vert y_A \right\Vert_2\sqrt{(\# S_A)(\# S_B)} M\le \left\Vert y_A \right\Vert_2 (\left\Vert y_{A}\right\Vert _{0}+\left\Vert y_{B}\right\Vert _{0}) \frac{M}{2}.
\]
This concludes the proof, since the orthonormality of  $A$ and $B$ implies $\left\Vert y_A \right\Vert_2=\left\Vert h \right\Vert_2=\left\Vert y_B \right\Vert_2$ (Lemma~\ref{lem:properties} part 1).
\end{proof}
In general, it is easy to see that if $A$ and $B$ are orthonormal bases then $1/\sqrt{n}\le M\le 1$ \cite{2001-donoho-huo}. As a simple consequence of this result \cite[Theorem 2]{2002-elad-uncertainty},
we have that if $y^{1}\in\R^{2n}$ and $ $$y^{2}\in\R^{2n}$ are
two representations of $ $$h$ with respect to the concatenated dictionary
$A=[A_{f,}A_{g}]$, then
\[
\left\Vert y^{1}\right\Vert _{0}+\left\Vert y^{2}\right\Vert _{0}\ge2/M.
\]
Therefore, if $\tilde{f}$ and $\tilde{g}$ have representations $\y_{f}$
and $\y_{g}$ satisfying $\left\Vert \y_{f}\right\Vert _{0}+\left\Vert \y_{g}\right\Vert _{0}<1/M$,
then the above method provides the correct reconstruction.

In practice, the assumption $\left\Vert \y_{f}\right\Vert _{0}+\left\Vert \y_{g}\right\Vert _{0}<1/M$
is almost never satisfied, and so the above argument remains only
a theoretical speculation. However, when the multi-measurement case
is considered, correct and stable reconstruction can be rigorously
proved. This theoretical result is also validated by several numerical
simulations. These aspects are discussed in the following sections.

\section{\label{sec:Disjoint-sparsity-for}Disjoint sparsity for morphological
component analysis}

\subsection{Introduction and main assumptions}

Motivated by several hybrid imaging modalities (see Section~\ref{sec:Applications-to-hybrid}),
we generalize the MCA problem to a multi-measurement setting. The
reader is referred to \cite{2001-bofill-zibulevski,2005-georgiev-theis-cichocki,2006-MMCA}
for other similar variations.

Let $h_{1},\dots,h_{N}\in\R^{n}$ be $N$ signals that can be decomposed
as
\[
h_{i}=\tilde{f}+\tilde{g}_{i},\qquad i=1,\dots,N,
\]
with $\tilde{f},\tilde{g}_{i}\in\R^{n}$. We want to study the problem
of finding $\tilde{f}$ and $\tilde{g}_{i}$ from the knowledge of
$h_{i}$ for $i=1,\dots,N$. The case $N=1$ was discussed in the
previous section. We shall show that as $N$ becomes bigger, the above
problem becomes much more treatable, and that the sparsity approach
introduced before always provides the correct reconstruction, also
in the presence of noise. As before, let $A_{f}$ and $A_{g}$ be
the dictionaries with respect to which $\tilde{f}$ and $\tilde{g}_{i}$
have sparse representations, respectively. Note that all the $\tilde{g}_{i}$'s
can be sparsely represented with the same dictionary $A_{g}$: this is a crucial assumption of this approach. Assume that the atoms of $A_{f}$  are normalized to $1$ and that 
\begin{equation}\label{ass:A_g}
\text{ the atoms of $A_g$ constitute an orthonormal set of $\R^n$.}
\end{equation}
Thus, $A_g$ can be completed to an orthonormal basis $[A_g,A_g^\perp]$ for some $A_g^\perp\in\R^{n\times (n-m_g)}$.

The reconstruction method applied to this case consists in the minimization
of
\begin{equation}
\min_{y\in\R^{m_{f}+Nm_{g}}}\left\Vert y\right\Vert _{0}\quad\text{subject to $\bigl\Vert A_{f}y_{f}+A_{g}y_{g}^{i}-h_{i}\bigr\Vert_{2}\le\epsilon$, $i=1\dots,N$},
\label{eq:min multi}
\end{equation}
where we have used the notation $y=\,^{t}[^{t}y_{f},\,^{t}y_{g}^{1},\dots,\,^{t}y_{g}^{N}]$. Here, the superscript $t$ denotes the transpose. 
To model the case with added noise, we write
\begin{equation}
h_{i}=\tilde{f}+\tilde{g}_{i}+n_{i},\qquad i=1,\dots,N,\label{eq:hi noise}
\end{equation}
where $\tilde{f}$ and the $\tilde{g}_{i}$'s represent the true signals,
the $h_{i}$'s are the measured signals and $n_{i}$ is such that
\begin{equation}
\left\Vert n_{i}\right\Vert _{2}\le\eta,\qquad i=1,\dots,N\label{eq:noise}
\end{equation}
for some small $\eta>0$. 

In the applications we have in mind (Section~\ref{sec:Applications-to-hybrid}),
the signal $\tilde{f}$ represents (the logarithm of) a physical constitutive
parameter, while the $\tilde{g}_{i}$'s usually quantify the injected
fields, e.g., the electric field or the light intensity. As such, $\tilde{f}$
is given and fixed, and we have no control on it. On the other hand,
the $\tilde{g}_{i}$'s come from the measurements, and can be indirectly
controlled. More precisely, the $\tilde{g}_{i}$'s depend on the solutions
of a certain PDE, whose coefficients are unknown, but whose boundary
values can be chosen: in this sense the $\tilde{g}_{i}$'s can be
controlled. It is therefore natural to give some assumptions on the
$\tilde{g}_{i}$'s.

The main requirements are that the $\tilde{g}_{i}$'s should be \emph{sufficiently
many and incoherent}%
\footnote{Similar assumptions of enough independent measurements are required
also when using PDE methods for hybrid inverse problems (see $\S$~\ref{sub:Introduction-to-hybrid}).%
}. This will be mainly expressed by means of their disjoint sparsity
with respect to $A_{g}$. (Disjoint sparsity was used in \cite{2006-MMCA},
while joint sparsity has been extensively used in compressive sensing
\cite{2005-MMV,2006-chen-huo-MMV}.) We shall therefore write
\begin{equation}
\bigl\Vert A_{f}\y_{f}-\tilde{f}\bigr\Vert_{2}\le\rho_f,\quad\left\Vert A_{g}\y_{g}^{i}-\tilde{g}_{i}\right\Vert _{2}\le\rho_g,\qquad i=1,\dots,N\label{eq:thm-assumption}
\end{equation}
for some $\rho_f,\rho_g>0$.
The approximation allows a small error between the true signals and
their sparse representations. We shall prove that under suitable assumptions,
a minimizer of \eqref{eq:min multi} provides the correct reconstruction,
up to a factor that is small in $\epsilon:=\rho_f+\rho_g+\eta$. In terms of the coefficient vectors
$\{\y_{g}^{1},\dots,\y_{g}^{N}\}$, the required assumptions can be
written as follows.
\begin{defn}
\label{def:complete}Take $\beta,D>0$,  $N\in\N^{*}$, $\y_{f}\in\R^{m_{f}}$
and $\y_{g}^{1},\dots,\y_{g}^{N}\in\R^{m_{g}}$. We say that $\{\y_{g}^{1},\dots,\y_{g}^{N}\}$
is a $(\beta,D)$-complete set of measurements if the following two conditions
hold true:
\begin{description}
\item [{CS1}] if $|\y_{g}^{i}(\alpha)-\y_{g}^{j}(\alpha)|\le\beta$   and $\y_{g}^{i}(\alpha)\y_{g}^{j}(\alpha)\neq 0$ for
some $\alpha\in\{1,\dots,m_{g}\}$ then $i=j$;
\item [{CS2}] for every $q\in\R^{m_f}$ such that $\left\Vert A_f q \right\Vert_2> D$ and $\left\Vert ^t\!A_g^\perp A_f q \right\Vert_2\le 2/3$
there holds
\begin{equation}
\#(\supp q\setminus\supp\y_{f})+\sum_{i=1}^{N}\#( S_q\setminus\supp\y_{g}^{i})
>\#\bigcup_{i=1}^{N}\supp\y_{g}^{i}+\left\Vert \y_{f}\right\Vert _{0},\label{eq:CS2}
\end{equation}
where $S_q=\{\alpha:|(^t\!A_{g}A_f q)(\alpha)|\ge 1\}$.
\end{description}
\end{defn}
Let us comment on these two requirements. The first condition, CS1,
is a constraint on the incoherence of the $\tilde{g}_{i}$'s in terms
of the values of their coefficient vectors. More precisely, the coefficients
relative to the same atom for two different $\tilde{g}_{i}$'s cannot
be too close. This assumption could be relaxed by allowing a fixed
number of $\y_{g}^{i}$ to have the same coefficients, but for simplicity
of exposition we have decided to omit this generalization.

While CS1 is mainly a technical hypothesis, condition CS2 is the main
assumption related to the disjoint sparsity of the $\tilde{g}_{i}$'s.
Indeed, when the sets $\supp\y_{g}^{i}$ are small (the representation
is sparse) and substantially change when $i$ varies (disjoint), then
it becomes possible to satisfy CS2 by taking $N$ large enough, since
the right-hand side is bounded by $\left\Vert \y_{f}\right\Vert _{0}+m_{g}$. Note that CS2
can always be satisfied by choosing the $\y_{g}^{i}$ so that $\#\{i:\y_{g}^{i}(\alpha)=0\}>\#\bigcup_{i}\supp\y_{g}^{i}+\left\Vert \y_{f}\right\Vert _{0}$
for every $\alpha$, but this represents only a worst case scenario. In general, the smaller $\supp\y_{f}$ and $\supp\y_{g}^{i}$
are, the easier it becomes to satisfy CS2.

\begin{rem}\label{rem:noise}
It is expected that the number of measurements $N$ should increase as the noise level $\eta$ becomes bigger. Indeed, to have a good reconstruction, we need to keep $\epsilon=\rho_f+\rho_g+\eta$ small. Thus, if the noise level $\eta$ increases, the quantities $\rho_f$ and $\rho_g$ have to become smaller. In other words, the sparse approximations in \eqref{eq:thm-assumption} have to be more precise, which in turn requires  the support of  $\y_f$ and $\y_g^i$ to be bigger, and so a higher number of measurements is needed to satisfy CS2, as we heuristically observed above (see Example~\ref{exa:comb} below for a concrete example).
\end{rem}

\subsection{Uncertainty principles}
As we have already pointed out, the smaller $\supp\y_{f}$ and $\supp\y_{g}^{i}$
are, the easier it becomes to satisfy CS2. Moreover, when $A_{f}$
and $A_{g}$ are two orthonormal bases, by Proposition~\ref{prop:uncertainty}
we have
\[
\bigl\Vert \,q\bigr\Vert_{0}+\left\Vert ^t\!A_{g}A_f q\right\Vert _{0}\ge2/M,
\]
where $M$ is the mutual coherence of $A_{f}$ and $A_{g}$ defined in \eqref{eq:mutual}. However, this inequality is not directly applicable to our context, since $\left\Vert ^t\!A_{g}A_f q\right\Vert _{0}$ does not appear explicitly in \eqref{eq:CS2} and  $A_g$ will not be a basis in the applications. The following generalization of the uncertainty principle guarantees that the same bound holds, provided that $D$ is big enough. The proof is given in Appendix~\ref{subsec:appendix}.
\begin{prop}\label{prop:normalized_UP}
Assume that $A_f$  is an  orthonormal basis of $\R^n$ and that $A_g$ is an orthonormal set of $\R^n$ and let $M$ denote their mutual coherence.  There exists $D>0$ such that
\[
 \bigl\Vert q \bigr\Vert_{0}+\#\{\alpha:|(^t\!A_{g}A_f q)(\alpha)|\ge 1\}\ge2/M,
\]
for every $q\in\R^n$ with $\left\Vert A_f q\right\Vert _{2}>D$ and $\left\Vert ^t\!A_g^\perp A_f q\right\Vert _{2}\le 2/3$.
\end{prop}
\begin{rem}\label{rem:N1}
In the single measurement case (if  $A_f$ is an  orthonormal basis), CS1 and CS2 correspond to the condition $\left\Vert \y_{f}\right\Vert _{0}+\left\Vert \y_{g}\right\Vert _{0}<1/M$, thereby reobtaining the hypothesis discussed in the previous section. Indeed, if $N=1$ then CS1 is immediately satisfied. Moreover, CS2 becomes
\[
 \#\supp q\setminus\supp\y_{f}+\#\{\alpha:|(^t\!A_{g}A_f q)(\alpha)|\ge 1\}\setminus\supp\y_{g}
>\left\Vert \y_{g}\right\Vert _{0}+\left\Vert \y_{f}\right\Vert _{0}.
\]
By Proposition~\ref{prop:normalized_UP}, the left hand side is bounded by below by $2/M-(\left\Vert \y_{f}\right\Vert _{0}+\left\Vert \y_{g}\right\Vert _{0})$, and so the above inequality is satisfied provided that $\left\Vert \y_{f}\right\Vert _{0}+\left\Vert \y_{g}\right\Vert _{0}<1/M$.
\end{rem}
The more incoherent the two bases are, the bigger $2/M$ is, and therefore
the bigger the left-hand side in \eqref{eq:CS2} is. It is not a surprise
that the incoherence of the bases makes these assumptions easier to
be satisfied, since this was the starting point of the MCA discussed
in the previous section. 

Let us discuss a simple example to show the relation between the sparsity of the coefficients and the number of measurements $N$.
\begin{example}\label{exa:comb}
Let us consider the simplest problem of the separation of spikes and sinusoids in 1D. Let $A_f=I_n\in\R^{n\times n}$ be the identity matrix and $A_g$ be the Fourier basis. Their mutual coherence is $M=1/\sqrt{n}$. Let $\y_f\in\R^n$ be such that $\left\Vert \y_f\right\Vert_0=k$ and $\y_g^i\in\R^n$ be such that $\left\Vert \y_g^i\right\Vert_0=l$ and with disjoint support, for $i=1,\dots,N$. We would like to investigate the link between the number of measurements $N$ and the sparsity of $\y_f$ and $\y_g^i$, i.e.\ with the quantities $k$ and $l$.

Thanks to the assumption on the supports of the $\y_g^i$'s, condition CS1 is automatically satisfied for any $\beta>0$. As far as CS2 is concerned, unfortunately we cannot check its validity for all possible choices of $q$. Thus, we make an heuristic choice with one of the worst possibilities, the Dirac comb, for which the inequality of the uncertainty principle becomes an equality. Set $q_\delta(\alpha)=1$ when $\alpha$ is a multiple of $\sqrt{n}$ and $q_\delta(\alpha)=0$ otherwise. It turns out that $A_g q_\delta=q_\delta$, and so $\left\Vert q_\delta \right\Vert_0=\# S_{q_\delta} =\sqrt{n}$. Provided that $l<\sqrt{n}$, it turns out that condition CS2 for $y_\delta$ is satisfied if
\[
 \frac{2k}{N+1}+l<\sqrt{n}.
\]
Thus, the bigger the number of measurements is, the bigger $k$ and $l$ can be.

Assume now that $A_g$ contains only a subset of the Fourier basis of cardinality $m_g$ (as we shall do in the numerical simulations), and select a vector $q_\delta$ such that  $\left\Vert q_\delta \right\Vert_0=\# S_{q_\delta}=\sqrt{n}$ as before. Thus, CS2 is satisfied if
\[
 2k<N(\sqrt{n}-l)+\sqrt{n}-m_g,
\]
independently of the supports of the $\y_g^i$'s. For a fixed $l<\sqrt{n}$, the higher $N$ is, the bigger $k$  can be.
\end{example}

When considering wavelets and the Fourier basis, as in the numerical simulations of this paper, there is the added difficulty that the mutual coherence is not small, but in fact of order 1, which makes the uncertainty principle discussed above of no practical use. However, if we consider Haar wavelets and low frequency trigonometric polynomials, as in Section~\ref{sec:Applications-to-hybrid},  it is possible to improve this bound (a similar result for spikes and sinusoids is given in  \cite[Lemma 1]{1989-donoho-starck}). We now discuss this result.

We consider two dimensional signals in $\R^{n}$, where $n=d\times d$ and $d=2^J$ for some $J\in\N^*$. Let $A_f$ be the associated orthonormal basis consisting of 2D periodized Haar wavelets (see $\S$~\ref{subsec:appendix2}). Let $A_g$ consist of low frequency non-constant trigonometric polynomials. More precisely, for $L=1,\dots,d/2-1$ consider the following families of real sinusoids
\begin{equation}\label{eq:realsinusoids}
 \begin{aligned}
 & \chi_l^1(\alpha)=c_l^1\sin(2\pi l_1 \alpha_1/d)\sin(2\pi l_2 \alpha_2/d),\qquad l_1,l_2=1,\dots,L,\\
 & \chi_l^2(\alpha)=c_l^2\sin(2\pi l_1 \alpha_1/d)\cos(2\pi l_2 \alpha_2/d),\qquad l_1=1,\dots,L,\,l_2=0,\dots,L,\\
  & \chi_l^3(\alpha)=c_l^3\cos(2\pi l_1 \alpha_1/d)\sin(2\pi l_2 \alpha_2/d),\qquad l_1=0,\dots,L,\,l_2=1,\dots,L,\\
   & \chi_l^4(\alpha)=c_l^4\cos(2\pi l_1 \alpha_1/d)\cos(2\pi l_2 \alpha_2/d),\qquad l\in \{0,\dots,L\}^2\setminus\{(0,0)\},
\end{aligned}
\end{equation}
where $c_l^i>0$ are suitable normalization factors chosen so that $\norm{\chi^i_l}_2=1$. Let $A_g$ be the collection of all these real sinusoids. Assumption \eqref{ass:A_g} is satisfied. The number of atoms is given by $m_g=4L^2+4L$. The proof of the following result is given in Appendix~\ref{subsec:appendix2}: it is heavily based on the structure of Haar wavelets, and so it is expected that this result would fail with smoother wavelets.
\begin{prop}\label{prop:new_uncertainty}
Let $A_f$ and $A_g$ be the dictionaries of 2D Haar wavelets and low frequency non-constant real sinusoids discussed above, respectively. Assume that $L< 2^B$ for some $B\le J-2$. There exists $D>0$ such that
\[
 \norm{q}_0\ge  \sum_{j=1}^{J-B-1} (2^{J-j}-2L)^2,
\]
for every $q\in\R^n$ with $\norm{A_f q}_2>D$ and $\norm{\,^t\!A_g^\perp A_f q}_2\le 2/3$.
\end{prop}
\begin{example}\label{exa:2}
In  Section~\ref{sec:Applications-to-hybrid} we shall set $J=7$ and $L=15$. In this case, the above inequality reads $\norm{q}_0\ge 1160$, which is sensibly better than the uncertainty principle based on the mutual coherence. Thus, arguing as in Remark~\ref{rem:N1}, in the single-measurement case, condition CS2 is satisfied provided that
\begin{equation}\label{eq:1160}
 2\norm{\y_f}_0+\norm{\y_g}_0<1160.
\end{equation}
It should be observed that in the numerical simulations we add also the constant matrix to the dictionary $A_g$. Even if this is not allowed by the above result (only 4 wavelets are needed to represent it), it seems not to create any issues for the reconstruction. This might be due to the following remark: most images are not constant.
\end{example}

\begin{rem}\label{rem:robustUP}
We conclude this part by observing that these uncertainty principles always take into account the worst case scenarios; namely,  for most vectors the minimum is not attained. Improved bounds (called \emph{robust} uncertainty principles) that hold for the overlwhelming mojority of vectors were  proved in \cite{2006-candes-romberg} for spikes and sinusoids. Moreover, it was proven that the separation problem can be succesfully solved in most cases by using $l_0$ minimization, provided that corresponding sparsity conditions are satisfied. These conditions are much weaker than the ones based on exact uncertainty principles. It would be interesting to investigate whether such extensions hold in our context too.

In view of this, while the uncertainty principle in Proposition~\ref{prop:new_uncertainty} does not contain the term corresponding to $S_q$, it seems reasonable to assume that for most vectors the quantity $\# S_q$ is not small when $\norm{q}_0$ is close to the minimum. This explains why CS2 is easily satisfied when $N$ is bigger, at least for most $q$.
\end{rem}

\subsection{Main result}

The main result of this section states that if the $\y_{g}^{i}$'s
constitute a complete set, then the signals $\tilde{f}$ and $\tilde{g}_{i}$
can be stably recovered from the knowledge of $h_{i}=\tilde{f}+\tilde{g}_{i}+n_{i}$
by minimizing \eqref{eq:min multi}.
\begin{thm}
\label{thm:main}Assume that \eqref{ass:A_g} holds true. Let $\beta,D,\rho_f,\rho_g,\eta>0$ and $N\in\N^{*}$ be such that $\epsilon:=\rho_f+\rho_g+\eta\le \beta/3$. Take 
$\tilde{y}_{f}\in\R^{m_{f}}$ and 
let $\{\y_{g}^{1},\dots,\y_{g}^{N}\}$ be $(\beta,D)$-complete. Assume that $\tilde{f},\tilde{g}_{i},n_{i},h_i\in\R^{n}$ satisfy \eqref{eq:hi noise}, \eqref{eq:noise}
and \eqref{eq:thm-assumption} and let $y_{f}\in\R^{m_{f}}$ and $y_{g}^{i}\in\R^{m_{g}}$ realize
the minimum of

\begin{equation}
\min_{y\in\R^{m_{f}+Nm_{g}}}\left\Vert y\right\Vert _{0}\quad\text{subject to $\bigl\Vert A_{f}y_{f}+A_{g}y_{g}^{i}-h_{i}\bigr\Vert_{2}\le\epsilon$, $i=1\dots,N$}.\label{eq:thm-min}
\end{equation}
Then for every $i=1,\dots,N$
\[
\bigl\Vert A_{f}y_{f}-\tilde{f}\bigr\Vert_{2}\le(3D+1) \epsilon,\quad\left\Vert A_{g}y_{g}^{i}-\tilde{g}_{i}\right\Vert _{2}\le (3D+2)\epsilon.
\]
\end{thm}
Thanks to this result, the reconstruction can be performed by minimizing
\eqref{eq:thm-min} and then taking $\tilde{f}\approx A_{f}y_{f}$
and $\tilde{g}_{i}\approx A_{g}y_{g}^{i}$. The practical details
of such minimization will be discussed in Section~\ref{sec:Numerical-implementation}.

In the proof of this theorem we shall make use of the following properties,
whose proofs are immediate.
\begin{lem}
\label{lem:properties}Let $A\in\R^{n\times m}$ constitute an orthonormal set of $\R^n$ and let $A^\perp\in\R^{n\times(n-m)}$ complete $A$ to an orthobasis of $\R^n$. The following properties hold true.
\begin{enumerate}
\item For every $v\in\R^{m}$ we have $\left\Vert Av\right\Vert _{2}=\left\Vert v\right\Vert _{2}$.
\item For every $v\in\R^{n}$ we have $\left\Vert ^t\!Av\right\Vert _{2}\le \left\Vert v\right\Vert _{2}$.
\item We have $^t\! A A=I$ and $^t\! A^\perp A =0$.

\item For every $a,b\in\R^{m}$ we have
\[
\left\Vert a+b\right\Vert _{0}=\left\Vert a\right\Vert _{0}+\#(\supp b\setminus\supp a)-\#\{\alpha:a(\alpha)=-b(\alpha)\neq0\}.
\]
\end{enumerate}
\end{lem}
We are now ready to prove Theorem~\ref{thm:main}.
\begin{proof}[Proof of Theorem~\ref{thm:main}]
Write $f:=A_{f}y_{f}$,
$g_{i}:=A_{g}y_{g}^{i}$, $e_{f}:=y_{f}-\y_{f}$, $e_{g}^{i}:=y_{g}^{i}-\y_{g}^{i}$,
 $e_{g}:=-^t\!A_{g}A_{f}e_{f}$ and $r^{i}:=e_{g}^{i}-e_{g}$. Since $y_f$ and $y_{g}^{i}$ satisfy the constraint in \eqref{eq:thm-min} we have
\[
\begin{split}
 \norm{A_fy_f-A_f\y_f+A_gy^i_g-A_g\y^i_g}_2&\le \norm{(A_fy_f+A_gy^i_g-h_i)+(h_i-A_g\y^i_g-A_f\y_f)}_2\\
 & \le \epsilon + \norm{\tilde{f}-A_f\y_f}_2 +\norm{\tilde{g_i}-A_g\y_g^i}_2 +\norm{n_i}_2\\
 &\le 2\epsilon
\end{split}
\]
where the last inequality follows from \eqref{eq:noise} and \eqref{eq:thm-assumption}. As a consequence, since $r^i=\,^t\!A_g(A_fy_f-A_f\y_f+A_gy^i_g-A_g\y^i_g)$ (Lemma~\ref{lem:properties} part 3), by Lemma~\ref{lem:properties} part 2 we obtain $\norm{r^i}_2\le 2\epsilon$, whence
\begin{equation}
 \norm{r^i}_\infty\le 2\epsilon. \label{eq:norm_ri}
\end{equation}
Moreover, by Lemma~\ref{lem:properties} parts 2 and 3 we obtain
\begin{equation}
 \norm{^t\!A_g^\perp A_f e_f}_2= \norm{^t\!A_g^\perp (A_f y_f-A_f \y_f +A_gy^i_g-A_g\y^i_g)}_2\le 2\epsilon.\label{eq:bound_A_perp}
\end{equation}

Since $y_{f}$ and $y_{g}^{i}$ realize the minimum of \eqref{eq:thm-min}
we have  $\norm{y}_0\le \norm{\y}_0$ or, equivalently,
\[
(\left\Vert y_{f}\right\Vert _{0}-\left\Vert \y_{f}\right\Vert _{0})+\sum_{i=1}^{N}(\left\Vert y_{g}^{i}\right\Vert_0 -\left\Vert \y_{g}^{i}\right\Vert_0 )\le 0.
\]
where we have set $\y=\,^{t}[^{t}\y_{f},\,^{t}\y_{g}^{1},\dots,\,^{t}\y_{g}^{N}]$. Thus, since $y_{f}=\y_{f}+e_{f}$ and $y_{g}^{i}=\y_{g}^{i}+e_{g}^{i}$,
Lemma~\ref{lem:properties} part 4 yields
\begin{multline}
\#(\supp e_{f}\setminus\supp\y_{f})-\#\{\alpha:\y_{f}(\alpha)=-e_{f}(\alpha)\neq0\}\\
+\sum_{i=1}^{N}\#(\supp e_{g}^{i}\setminus\supp\y_{g}^{i})-\sum_{i=1}^{N}\#\{\alpha:\y_{g}^{i}(\alpha)=-e_{g}^{i}(\alpha)\neq0\}\le 0.\label{eq:forth}
\end{multline}

Observe now that by \eqref{eq:norm_ri} we have
\begin{equation}
\begin{aligned} & \#\{\alpha:\y_{f}(\alpha)=-e_{f}(\alpha)\neq0\}\le\left\Vert \y_{f}\right\Vert _{0},\\
 & \supp e_{g}^{i}=\supp(e_{g}+r^{i})\supseteq\{\alpha:|e_{g}(\alpha)|\ge3\epsilon\}.
\end{aligned}
\label{eq:fifth}
\end{equation}
Since $3\epsilon\le\beta$ and $\{\y_{g}^{1},\dots,\y_{g}^{N}\}$
is $(\beta,D)$-complete, by  \eqref{eq:norm_ri} and Definition~\ref{def:complete} (condition
CS1), we have
\begin{equation}
\sum_{i=1}^{N}\#\{\alpha:\y_{g}^{i}(\alpha)=-e_{g}^{i}(\alpha)\neq0\}
=\#\bigcup_{i=1}^{N}\{\alpha:\y_{g}^{i}(\alpha)=-e_{g}^{i}(\alpha)\neq0\}\le\#\bigcup_{i=1}^{N}\supp\y_{g}^{i}.\label{eq:sixth}
\end{equation}
Inserting \eqref{eq:fifth} and \eqref{eq:sixth} into \eqref{eq:forth}
gives
\[
\#(\supp e_{f}\!\setminus\supp\y_{f})+\!\sum_{i=1}^{N}\!\#(\{\alpha\!:\!|(^t\!A_{g}A_{f}e_{f})(\alpha)|\!\ge\!3\epsilon\}\!\setminus\supp\y_{g}^{i}) 
\!\le\! \left\Vert \y_{f}\right\Vert _{0}\!+\#\bigcup_{i=1}^{N}\!\supp\y_{g}^{i}.
\]
Set $q=e_f/(3\epsilon)$. Since $\supp e_f=\supp q$ we have
\[
\#(\supp q\setminus\supp\y_{f})+\sum_{i=1}^{N}\#(S_q\setminus\supp\y_{g}^{i}) 
\le \left\Vert \y_{f}\right\Vert _{0}+\#\bigcup_{i=1}^{N}\supp\y_{g}^{i}.
\]
where $S_q=\{\alpha:|(^t\!A_{g}A_f q)(\alpha)|\ge 1\}$. Moreover, by \eqref{eq:bound_A_perp} we have $\norm{^t\!A_g^\perp A_f q}_2\le 2/3$. Thus, since $\{\y_{g}^{1},\dots,\y_{g}^{N}\}$ is $(\beta,D)$-complete,
by Definition~\ref{def:complete} (condition CS2) we obtain that $\left\Vert A_f e_f\right\Vert _{2}\le3D\epsilon$. As a result, by construction of $e_f$, \eqref{eq:thm-assumption} and the triangle inequality we have
\begin{equation}\label{eq:f-ftilde}
  \bigl\Vert A_f y_f -\tilde{f}\bigr\Vert _{2}\le3D\epsilon +\rho_f \le (3D+1)\epsilon.
\end{equation}
This proves the desired bound for $f$. It remains to prove the
corresponding estimate for $g_i=A_{g}y_{g}^i$. In order to do this, observe
that by \eqref{eq:thm-assumption}, \eqref{eq:f-ftilde} and the triangle inequality we obtain
\[
\begin{split}
\bigl\Vert g_i-\tilde{g}_i\bigr\Vert_{2}&\le\bigl\Vert (f+g_i-h_i) + (\tilde{f}-f)+ n_i\bigr\Vert_{2}\le \epsilon+ 3D\epsilon+\rho_f+\eta\le (3D+2)\epsilon.
\end{split}
\]
This concludes the proof.
\end{proof}

\section{\label{sec:Numerical-implementation}Numerical implementation}

\subsection{Orthogonal Matching Pursuit}

The simplest available algorithm for the minimization of problems
of the type
\[
\min_{y\in\R^{m}}\left\Vert y\right\Vert _{0}\quad\text{subject to \ensuremath{\left\Vert Ay-f\right\Vert _{2}\le\epsilon}},
\]
for $f\in\R^{n}$ and $A\in\R^{n\times m}$, is the Orthogonal Matching
Pursuit (OMP) \cite{1993-mallat-zhang,2009-donoho-review}. This algorithm
belongs to a wider class of methods called greedy algorithms, in which
the set of the used atoms of $A$ is increased step by step. In OMP,
the best coefficients for the atoms are recomputed at each iteration,
which makes it more efficient compared to the standard matching pursuit.
Greedy algorithms, and hence OMP, are not a priori equivalent to the minimization of the above problem, and the convergence to a minimizer is not always guaranteed, but they have been proved to perform well in most cases \cite{2004-tropp-greed}.  The reader is referred
to \cite{2009-donoho-review} for more details on this topic.

The adaptation of OMP to the problem of our interest is quite straightforward.
By Theorem~\ref{thm:main}, we need to minimize \eqref{eq:thm-min}, i.e.
\[
\min_{y\in\R^{m_{f}+Nm_{g}}}\left\Vert y\right\Vert _{0}\quad\text{subject to \ensuremath{\bigl\Vert A_{f}y_{f}+A_{g}y_{g}^{i}-h_{i}\bigr\Vert_{2}\le\epsilon,\quad i=1,\dots,N,}}
\]
given $N$ signals $h_{i}\in\R^{n}$ and two dictionaries $A_{f}\in\R^{n\times m_{f}}$
and $A_{g}\in\R^{n\times m_{g}}$. Setting
\[
A=\begin{bmatrix}A_{f} & A_{g} & 0 & \cdots & 0\\
A_{f} & 0 & A_{g} & \ddots & \vdots\\
\vdots & \vdots & \ddots & \ddots & 0\\
A_{f} & 0 & \cdots & 0 & A_{g}
\end{bmatrix},\; y=\begin{bmatrix}y_{f}\\
y_{g}^{1}\\
y_{g}^{2}\\
\vdots\\
y_{g}^{N}
\end{bmatrix}\;\text{and}\; h=\begin{bmatrix}h_{1}\\
h_{2}\\
\vdots\\
h_{N}
\end{bmatrix},
\]
the above problem is equivalent to
\[
\min_{y\in\R^{m_{f}+Nm_{g}}}\left\Vert y\right\Vert _{0}\quad\text{subject to \ensuremath{\bigl\Vert Ay-h\bigr\Vert_{2}\le\sqrt{N}\epsilon,}}
\]
for which the OMP can be applied directly, as discussed above. As before, OMP is not guaranteed to converge to a minimizer of this functional. In other words, even though Theorem~\ref{thm:main} gives (almost) exact recovery via the minimization of \eqref{eq:thm-min}, OMP may not provide a faithful reconstruction. However, the numerical simulations (Section~\ref{sec:Numerical-implementation}) suggest that in practice a correct reconstruction is always achieved via OMP. It would be interesting to consider this issue from the theoretical point of view, but this goes beyond the scope of this work.

Let us also mention other alternatives for the minimization of these problems, such as Basis Pursuit \cite{2001-basispursuit}, Block-Coordinate-Relaxation \cite{1998-bruce-etal}, Iterative Hard Thresholding \cite{2008-blumensath-davies} and Stagewise OMP \cite{2012-donoho-etal}. In particular, in Basis Pursuit the $\ell^0$-penalization term is substituted by an $\ell^1$ term, which makes the functional convex: the minimization can be easily achieved with standard tools. This would complicate the corresponding reconstruction result (Theorem~\ref{thm:main}) and consequently the assumptions on the original signals (Definition~\ref{def:complete}). Achieving this would be a natural follow-up of this work.

The Block-Coordinate-Relaxation method was adapted to the minimization of the functional related to the separation problem with $N=1$ measurement  \cite{2004-starck-MCA}. This method, sometimes referred as MCA, is a combination of Basis Pursuit (the $\ell^1$ norm is used instead of the $\ell^0$ norm) and of a two-step iterative procedure where the two components of the signal are minimized separately. Moreover, the minimization is expressed directly in terms of the signals $A_f y_f$ and $A_g y_g$, without the need of storing the full dictionary matrices, which may take a lot of memory space. It would be very interesting to see whether this method can be generalized to the multi-measurement case considered in this paper.

\subsection{\label{sub:Dictionaries-for-image}Dictionaries for image content}

Let us now discuss what choices may be done for the dictionaries $A_{f}$
and $A_{g}$ in the context we are interested in, namely hybrid imaging
inverse problems. As we shall see in Section~\ref{sec:Applications-to-hybrid},
in such problems the signal $\tilde{f}$ will typically represent
(the logarithm of) a constitutive parameter of the tissue under investigation,
like for instance the electric permittivity, electric conductivity
or the optical absorption. As such, the image $\tilde{f}$ can be
supposed to be piecewise smooth: the discontinuities are usually the
inclusions we would like to determine. On the other hand, the signals
$\tilde{g}_{i}$'s usually represent the injected fields, such as
the electric field or the light intensity, and are the solutions to
certain PDE. As such, they enjoy higher regularity properties, and
the images $\tilde{g}_{i}$'s can be supposed to be smooth.

These different features represent the foundation of the signal separation
method discussed in the previous section. In order to exploit this
diversity it is necessary to choose suitable dictionaries $A_{f}$
and $A_{g}$, with respect to which $\tilde{f}$ and the $\tilde{g}_{i}$'s
have sparse representations, respectively.

As far as $A_{f}$ is concerned, wavelets have been widely used to
sparsely represent piecewise smooth images \cite{2005-stark-elad-donoho}.
In particular, Haar wavelets will be used in this work: the choice is motivated by Proposition~\ref{prop:new_uncertainty}. It is worth noting that in recent years a large
number of new multi-dimensional transforms have been introduced in
order to capture the directional features of two-dimensional images
\cite{2010-rubinstein}. Among the most known, there are curvelets
\cite{2004-curvelets}, ridgelets \cite{1999-ridgelets} and
shearlets \cite{2005-labate-lim-kut-weiss,2011-kutyniok-lim,2012-kittipoom-kutyniok-lim}. The use of these transforms may
give better results, but a deep investigation of the best choice for
the dictionaries goes beyond the scope of this paper. Thus, we have
decided to make the most classical choice of wavelets, since, even though possibly not optimal, it is sufficient
to properly illustrate the disjoint sparsity signal separation method.

The situation is simpler for the choice of $A_{g}$. Indeed, a good
representation of smooth functions may be obtained by choosing low
frequency trigonometric polynomials. This is a simple consequence
of the correspondence between the smoothness of a function and the
decay of its Fourier transform. This represents our choice for $A_{g}$
in this paper.

These dictionaries purely represent a general guideline for the choices
of $A_{f}$ and $A_{g}$. Additional information on the particular
physical model may be used to select dictionaries more adapted to
the features of the images under consideration.

\section{\label{sec:Applications-to-hybrid}Applications to hybrid inverse
problems}

\subsection{\label{sub:Introduction-to-hybrid}Introduction}

We have seen in the introduction that a hybrid problem usually involves
two steps. First, internal functionals are measured inside the domain
and, second, from their knowledge the unknown coefficients of the
PDE have to be reconstructed. These internal data are linear or quadratic
functionals of the unknowns and of the solutions of the direct problem.
Let us mention some examples.
\begin{itemize}
\item In photoacoustic tomography \cite{2010-ammari-bossy, 2011-ammari-bossy, bal2010inverse,bal-kui-2011} the
internal data take the form
\[
H(x)=\Gamma(x)\mu(x)u(x),\qquad x\in\Omega,
\]
where $\Gamma$ is the Grüneisen parameter, $\mu$ is the optical
absorption and $u$ is the light intensity. The main unknown of the
problem is $\mu$. The photoacoustic image $H$ gives only qualitative
information on the medium, as $\mu$ is multiplied by $\Gamma$ and
$u$. The problem of quantitative photoacoustics is the reconstruction
of $\mu$ from $H$. Under the diffusion approximation, the light
intensity $u$ solves
\begin{equation}
-\div(D\nabla u)+\mu u=0\quad\text{in }\Omega,\label{eq:pat-diffusive}
\end{equation}
where $D$ is the diffusion parameter. This equation should be augmented
with suitable boundary conditions, such as of Dirichlet or Robin type.
\item In thermoacoustic tomography \cite{ammari2012quantitative} the internal
data take the form
\[
H(x)=\sigma(x)|u(x)|^{2},\qquad x\in\Omega,
\]
where $\sigma$ is the unknown conductivity and $u$ is the electric
field. The problem of quantitative thermoacoustics is the reconstruction
of $\sigma$ from $H$. In the scalar approximation, $u$ is the solution
of the Helmholtz equation
\[
\Delta u+(\omega^{2}+i\omega\sigma)u=0\quad\text{in }\Omega,
\]
where $\omega$ is the angular frequency. As before, this equation
should be augmented with suitable boundary conditions.
\item In microwave imaging by ultrasound deformation \cite{cap2011,alberti2013multiple}
the internal data take the form
\[
H(x)=\epsilon(x)u(x){}^{2},\qquad x\in\Omega,
\]
where $\epsilon$ is the unknown electric permittivity and $u$ is
the electric field. The problem is now to reconstruct $\epsilon$
from $H$. In the scalar approximation, $u$ is the solution of the
Helmholtz equation
\[
\Delta u+\omega^{2}\epsilon u=0\quad\text{in }\Omega,
\]
augmented with suitable boundary conditions.
\item In ultrasound modulated diffuse optical tomography 
\cite{2014a-ammari-seppecher, 2014b-ammari-seppecher, 2014c-ammari-seppecher} the internal data are $\div (u^2 \nabla \mu)$, 
where $u$ solves (\ref{eq:pat-diffusive}) and $\mu$ is the optical absorption.

\end{itemize}
In all the above examples, the measurement $H$ is the product of
the desired unknown and other quantities. Thus, the problem is extracting
the desired unknown from $H$. PDE techniques are usually used to
solve the problem, but, as discussed in the introduction, have several
drawbacks. 

All the above problems consist in the determination of two functions
from the knowledge of their product. Taking logarithms, in a multi-measurement
setting this is equivalent to finding $f(x)$ and $g_{i}(x)$ from
the knowledge of their sums
\[
h(x)=f(x)+g_{i}(x),\qquad x\in\Omega.
\]
The disjoint sparsity signal separation method can be applied since
$f$ and the $g_{i}$'s have different level of smoothness, and so
can be sparsely represented with respect to different dictionaries
(see $\S$~\ref{sub:Dictionaries-for-image}).

In particular, Theorem~\ref{thm:main} guarantees unique and stable
reconstruction of the the coefficients, provided that we can construct
many incoherent $u_{i}$ of the above problems (by changing the boundary
values), so that the corresponding ${g}_{i}$'s give a complete
set of measurements, according to Definition~\ref{def:complete}.
As we shall see below in the numerical simulations, the conclusion of Theorem~\ref{thm:main} is verified in several cases, choosing different boundary values. Unfortunately, it is not possible to check the conditions of  Definition~\ref{def:complete} numerically, as this would require an infinite number of computations. These conditions are used to heuristically inform the choice of incoherent illuminations. On the other hand, at the current state, we are unable to give a general strategy to build boundary values so that the corresponding solutions will satisfy the hypotheses of Definition~\ref{def:complete}: this represents a very interesting open problem regarding
the interplay of the PDE theory with the disjoint sparsity approach (see Section~\ref{sec:Conclusions}).

It is worth mentioning that similar assumptions of incoherence, usually
in terms of linear independence of gradients of solutions, are often
necessary for the PDE-based reconstruction methods (see, e.g., \cite{2009-cap-etal,bal2010inverse,cap2011,ammari2012quantitative,alberti2013multiple,bal2012inversediffusion,2014-albertigs,2013-albertigs,2014-scherzer-naetar}
and references therein).

As an example, in the rest of this section we shall apply the method
to the reconstruction in quantitative photoacoustic tomography. All
the other modalities mentioned above can be treated with minor modifications.

\subsection{\label{sub:Quantitative-photoacoustic-tomog-gamma=00003D1}Quantitative
photoacoustic tomography, the case $\Gamma=1$}

In photoacoustic tomography, the object under investigation is illuminated
with light radiation, whose absorption causes local heating of the
medium. The temperature rise creates an expansion of the tissue, thereby
producing an acoustic wave, that can be measured on the boundary of
the domain. The first step of this hybrid modality consists in the
recovery of the amount of absorbed radiation by inverting the wave
equation, from the knowledge of the boundary values. This problem
has attracted much attention in the last years: the reader is referred
to \cite{2011-kuchment-kunyansky} and to the references therein for
more details. In this paper, we shall suppose that the first step
has been performed, and that we have access to the amount of absorbed
radiation
\[
H(x)=\Gamma(x)\mu(x)u(x),\qquad x\in\Omega,
\]
where $\Gamma$ is the Grüneisen parameter, $\mu$ is the optical
absorption and $u$ is the light intensity. The problem of quantitative
photoacoustic tomography consists in the reconstruction of $\mu$
from the knowledge of $H$. Much research has been done on this over
the last years, see e.g.\ \cite{2006-cox-arridge,2009-razansky,2009-cox,2010-ammari-bossy,2011-cox-tarvainen-arridge,2011-shao-cox-zemp,2012-gao-osher-zhao,2013-ren-gao-zhao,2014-pulkkinen-cox-arridge}
and references therein. Sparse representations were first used in
\cite{2009-razansky}. Recently, one-step methods have been introduced to perform both steps at the same time \cite{2015-haltmeier-etal,2015-ding-etal}: it would certainly be interesting to see whether the signal separation approach can be employed in a one-step reconstruction method. 

Light propagation may be modeled by a radiative transport equation
or, when the scattering coefficient is large and the absorption is
small, by its diffusion approximation \eqref{eq:pat-diffusive}. We
consider here the simplest case when $\Gamma=1$, the general case
is discussed later in $\S$~\ref{sub:Quantitative-photo-acoustic-tomo}.
By using multiple measurements, $\mu$ can be recovered both in the
transport regime \cite{2010-bal-jollivet-jugnon} and in the diffusive
regime \cite{bal2010inverse}, under suitable assumptions on the solutions.

We now describe how to apply the disjoint sparsity approach to this
problem. If compared to the PDE approach, the separation of $\mu$ and $u$ does not require the use
of a PDE, and so no specific model of light propagation has to be assumed for the inverse
problem. Moreover, no a priori knowledge of relevant boundary conditions (which may be unrealistic) is required, in contrast to the PDE approach. These aspects should make this approach more feasible. For simplicity, we shall construct the solutions $u_{i}$
via \eqref{eq:pat-diffusive} with $D=1$ and Dirichlet boundary values,
namely
\begin{equation}
\left\{ \begin{array}{l}
-\Delta u_{i}+\mu u_{i}=0\quad\text{in }\Omega,\\
u_{i}=\phi_{i}\quad\text{on }\bo.
\end{array}\right.\label{eq:qpat-gamma=00003D1-pde}
\end{equation}
However, this equation will \underline{not} be used for the inversion.

The joint sparsity method will be applied as follows. Let $\tilde{\mu}$
denote the true absorption. After constructing $N$ solutions $\tilde{u}_{1},\dots,\tilde{u}_{N}$
to the above equation, the quantities 
\[
H_{i}(x)=\tilde{\mu}(x)\tilde{u}_{i}(x),\qquad x\in\Omega,
\]
are measured. Taking logarithms and adding white Gaussian noise $n_{i}$,
we obtain
\begin{equation}
h_{i}=\log\tilde{\mu}+\log\tilde{u}_{i}+n_{i},\qquad i=1,\dots N.\label{eq:qpat-ni}
\end{equation}
The reconstruction of $\tilde{\mu}$ from the knowledge of the $h_{i}$'s
exactly corresponds to the problem discussed in Section~\ref{sec:Disjoint-sparsity-for}.
The method based on Theorem~\ref{thm:main} and whose numerical implementation
is described in Section~\ref{sec:Numerical-implementation} will
be used for the reconstruction.

In all the examples, we shall consider the two-dimensional domain
$[0,1]^{2}$ with $128\times128$ pixels. As far as the choice for
the dictionaries is concerned, we let $A_{f}$ consist of the orthonormal
basis of Haar wavelets (as in $\S$~\ref{subsec:appendix2}) and let $A_{g}$ consist of 961 low frequency
trigonometric polynomials (as in \eqref{eq:realsinusoids}, including the constant matrix), periodic over $[0,1]^{2}$.

The choice for $A_{g}$ forces to choose periodic boundary conditions,
and so we set
\begin{equation}
\begin{aligned} & \phi_{1}(x)=1,\\
 & \phi_{2}(x)=1-\sin(2\pi x_{1})/4,\\
 & \phi_{3}(x)=1-\sin(2\pi x_{2})/4,\\
 & \phi_{4}(x)=1-\cos(2\pi x_{2})/4,\\
 & \phi_{5}(x)=1-\cos(2\pi x_{1})/4.
\end{aligned}
\label{eq:phii}
\end{equation}
(For physical constraints, all the quantities have to be positive.)
Non-periodic choices for the boundary conditions would be allowed
with no added difficulties for the reconstruction, provided that the
dictionary $A_{g}$ is properly chosen. The above boundary values
are expected to generate incoherent solutions to \eqref{eq:qpat-gamma=00003D1-pde}
in such a way to satisfy the conditions of complete sets as in  Definition~\ref{def:complete}.
In this case, the assumptions of Theorem~\ref{thm:main} would be verified
and the disjoint sparsity separation method  would be guaranteed to provide
the right reconstruction, even in presence of noise. However, this is not guaranteed a priori: as mentioned above, the conditions of  Definition~\ref{def:complete} are very hard to check.

\subsubsection{\label{subsub:Example-1--}Example 1 - convex inclusions}

We consider three convex constant inclusions in a homogeneous background,
as shown in Figure~\ref{fig:A-nonoise-mureal}. We choose to stop
the iteration procedure of OMP after 1500 iterations, which gives
satisfactory results. More accurate stopping criteria may be considered
\cite{2009-razansky}, but this is not among the aims of this work.

In a first experiment we consider one noise-free measurement, namely
we set $N=1$ and $n_{1}=0$ in \eqref{eq:qpat-ni}. The results are
shown in Figure~\ref{fig:A-all}. The solution to \eqref{eq:qpat-gamma=00003D1-pde}
with boundary value $\phi_{1}=1$ is shown in Figure~\ref{fig:A-u},
and the corresponding measurement $H_{1}=\tilde{\mu}\tilde{u}_{1}$
in Figure~\ref{fig:A-H1}. As it is evident from the images, the
inclusions are still clearly recognizable in $H_{1}$, but the corresponding
values of the absorption are corrupted by the multiplication by $\tilde{u}_{1}$.
The sparse separation approach is thus applied as discussed above,
and the reconstructed $\mu$ is shown in Figure~\ref{fig:A-mu}.
The relative reconstruction error is
\[
\frac{\left\Vert \log\mu-\log\tilde{\mu}\right\Vert _{2}}{\left\Vert \log\tilde{\mu}\right\Vert _{2}}\approx1.5\%.
\]
This shows that, in absence of noise, the reconstruction is almost
exact, even with only one measurement. This is in total agreement with the theoretical result discussed above. In the absence of noise ($\eta=0$), it is possible to satisfy the conditions in Definition~\ref{def:complete} in view of Proposition~\ref{prop:new_uncertainty} and Example~\ref{exa:2}. Indeed, the sparse approximations of $\log\tilde{\mu}$ and of $\log\tilde{u}_1$ with 500 wavelets and 100 sinusoids, respectively, give small approximation errors, namely $\rho_f\approx 0.54$ and $\rho_g\approx 0.05$. By \eqref{eq:1160}, CS2 is satisfied. Thus, the assumptions of Theorem~\ref{thm:main} are verified with $\epsilon\approx 0.59$, and the resulting recontruction is very good ($\norm{\log\mu-\log\tilde{\mu}}_2\approx 0.29$). As already mentioned in Example~\ref{exa:2}, this argument is not fully rigorous, since Proposition~\ref{prop:new_uncertainty} does not allow the constant image to be in the dictionary $A_g$.
\begin{figure}
\subfloat[The true absorption $\tilde{\mu}$]{\includegraphics[width=0.38\columnwidth]{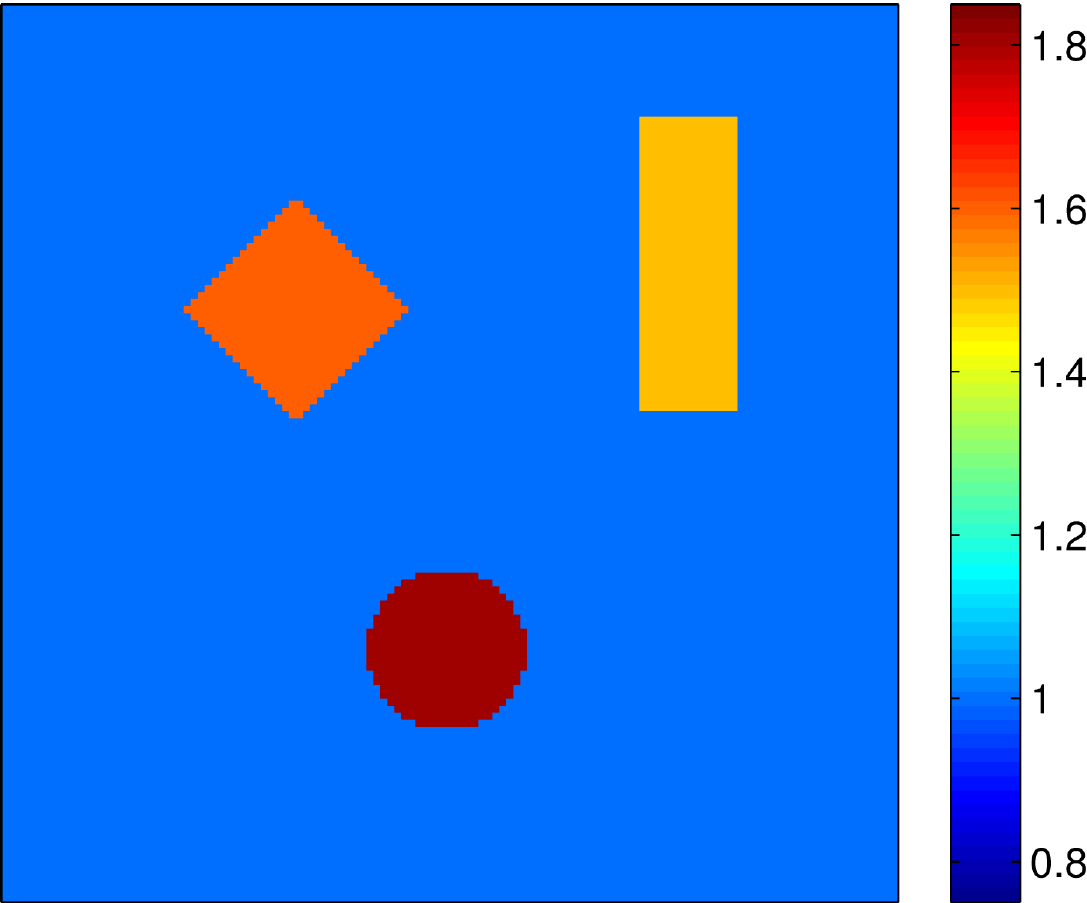}\label{fig:A-nonoise-mureal}}\quad{}\quad{}
\subfloat[The true intensity $\tilde{u}_{1}$]{\includegraphics[clip,width=0.38\columnwidth]{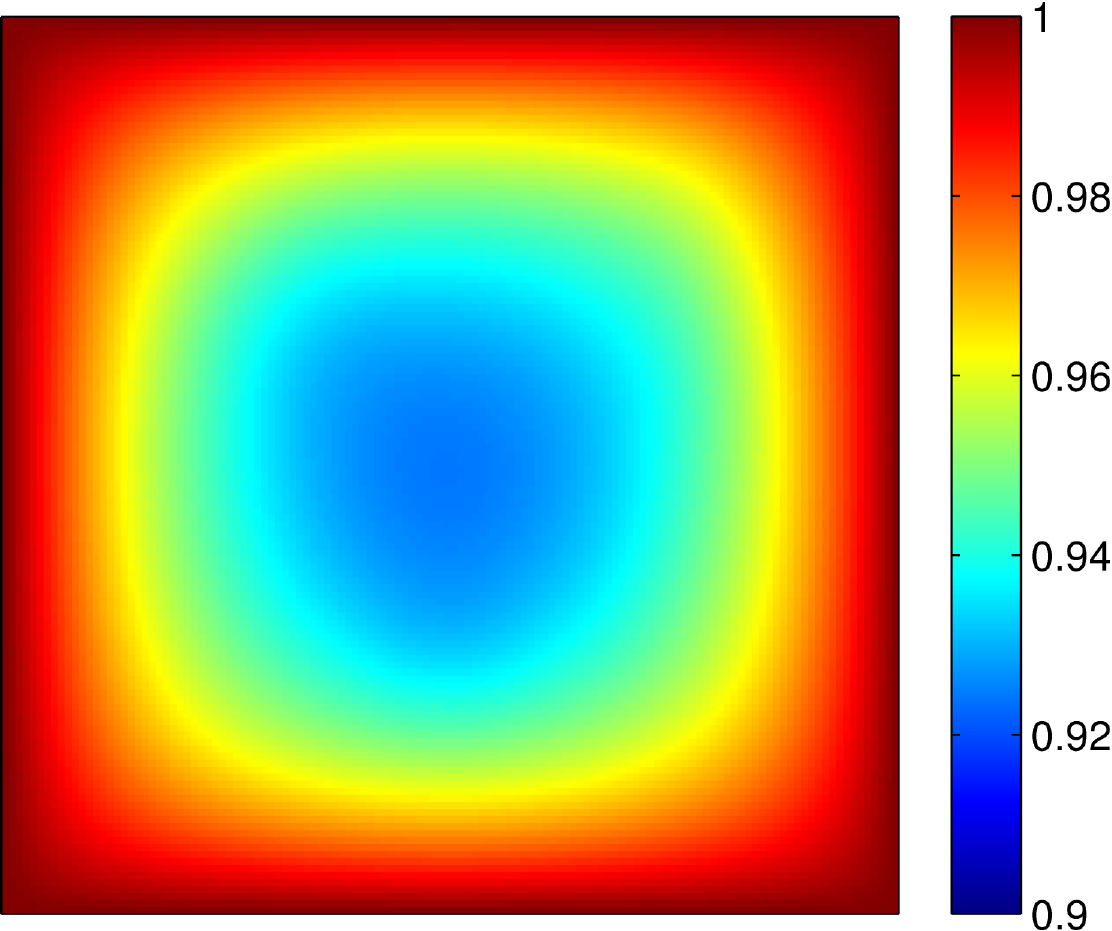}\label{fig:A-u}}

\subfloat[The datum $H_{1}$]{\includegraphics[width=0.38\columnwidth]{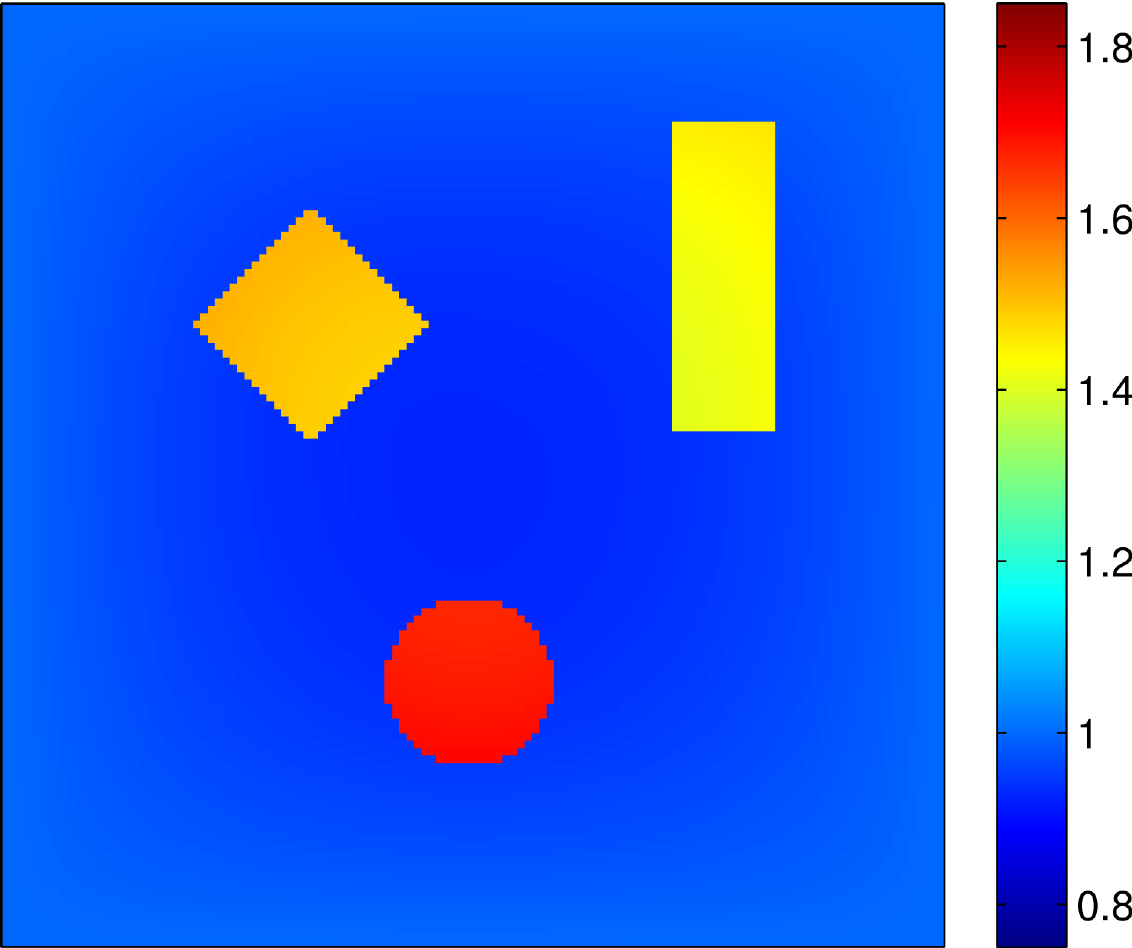}\label{fig:A-H1}}\quad{}\quad{}\subfloat[The reconstructed $\mu$ ]{\includegraphics[clip,width=0.38\columnwidth]{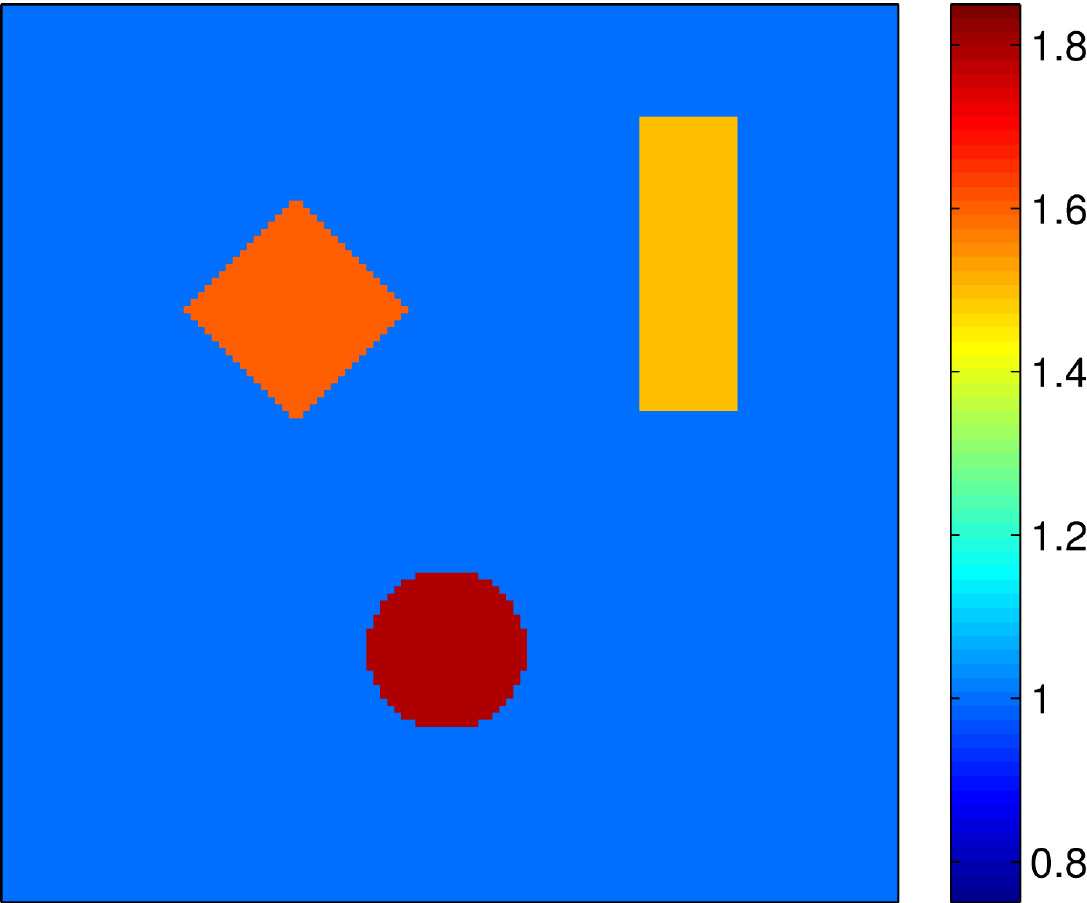}\label{fig:A-mu}

}\caption{Example 1. Noise-free case with one measurement.}
\label{fig:A-all}
\end{figure}

In a second experiment (Figure~\ref{fig:A-noise-all}) we add white
Gaussian noise $n_{i}$ in \eqref{eq:qpat-ni}. The noise level is
so that
\[
\frac{\left\Vert n_{i}\right\Vert _{2}}{\left\Vert \log(\tilde{\mu}\tilde{u}_{i})\right\Vert _{2}}\approx17.6\%.
\]
We tested the reconstruction procedure for $N=1,\dots,5$ and the
boundary values $\phi_{i}$ as in \eqref{eq:phii}. The data $H_{i}$,
$i=1,3,5$, are shown in Figures~\ref{fig:A-noise-H1}, \ref{fig:A-noise-H3}
and \ref{fig:A-noise-H5}, respectively. The reconstructed $\mu$'s
for $N=1$, $N=3$ and $N=5$ are shown in Figures~\ref{fig:A-noise-mu1},
\ref{fig:A-noise-mu3} and \ref{fig:A-noise-mu5}, respectively. The
reconstruction errors for $N=1,\dots,5$ are given  in Table~\ref{tab:A-noise}.

\begin{table}[h]
\caption{Example 1. The reconstruction errors $\left\Vert \log\mu-\log\tilde{\mu}\right\Vert _{2}/\left\Vert \log\tilde{\mu}\right\Vert _{2}$ for the noisy data.\label{tab:A-noise}}
\begin{tabular}{>{\centering}p{4cm}|>{\centering}p{1.2cm}>{\centering}p{1.2cm}>{\centering}p{1.2cm}>{\centering}p{1.2cm}>{\centering}p{1.2cm}}
$N$ & 1 & 2 & 3 & 4 & 5\tabularnewline
\hline 
No regularization & 74.4\% & 32.2\% & 28.4\% & 15.8\% & 7.5\%\tabularnewline
\hline 
TV regularization & 23.2\% & 6.0\% & 22.1\% & 4.8\% & 4.6\%\tabularnewline
\end{tabular}
\end{table}

\begin{figure}
\subfloat[The datum $H_{1}$.]{\includegraphics[clip,width=0.38\columnwidth]{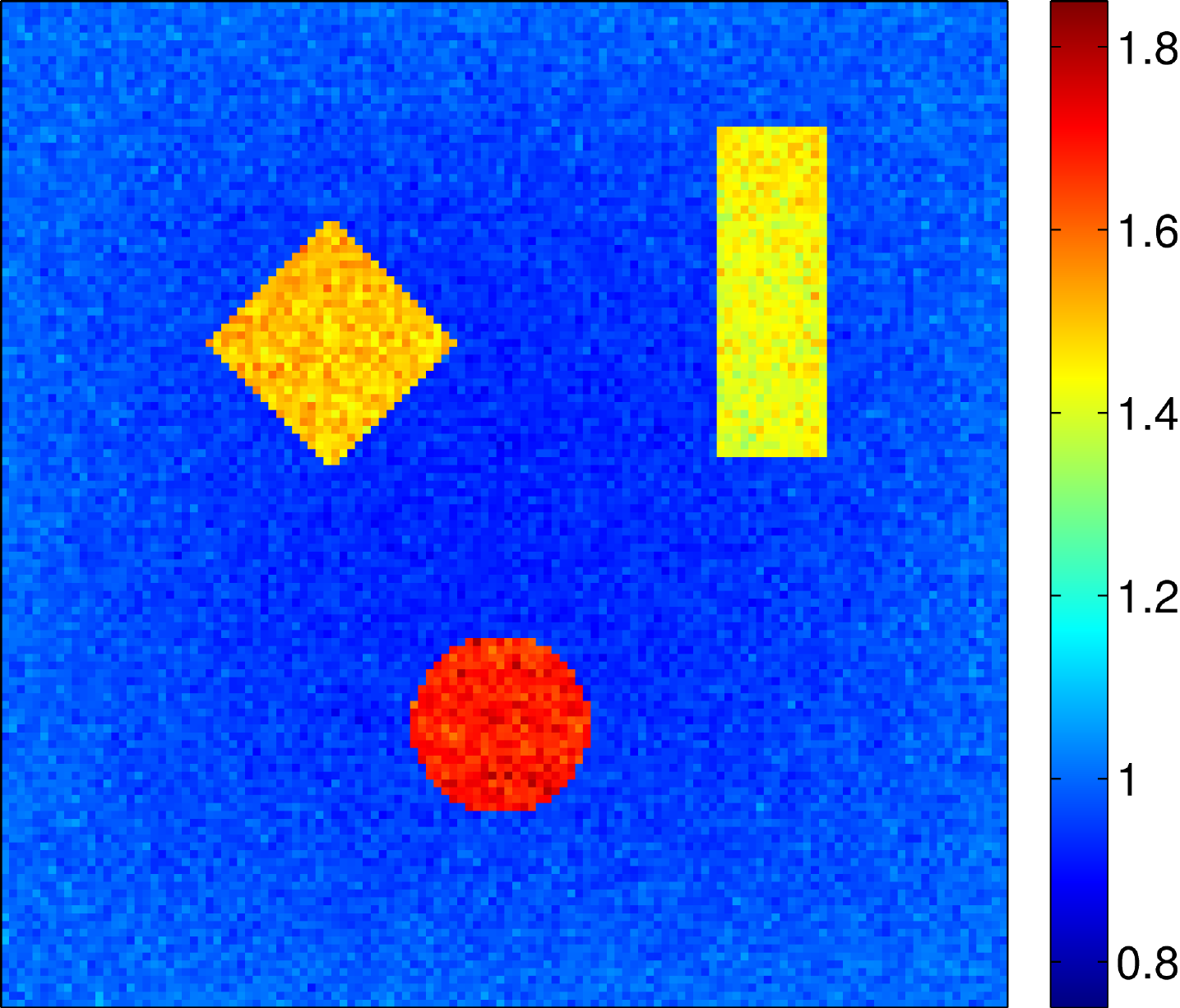}\label{fig:A-noise-H1}}\quad{}\quad{}
\subfloat[$\mu$, $N=1$.]{\includegraphics[clip,width=0.38\columnwidth]{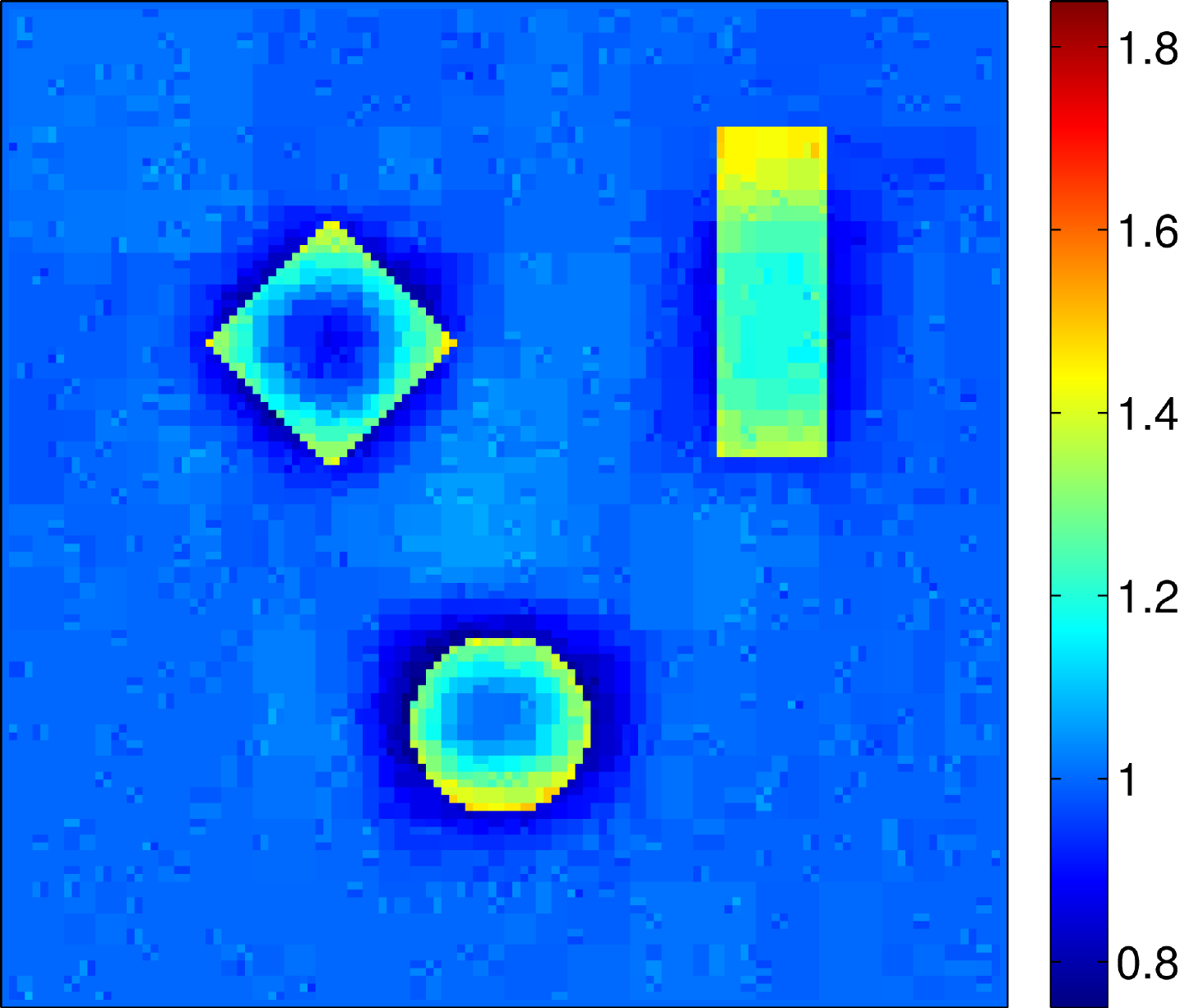}\label{fig:A-noise-mu1}}

\subfloat[The datum $H_{3}$.]{\includegraphics[width=0.38\columnwidth]{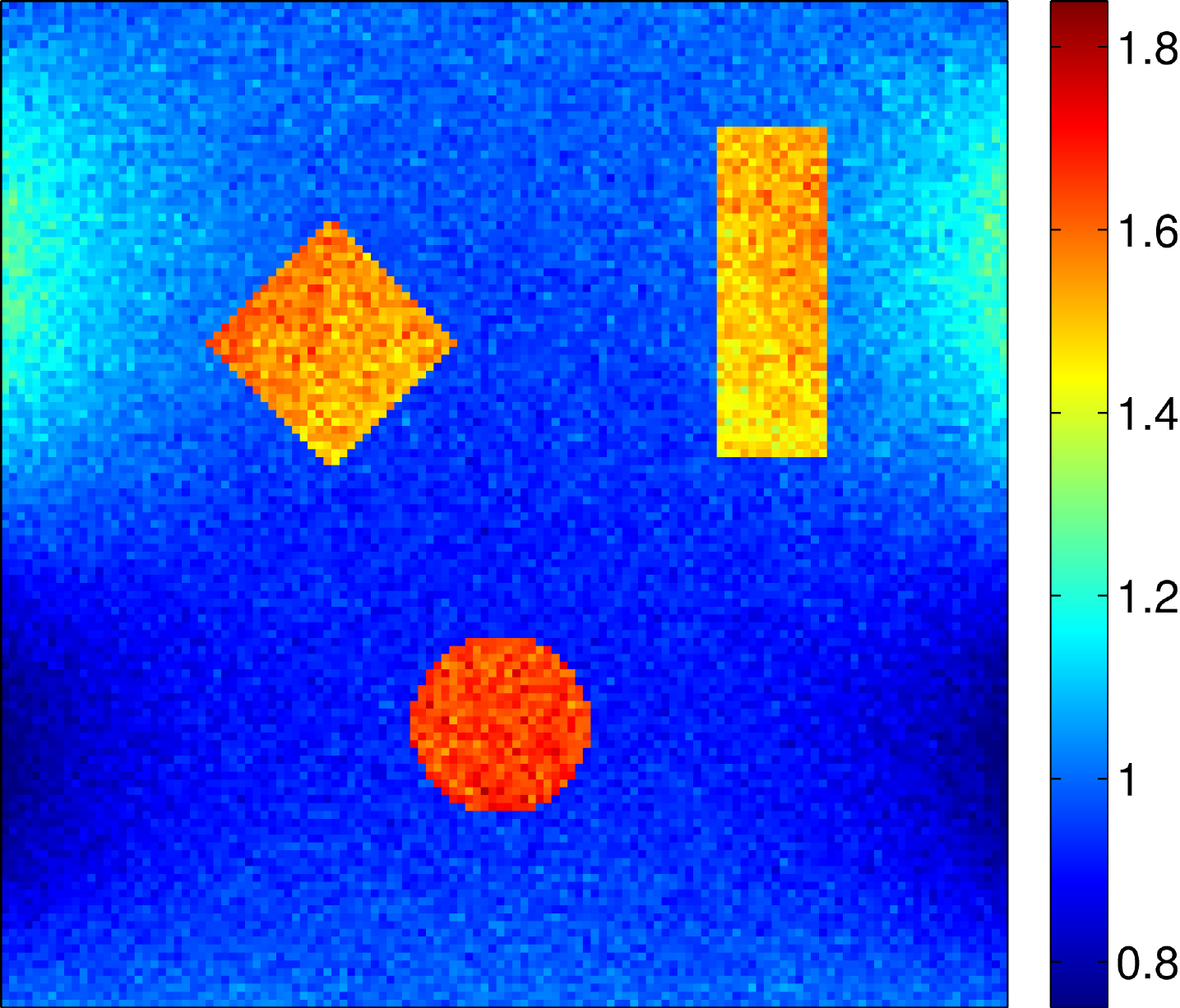}\label{fig:A-noise-H3}}\quad{}\quad{}\subfloat[$\mu$, $N=3$.]{\includegraphics[clip,width=0.38\columnwidth]{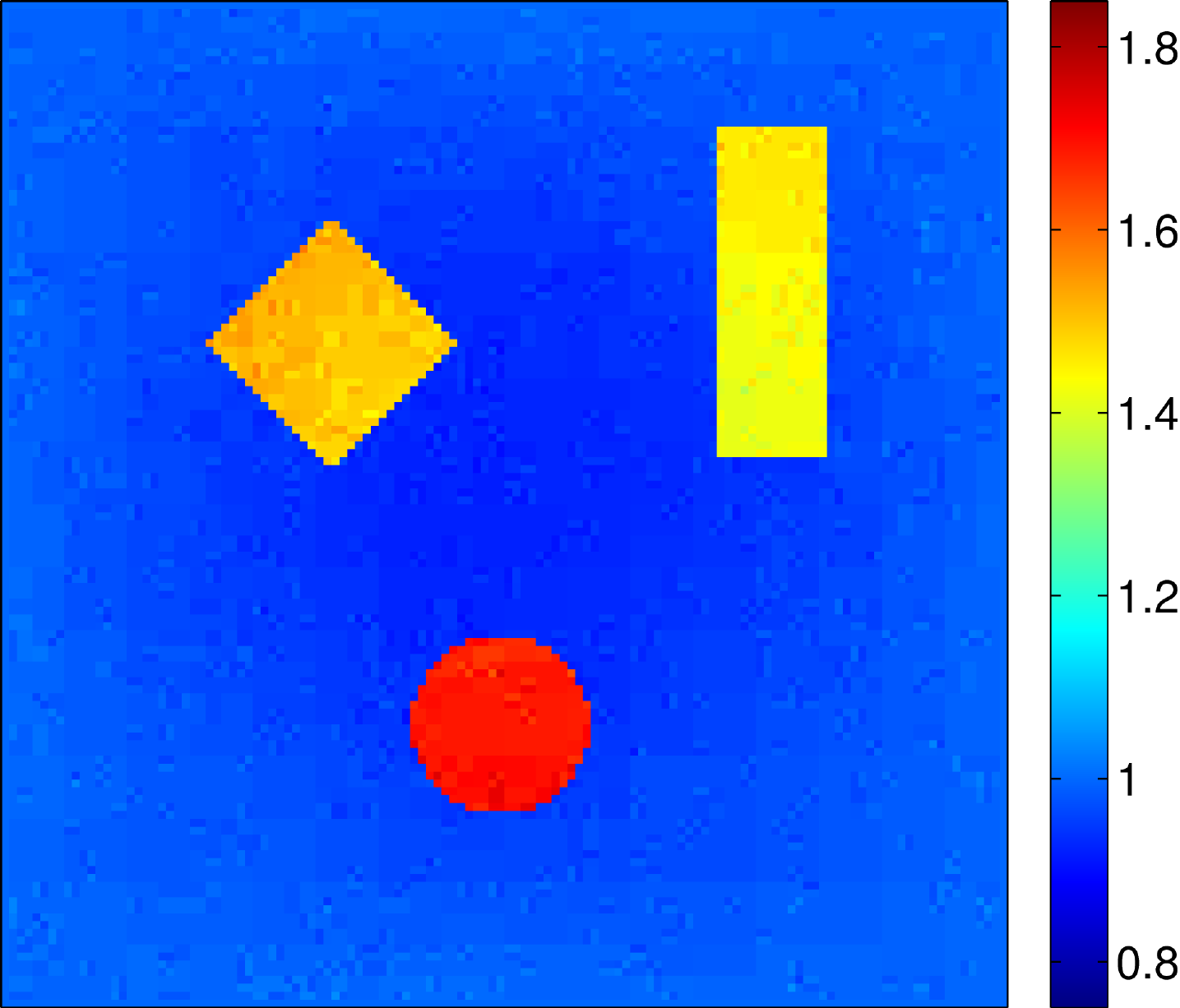}\label{fig:A-noise-mu3}}

\subfloat[The datum $H_{5}$.]{\includegraphics[width=0.38\columnwidth]{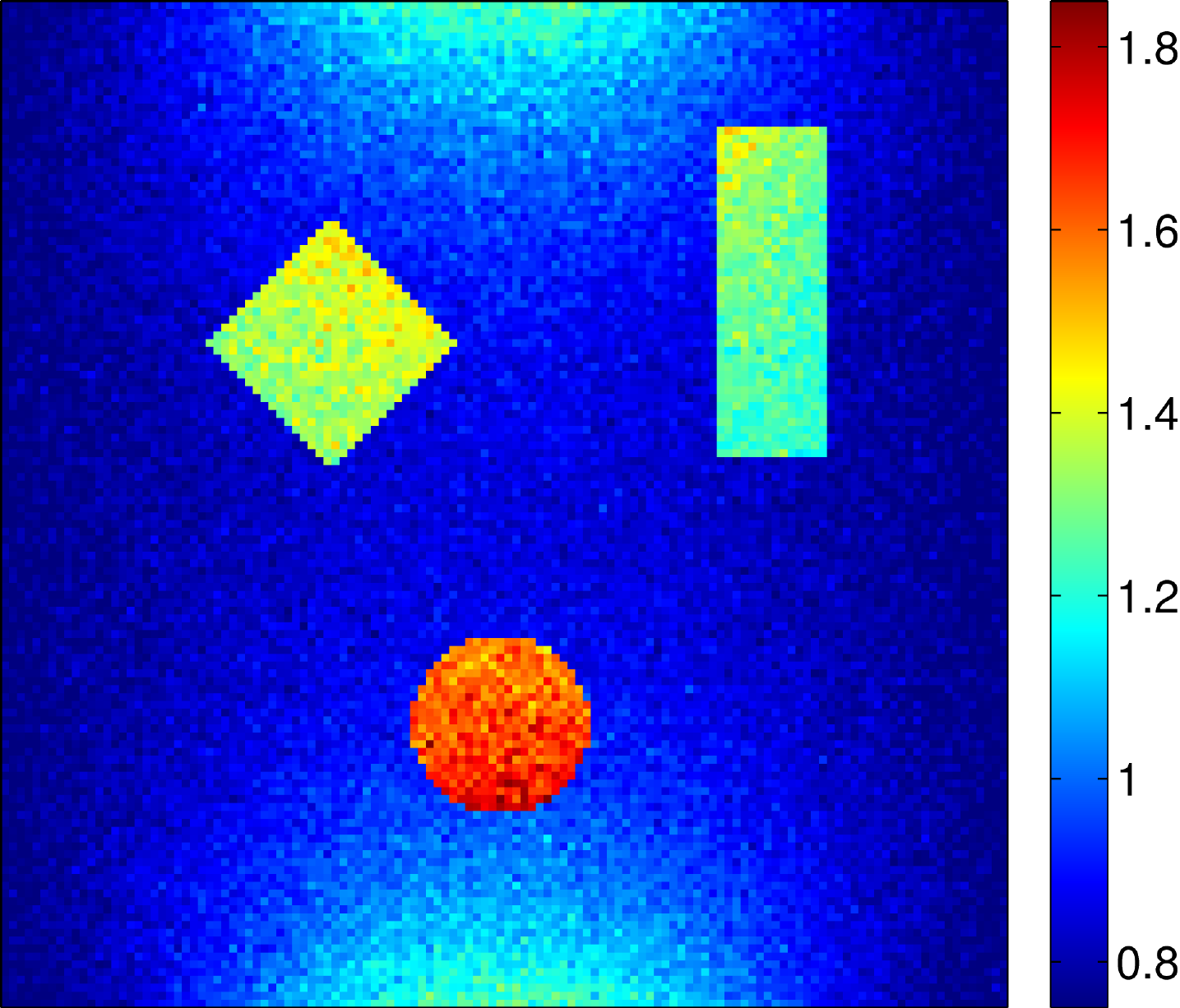}\label{fig:A-noise-H5}}\quad{}\quad{}\subfloat[$\mu$, $N=5$.]{\includegraphics[clip,width=0.38\columnwidth]{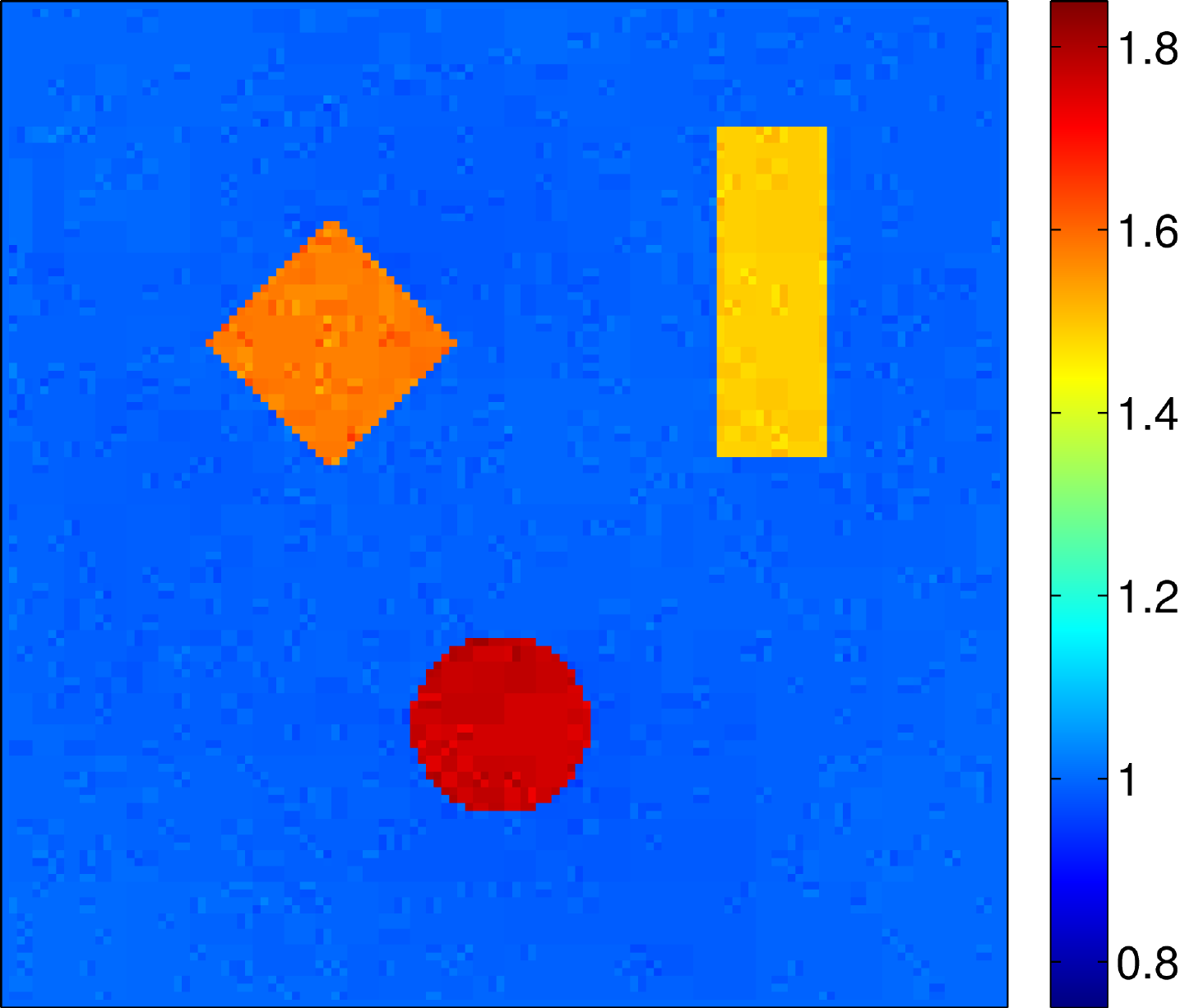}\label{fig:A-noise-mu5}}

\subfloat[$H_{2}$ with TV]{\includegraphics[width=0.38\columnwidth]{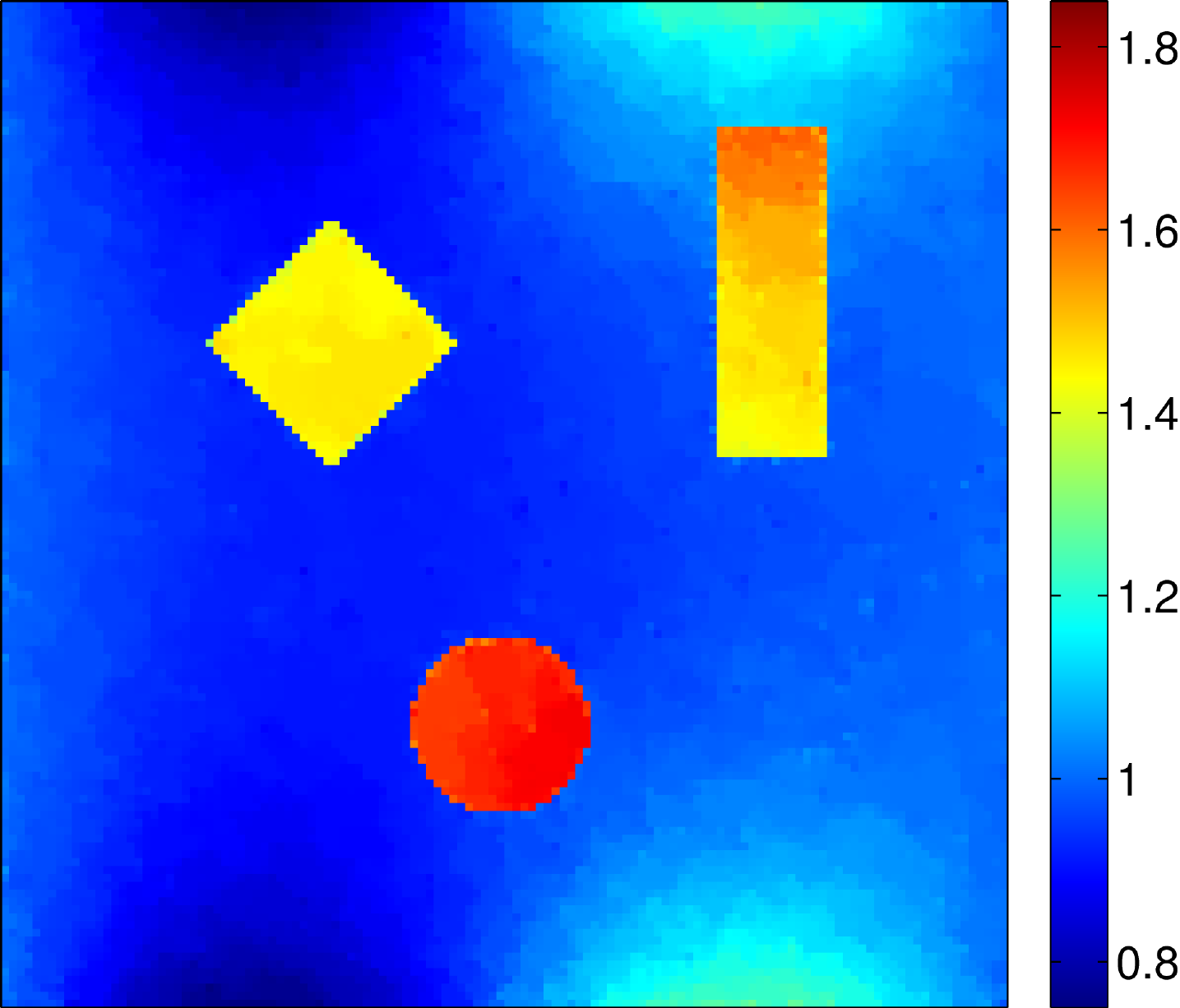}\label{fig:A-noise-H2-TV}}\quad{}\quad{}\subfloat[$\mu$ with TV, $N=2$.]{\includegraphics[clip,width=0.38\columnwidth]{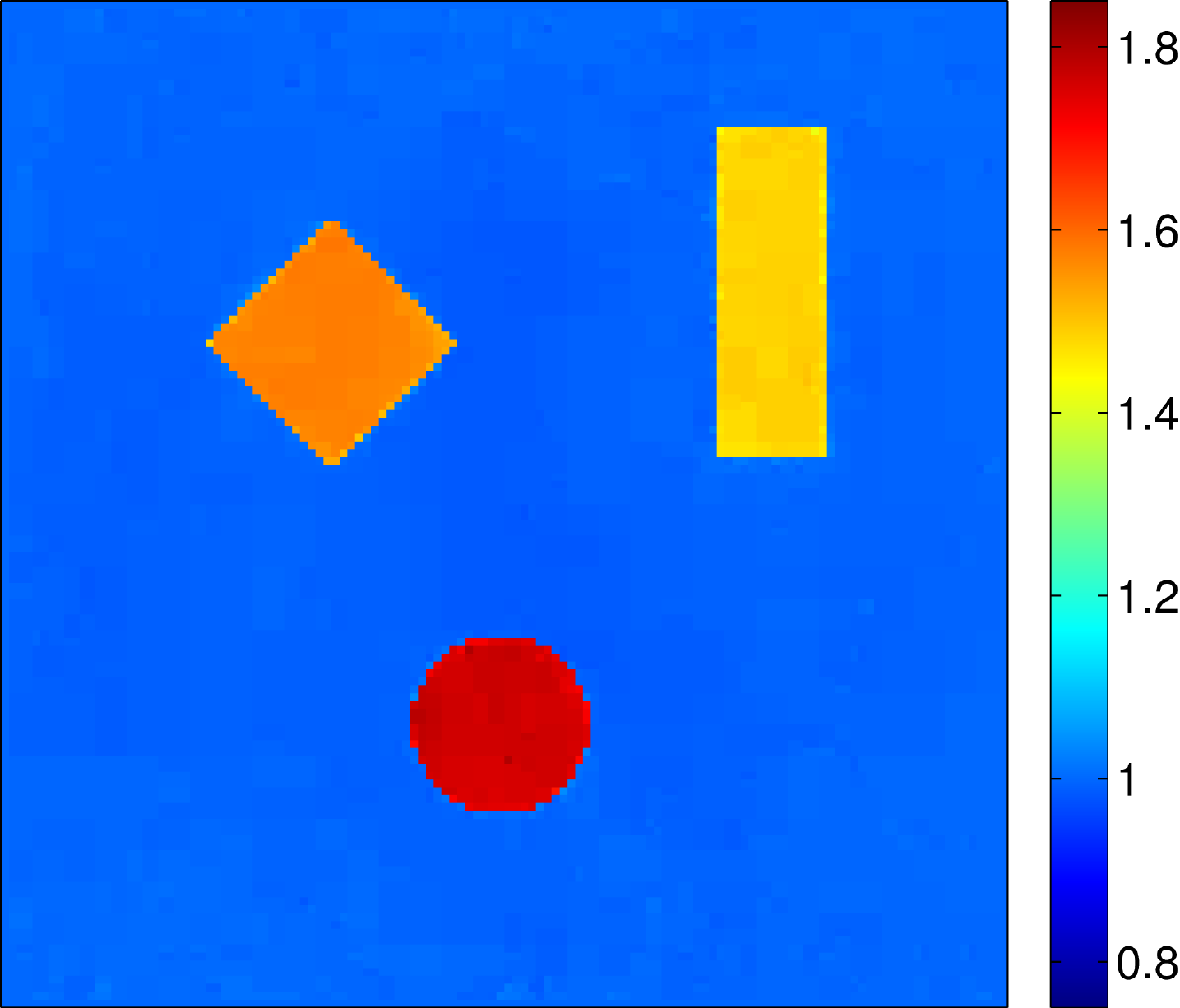}\label{fig:A-noise-mu2-TV}}\caption{Example 1. Noisy case with multiple measurements.}
\label{fig:A-noise-all}
\end{figure}

It is evident that the larger $N$ is, the more accurate the reconstruction
becomes. More precisely, when $N=1$ the reconstructed values of the
absorption in the inclusions are completely wrong. This is due to
the fact that the inclusions are roughly approximated by smooth atoms
in $A_{g}$ and then corrected with fewer atoms in $A_{f}$, and so
the sparsest approximation does not separate the two components as
desired. Choosing the same sparse approximations discussed above (so that CS2 is satisfied with $N=1$) gives a big value for $\epsilon$ in the noisy case, and so the reconstruction is not accurate. However, the problem is solved when more measurements are
added: CS2 is easily satisfied with lower values of $\rho_f$ and $\rho_g$
when $N$ becomes bigger (see Remark~\ref{rem:noise}).

The reconstruction with $N=5$ is very satisfactory if measurement
and reconstruction errors are compared. Indeed, the noise from the
data has almost disappeared in the reconstruction, without a separate
regularization. This is due to the implicit regularizing effect that
sparse representations provide.

We have also investigate the effect of an a priori total variation (TV) regularization \cite{1992-rudin-osher-fatemi} of the
measurements $h_i$ on the reconstruction, using a Matlab implementation based on the algorithm developed in \cite{2004-chambolle}. The regularization parameter was chosen a posteriori to achieve the best results, but in principle it can be learned from a training set \cite{2015-calatroni}. The corresponding reconstruction errors with different values of $N$ are shown in Table~\ref{tab:A-noise}: the improvement is significant only for a low number of measurements. For comparison, the regularized value of $H_2$ is shown in Figure~\ref{fig:A-noise-H2-TV} (where the usual staircase effect can be observed) and the reconstruction with $N=2$ measurements is shown in Figure~\ref{fig:A-noise-mu2-TV}.

\subsubsection{Example 2 - The Shepp\textendash{}Logan phantom}

Here, we let $\tilde{\mu}$ be the well-known Shepp-Logan phantom
(shown in Figure~\ref{fig:B-mureal}). We choose to stop the iterative
procedure of OMP after 2000 iterations. As above, we consider the
boundary conditions $\phi_{i}$ as in \eqref{eq:phii} and the corresponding
solutions $\tilde{u}_{i}$ to \eqref{eq:qpat-gamma=00003D1-pde},
for $i=1,\dots,5$, and measure the quantities $h_{i}$ as in \eqref{eq:qpat-ni}.

\begin{table}[h]
\caption{Example 2. Reconstruction errors for the noise-free case.\label{tab:B-nonoise}}
\begin{tabular}{>{\centering}p{4cm}|>{\centering}p{1.2cm}>{\centering}p{1.2cm}>{\centering}p{1.2cm}>{\centering}p{1.2cm}>{\centering}p{1.2cm}}
$N$ & 1 & 2 & 3 & 4 & 5\tabularnewline
\hline 
$\left\Vert \log\mu-\log\tilde{\mu}\right\Vert _{2}/\left\Vert \log\tilde{\mu}\right\Vert _{2}$ & 68.6\% & 24.8\% & 18.6\% & 11.4\% & 5.4\%\tabularnewline
\end{tabular}
\end{table}

\begin{figure}
\subfloat[$\tilde{\mu}$]{\includegraphics[clip,width=0.31\columnwidth]{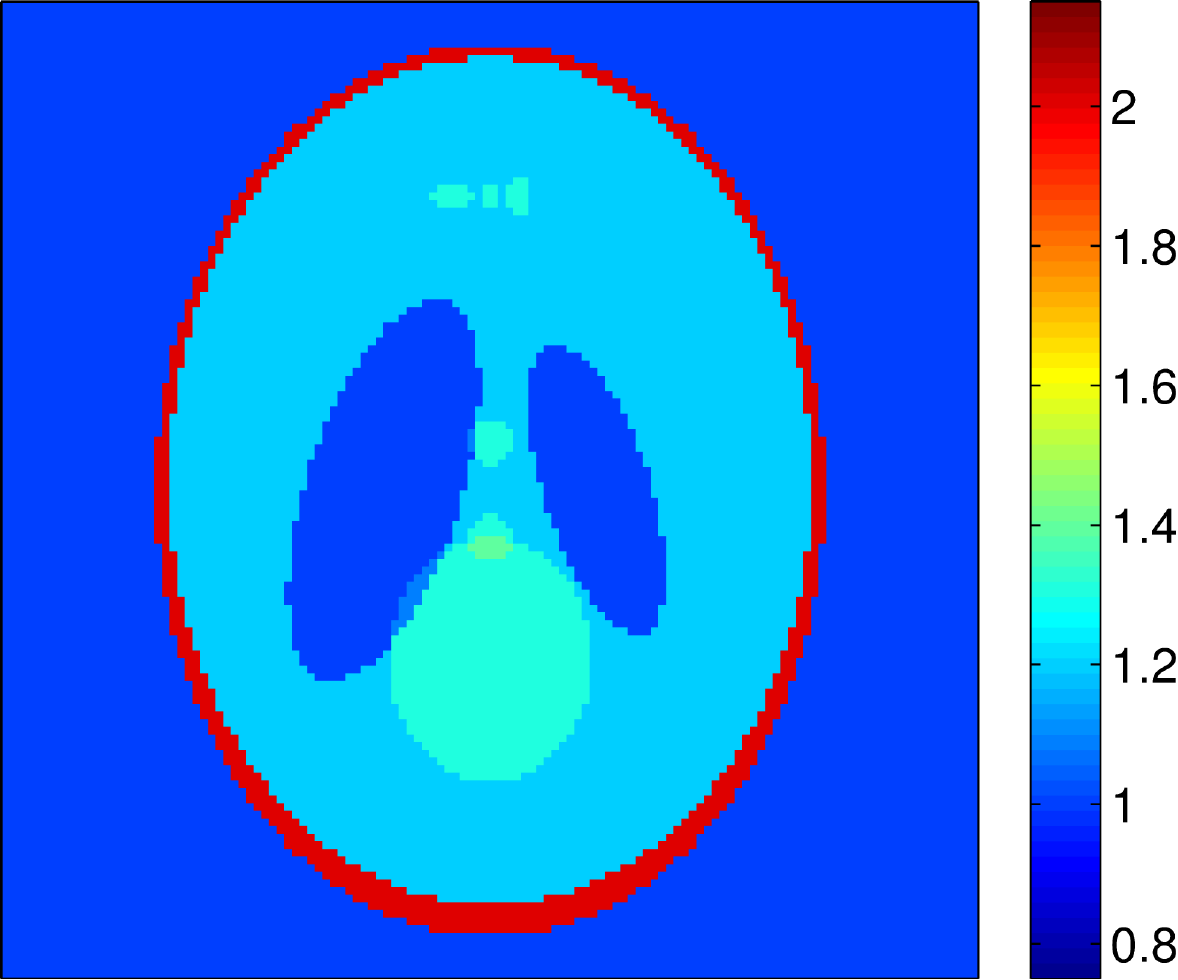}\label{fig:B-mureal}}\quad{}
\subfloat[$H_{2}$]{\includegraphics[clip,width=0.31\columnwidth]{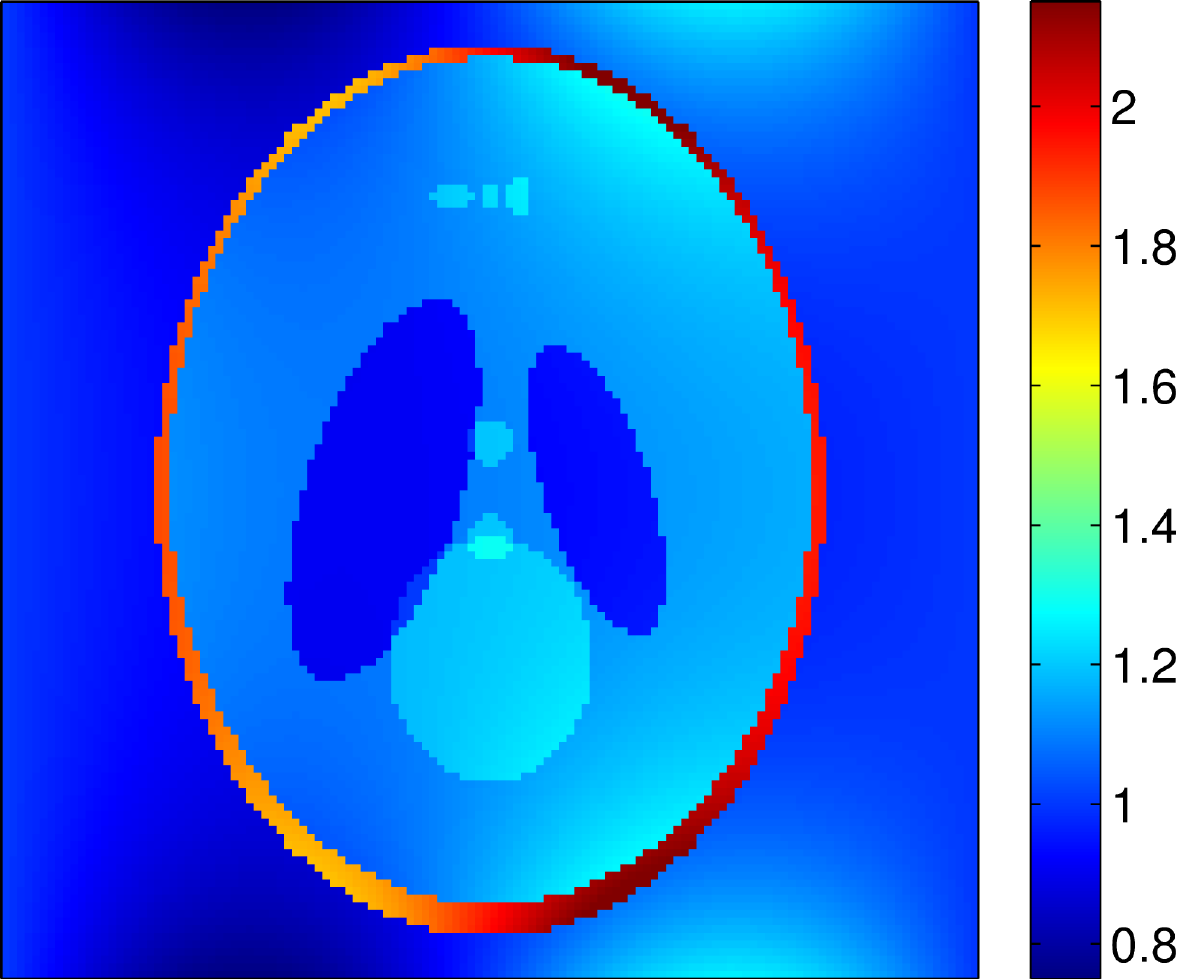}\label{fig:B-nonoise-H2}}\quad{}\subfloat[$H_{4}$]{\includegraphics[clip,width=0.31\columnwidth]{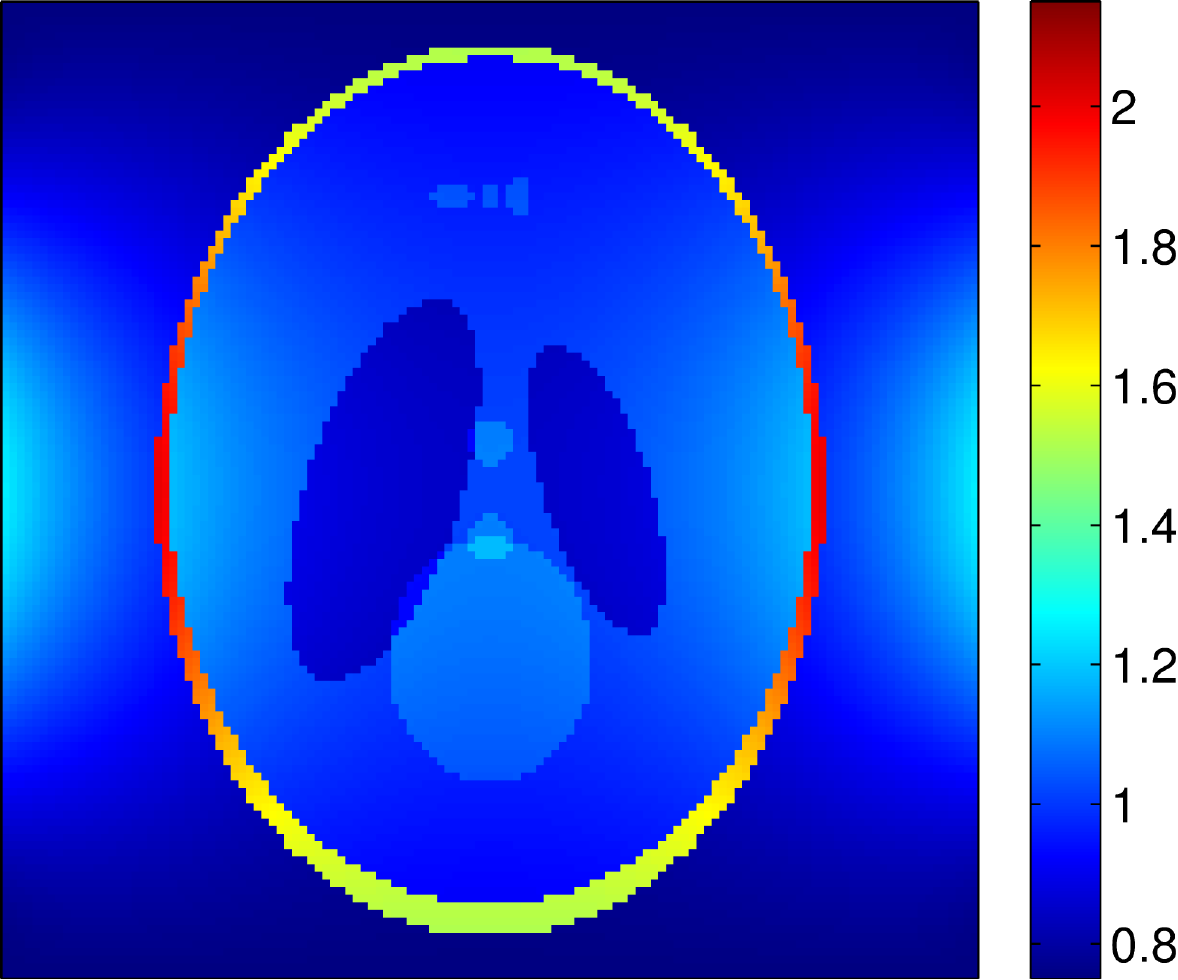}\label{fig:B-nonoise-H4}}

\subfloat[$\mu$, $N=1$]{\includegraphics[clip,width=0.31\columnwidth]{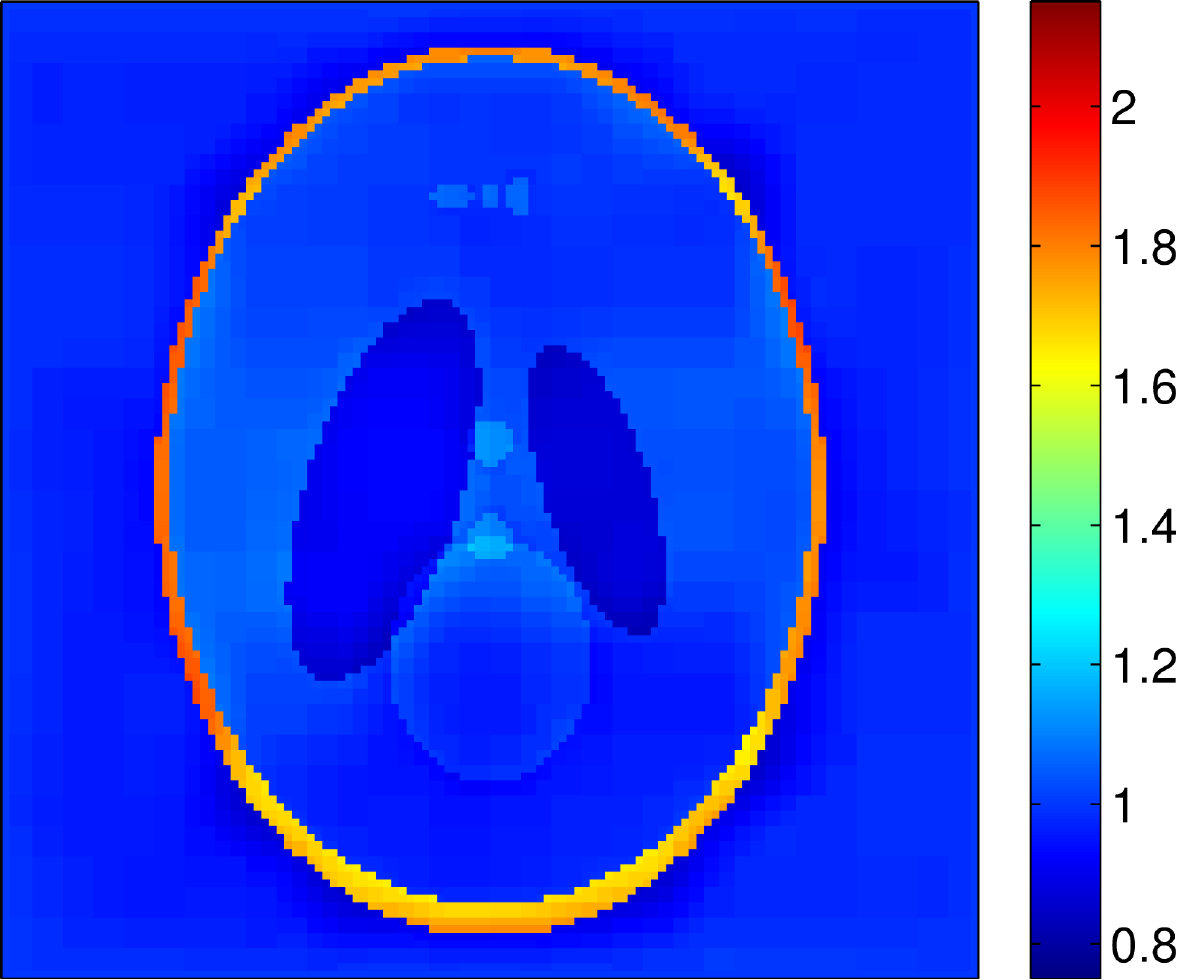}\label{fig:B-mureal-1}}\quad{}
\subfloat[$\mu$, $N=3$]{\includegraphics[clip,width=0.31\columnwidth]{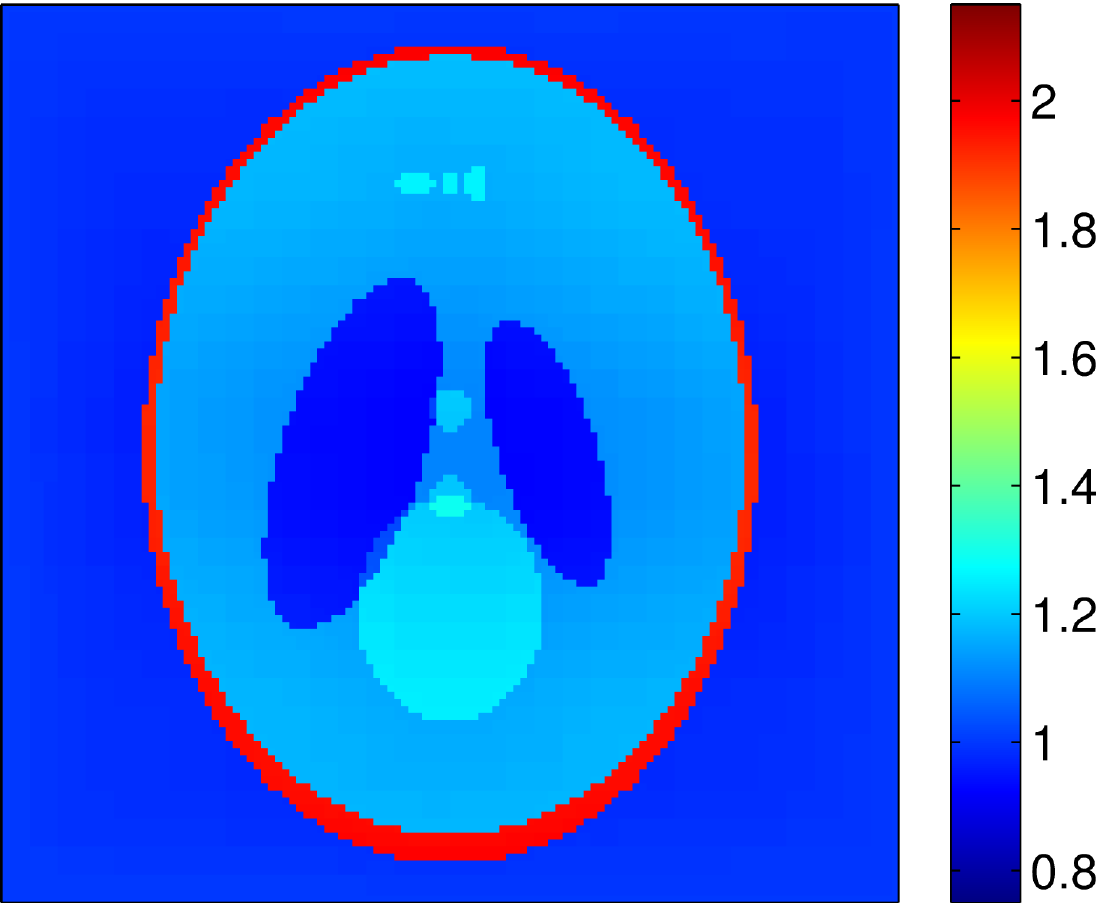}\label{fig:B-nonoise-H2-1}}\quad{}\subfloat[$\mu$, $N=5$]{\includegraphics[clip,width=0.31\columnwidth]{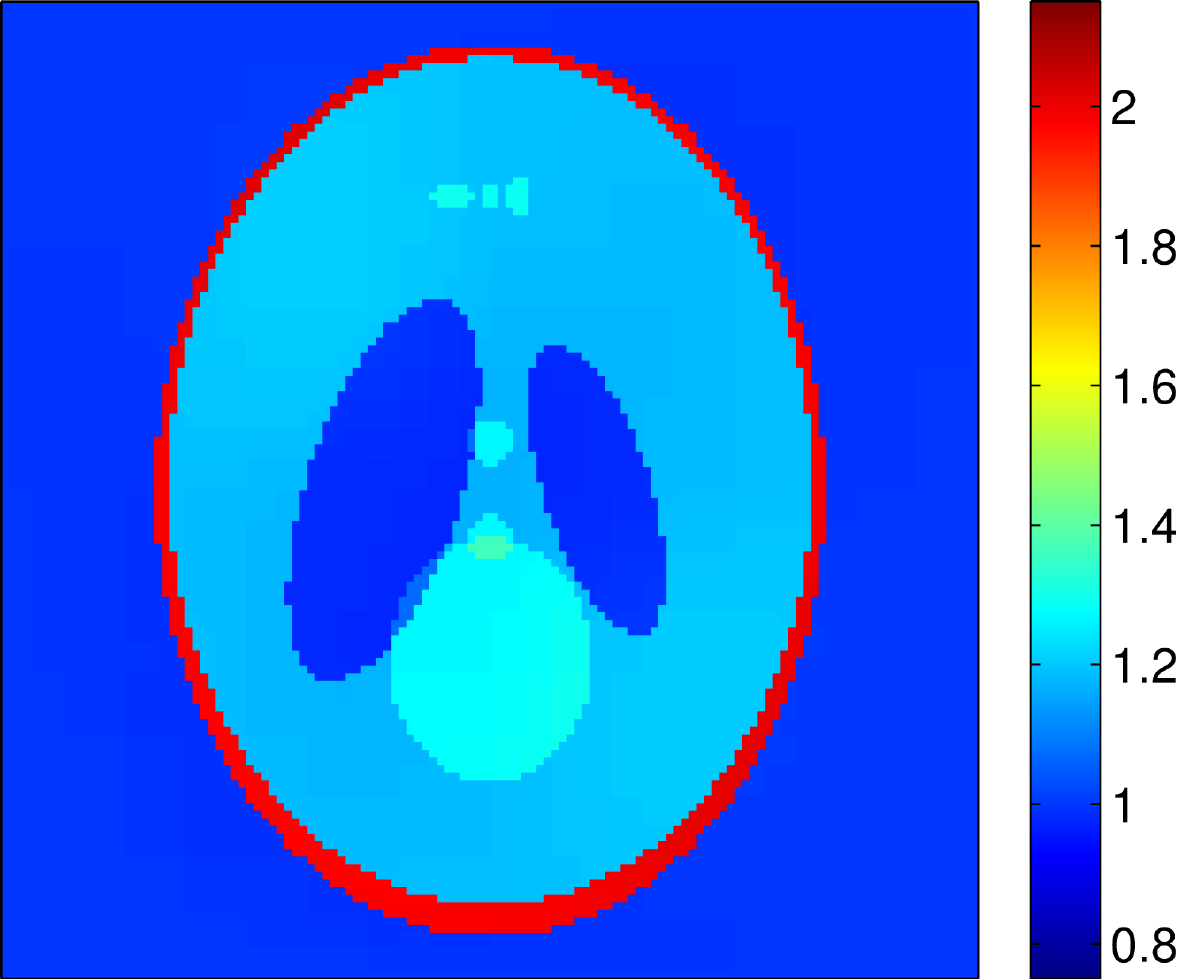}\label{fig:B-nonoise-H4-1}}

\caption{Example 2. Noise-free case with multiple measurements.}
\label{fig:B-nonoise}
\end{figure}

In a first experiment we consider the case without noise (Figure~\ref{fig:B-nonoise}).
The reconstruction errors for $N=1,\dots,5$ are shown in Table~\ref{tab:B-nonoise}.
We see that the reconstruction quality improves as $N$ increases,
as it is expected from the general theory discussed in Section~\ref{sec:Disjoint-sparsity-for}.
From a comparison with the previous case without noise, we notice
that more measurements are needed to have a satisfactory reconstruction.
This is due to the more complicated structure of the phantom, which
has a less sparse representation in terms of Haar wavelets than the
absorption considered in the previous case. Thus, a higher $N$ is
needed to satisfy the conditions in Definition~\ref{def:complete}.

In a second experiment we add white Gaussian noise $n_{i}$ to the
data in \eqref{eq:qpat-ni}. The noise level is such that 
\[
\frac{\left\Vert n_{i}\right\Vert _{2}}{\left\Vert \log(\tilde{\mu}\tilde{u}_{i})\right\Vert _{2}}\approx17.8\%.
\]
Motivated by the noisy-free case, we perform the reconstruction with
$N=5$ measurements. The reconstruction error is $\left\Vert \log\mu-\log\tilde{\mu}\right\Vert _{2}/\left\Vert \log\tilde{\mu}\right\Vert _{2}=17.0\%$,
that is comparable to the measurement error. The results are shown
in Figure~\ref{fig:B-noise15}. It is expected that adding more measurements
would improve the quality of the reconstruction.

\begin{figure}
\subfloat[$H_{3}$]{\includegraphics[clip,width=0.31\columnwidth]{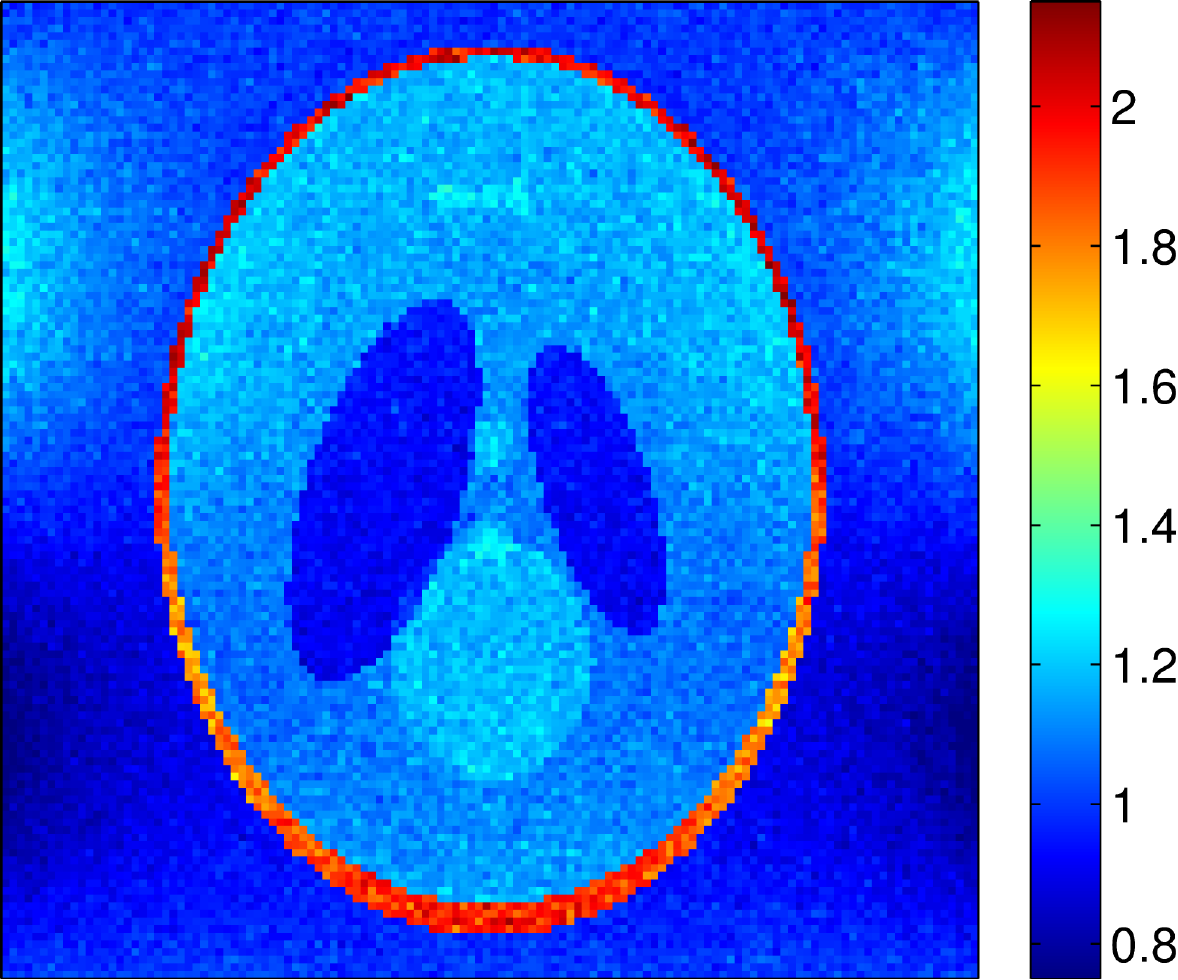}\label{fig:compare_u1-1}}\quad{}
\subfloat[$H_{5}$]{\includegraphics[clip,width=0.31\columnwidth]{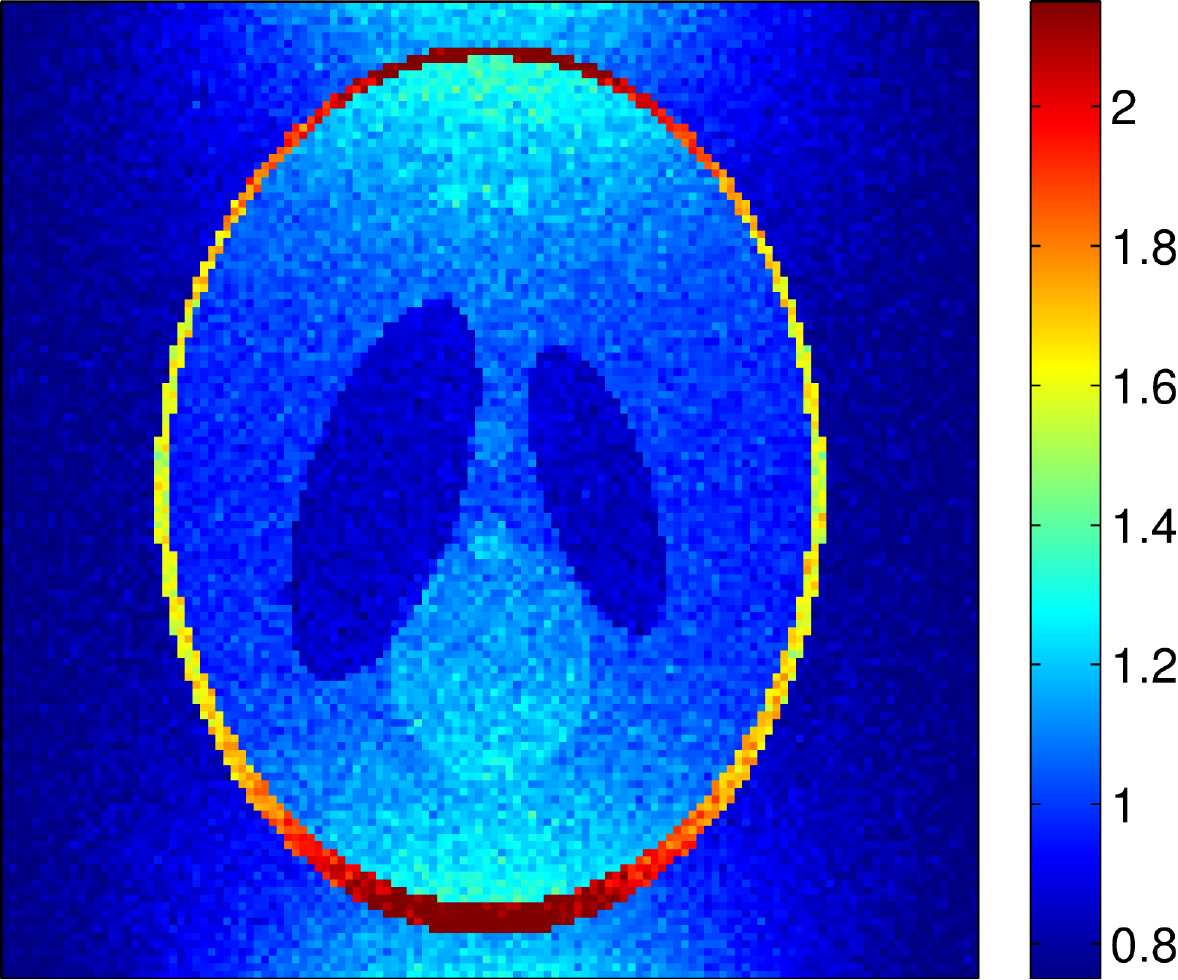}\label{fig:compare_u01-1}}\quad{}\subfloat[$\mu$, $N=5$]{\includegraphics[clip,width=0.31\columnwidth]{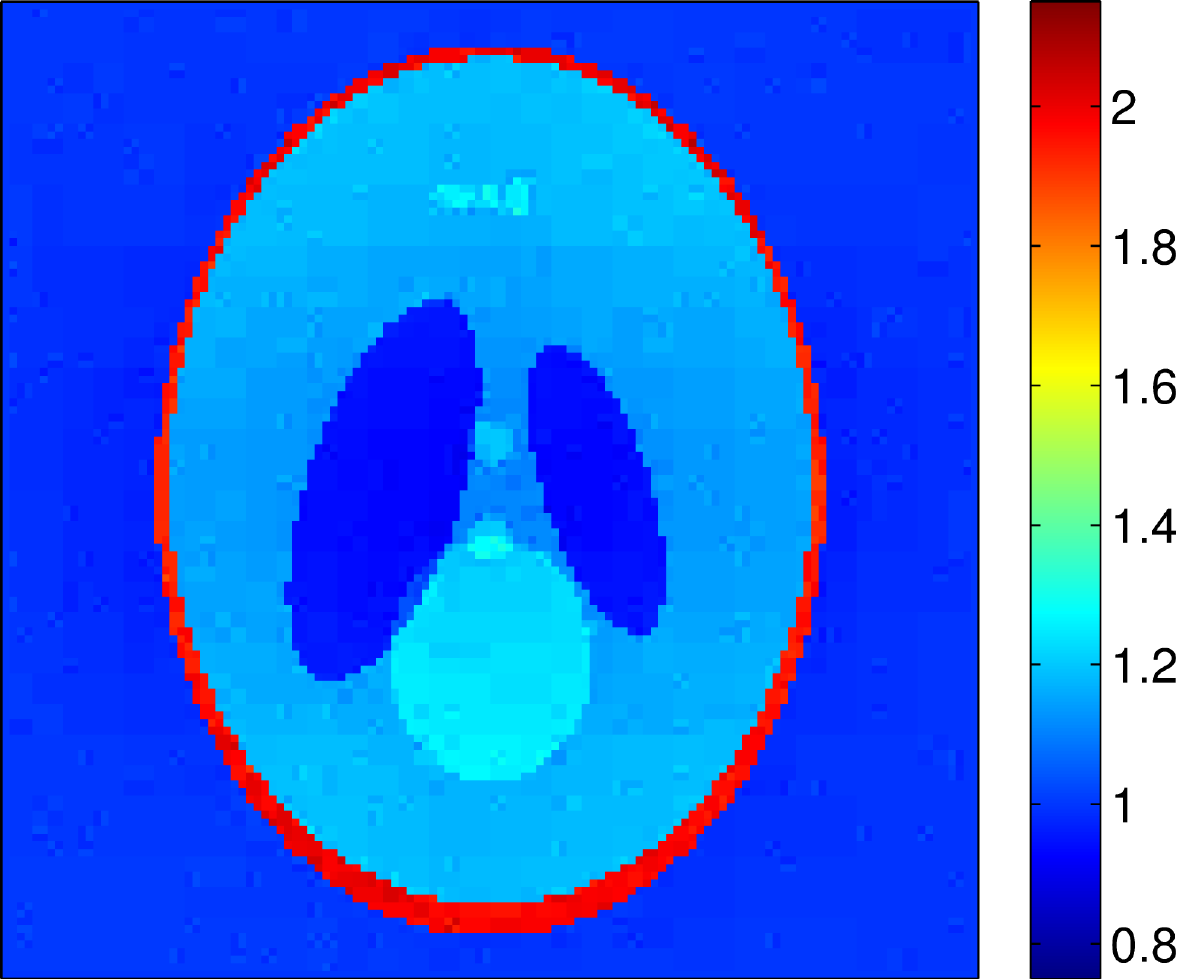}\label{fig:compare_u1/u01-1}}

\caption{Example 2. Noisy case with $N=5$ measurements.}
\label{fig:B-noise15}
\end{figure}

\subsection{\label{sub:Quantitative-photo-acoustic-tomo}Quantitative photo-acoustic
tomography in the diffusive regime with variable $\Gamma$}

\subsubsection{Introduction}

We consider here the problem of quantitative photoacoustic tomography,
as introduced above, without the further assumption $\Gamma=1$. Thus,
the unknown absorption $\mu$ has to be reconstructed from
\[
H(x)=\Gamma(x)\mu(x)u(x),\qquad x\in\Omega.
\]
We consider the diffusion approximation \eqref{eq:pat-diffusive}
of light propagation:
\begin{equation}
-\div(D\nabla u)+\mu u=0\quad\text{in }\Omega.\label{eq:pat-diffusive-1}
\end{equation}
For simplicity, here we shall augment this equation with Dirichlet
boundary conditions, that are supposed to be measurable. Note that
$D$, $\Gamma$ and $\mu$ are unknowns of the problem. In contrast
with the case $\Gamma=1$, where \eqref{eq:pat-diffusive-1} was merely
used to compute the data but not in the inverse problem, we shall
make full use of this PDE. Let us briefly review the main known results
on this inverse problem. Bal and Ren \cite{bal-kui-2011} showed that
when $\Gamma$, $D$ and $\mu$ are all unknown, then there is no
uniqueness for the inverse problem  even with all the measurements
$H$ for all solutions $u$ to \eqref{eq:pat-diffusive-1}; namely, the PDE approach alone is not sufficient to reconstruct all the parameters. If any
of these parameters is known, then the others may be reconstructed
by using the PDE. The same authors have proved that all the coefficients
may be uniquely reconstructed by using multi-frequency measurements,
under certain assumptions on the dependency of the parameters on the
frequency \cite{2012-bal-ren}. Naetar  and Scherzer \cite{2014-scherzer-naetar}
studied the case of piecewise constant parameters: all the unknowns can
be uniquely determined, but the method may be very sensitive to noise.

We propose here for the single-frequency case a mixed approach combining
the following aspects.
\begin{itemize}
\item As in \cite{bal-kui-2011}, the PDE \eqref{eq:pat-diffusive-1} can
be used in the reconstruction. However, one degree of freedom for
the parameters remain.
\item As in \cite{2014-scherzer-naetar}, the PDE method gives unique reconstruction
under the finite dimensionality assumption of the coefficient spaces.
\item The disjoint sparsity signal separation method may be applied to this
case as in $\S$~\ref{sub:Quantitative-photoacoustic-tomog-gamma=00003D1}.
\end{itemize}
The combination of such approaches consists in substituting the piecewise
constant assumption with the sparsity assumption, and then in the use of \eqref{eq:pat-diffusive-1}
to reconstruct all the parameters. More precisely, the reconstruction
algorithm proposed here is substantially divided into the following
three main steps.
\begin{enumerate}
\item By using the disjoint sparsity signal separation method applied to
\[
h_{i}=\log H_{i}=\log(\Gamma\mu)+\log u_{i},\qquad i=1,\dots,N,
\]
the solutions $u_{i}$ are reconstructed.
\item Following \cite{bal-kui-2011}, by using the PDE
\[
-\div(Du_{i}\nabla\frac{u_{j}}{u_{i}})=0\quad\text{in }\Omega,
\]
with three suitable measurements, the diffusion $D$ can be uniquely
determined.
\item Finally, the absorption can be directly reconstructed via
\[
\mu=\frac{\div(D\nabla u_{i})}{u_{i}}\quad\text{in }\Omega,
\]
and possibly averaging over $i$.
\end{enumerate}

\subsubsection{The reconstruction algorithm}

Even though theoretically satisfactory, the algorithm summarized above
is not applicable in practice as it stands. Indeed, the reconstruction
of $D$ in (2) is not too sensitive to errors in $u_{j}$, but that
of $\mu$ in (3) is. To understand this, we compare the solutions
$u_{i}$ to \eqref{eq:pat-diffusive-1} with $D=1$ and $\mu$ as
in Figure~\ref{fig:A-nonoise-mureal} and the solutions $u_{i}^{0}$
to \eqref{eq:pat-diffusive-1} with $D=1$ and $\mu=\mu^{0}=1$, with
boundary conditions given by \eqref{eq:phii}:
\begin{equation}
\left\{ \begin{array}{l}
-\div(D\nabla u_{i})+\mu u_{i}=0\quad\text{in }\Omega,\\
-\div(D\nabla u_{i}^{0})+\mu^{0}u_{i}^{0}=0\quad\text{in }\Omega,\\
u_{i}=u_{i}^{0}=\phi_{i}\quad\text{on \ensuremath{\bo}}.
\end{array}\right.\label{eq:u_u0}
\end{equation}
The solutions are shown in Figure~\ref{fig:comparison}. Looking
at the first two columns, it is clear that the variations between
$u_{i}$ and $u_{i}^{0}$ are minimal. This is due to the fact that
the two problems have the same diffusion coefficients and small variations
in the absorption coefficients. As we saw in $\S$~\ref{sub:Quantitative-photoacoustic-tomog-gamma=00003D1},
the reconstruction at step (1) cannot be at this level of precision,
and therefore $\mu$ cannot be reconstructed in this simple way. In
order to overcome this difficulty, we make the following observation.

\begin{figure}
\subfloat[$u_{1}$]{\includegraphics[clip,width=0.3\columnwidth]{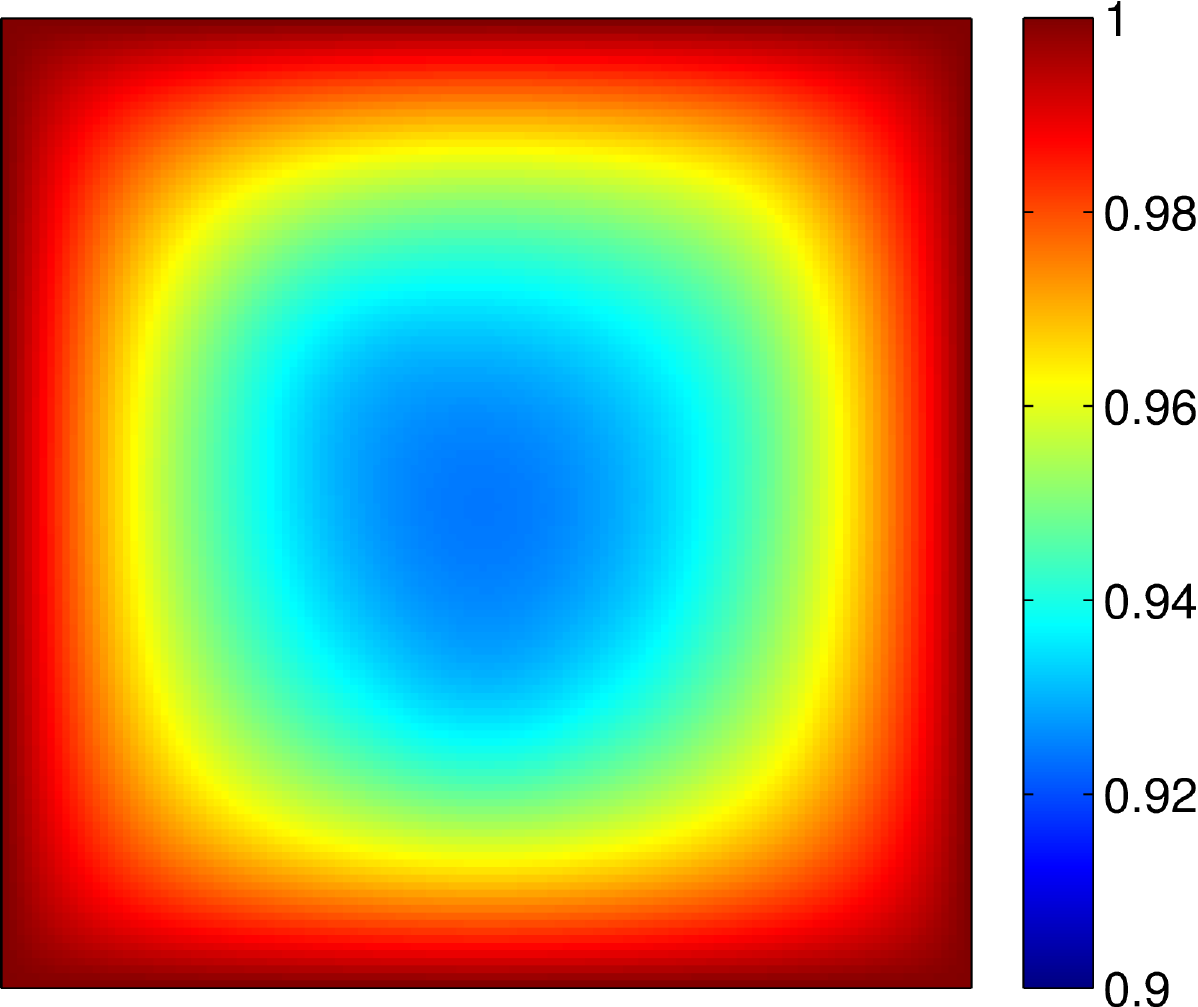}\label{fig:compare_u1}}\quad{}
\subfloat[$u_{1}^{0}$]{\includegraphics[clip,width=0.3\columnwidth]{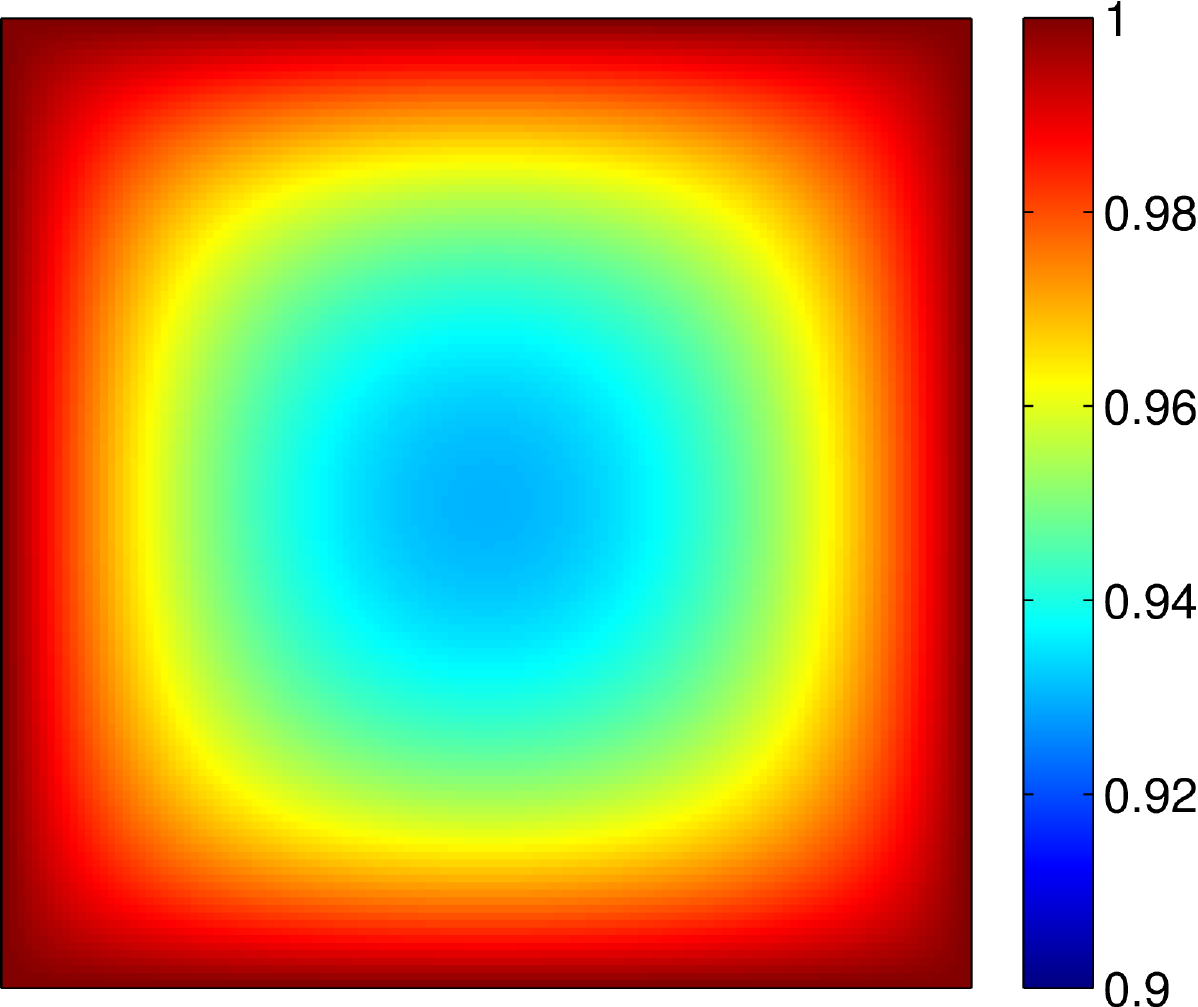}\label{fig:compare_u01}}\quad{}\subfloat[$u_{1}/u_{1}^{0}$]{\includegraphics[clip,width=0.31\columnwidth]{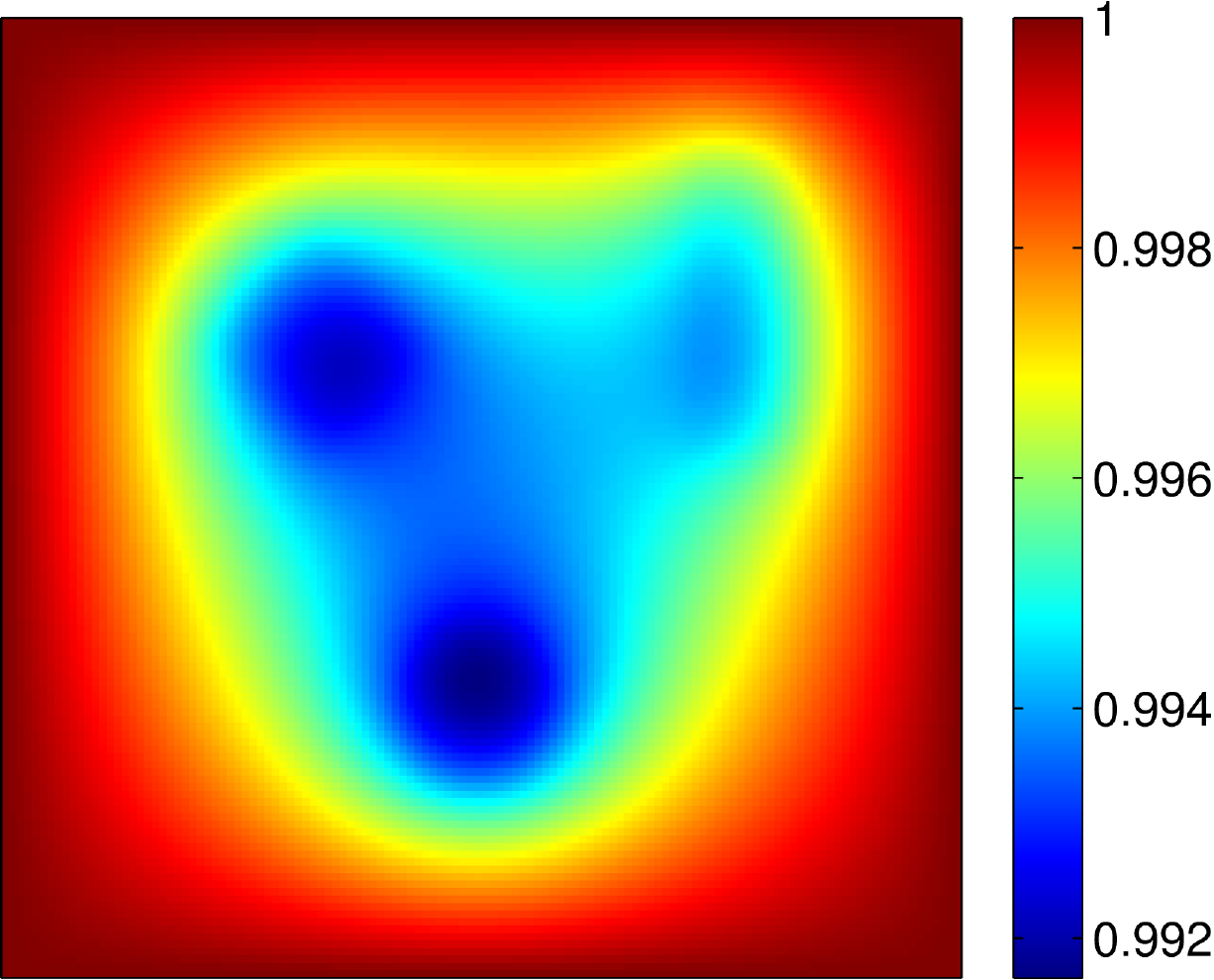}\label{fig:compare_u1/u01}}

\subfloat[$u_{2}$]{\includegraphics[clip,width=0.3\columnwidth]{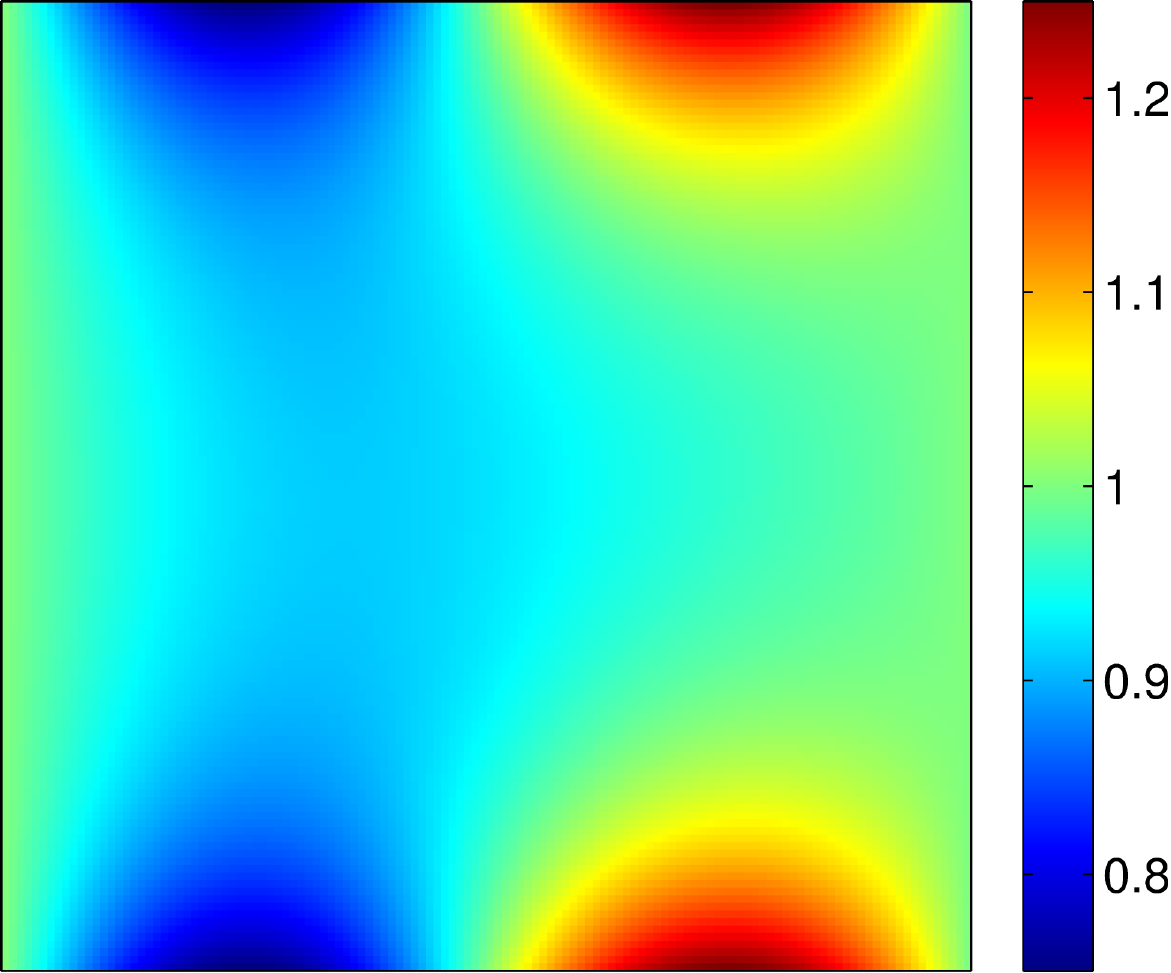}\label{fig:compare_u2}}\quad{}
\subfloat[$u_{2}^{0}$]{\includegraphics[clip,width=0.3\columnwidth]{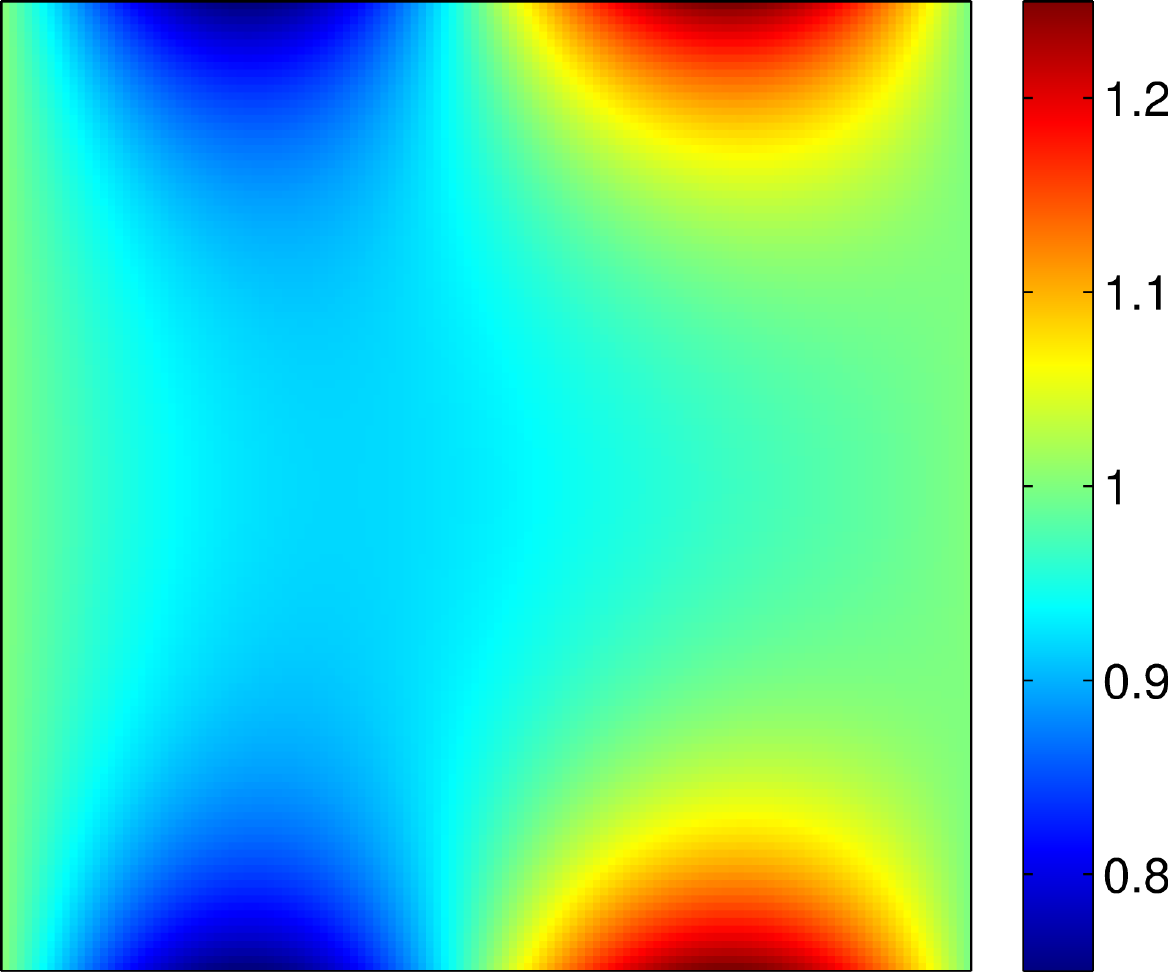}\label{fig:compare_u02}}\quad{}\subfloat[$u_{2}/u_{2}^{0}$]{\includegraphics[clip,width=0.31\columnwidth]{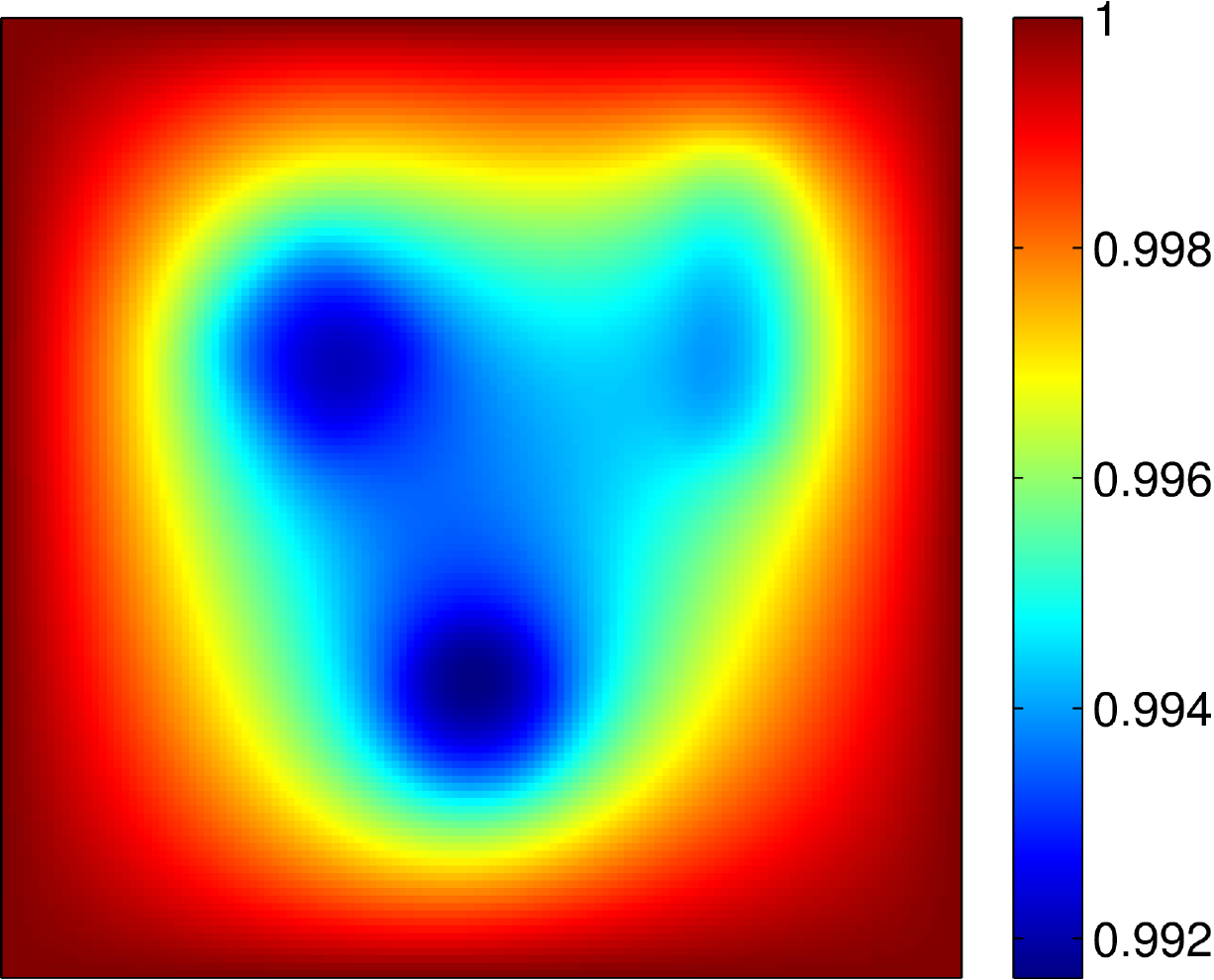}\label{fig:compare_u2/u02}}

\subfloat[$u_{3}$]{\includegraphics[clip,width=0.3\columnwidth]{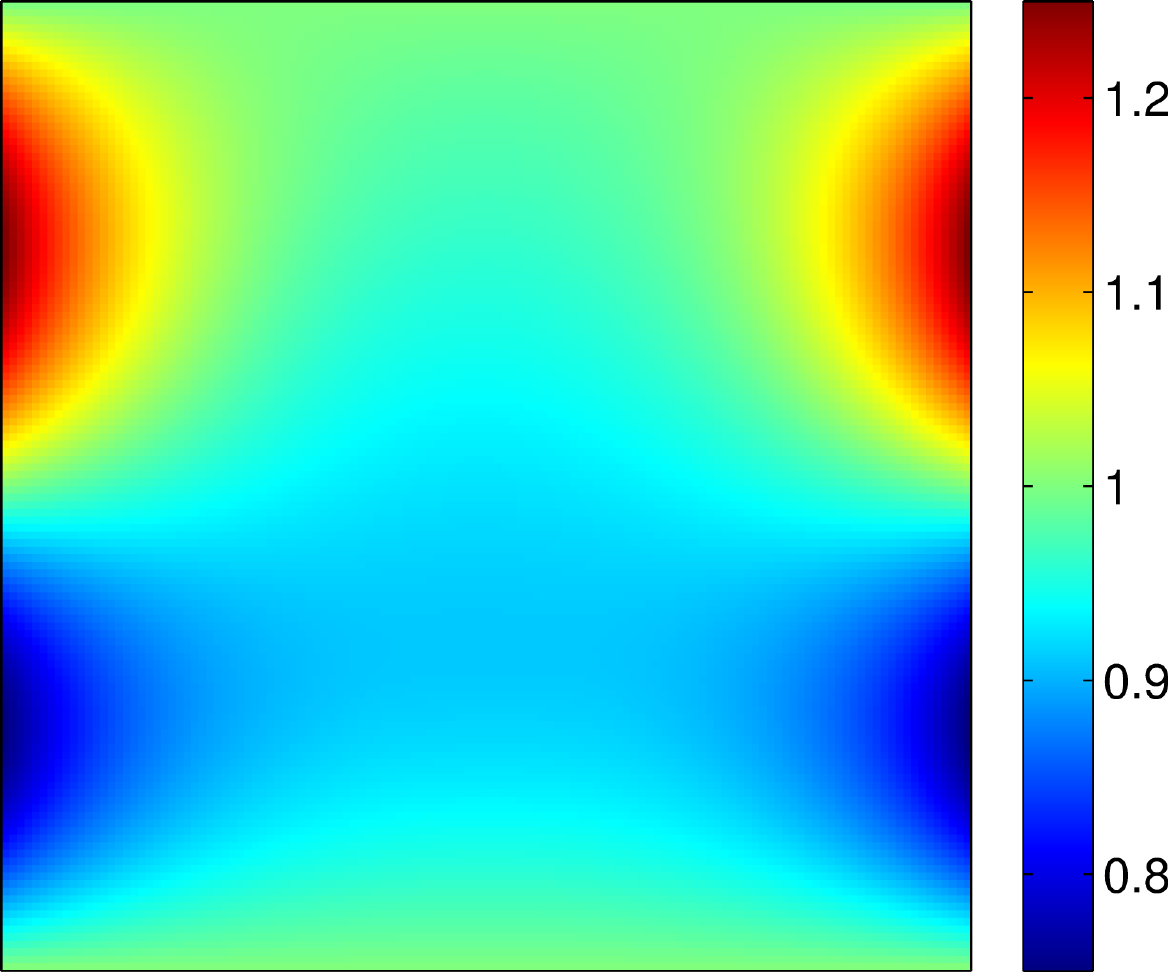}\label{fig:compare_u3}}\quad{}
\subfloat[$u_{3}^{0}$]{\includegraphics[clip,width=0.3\columnwidth]{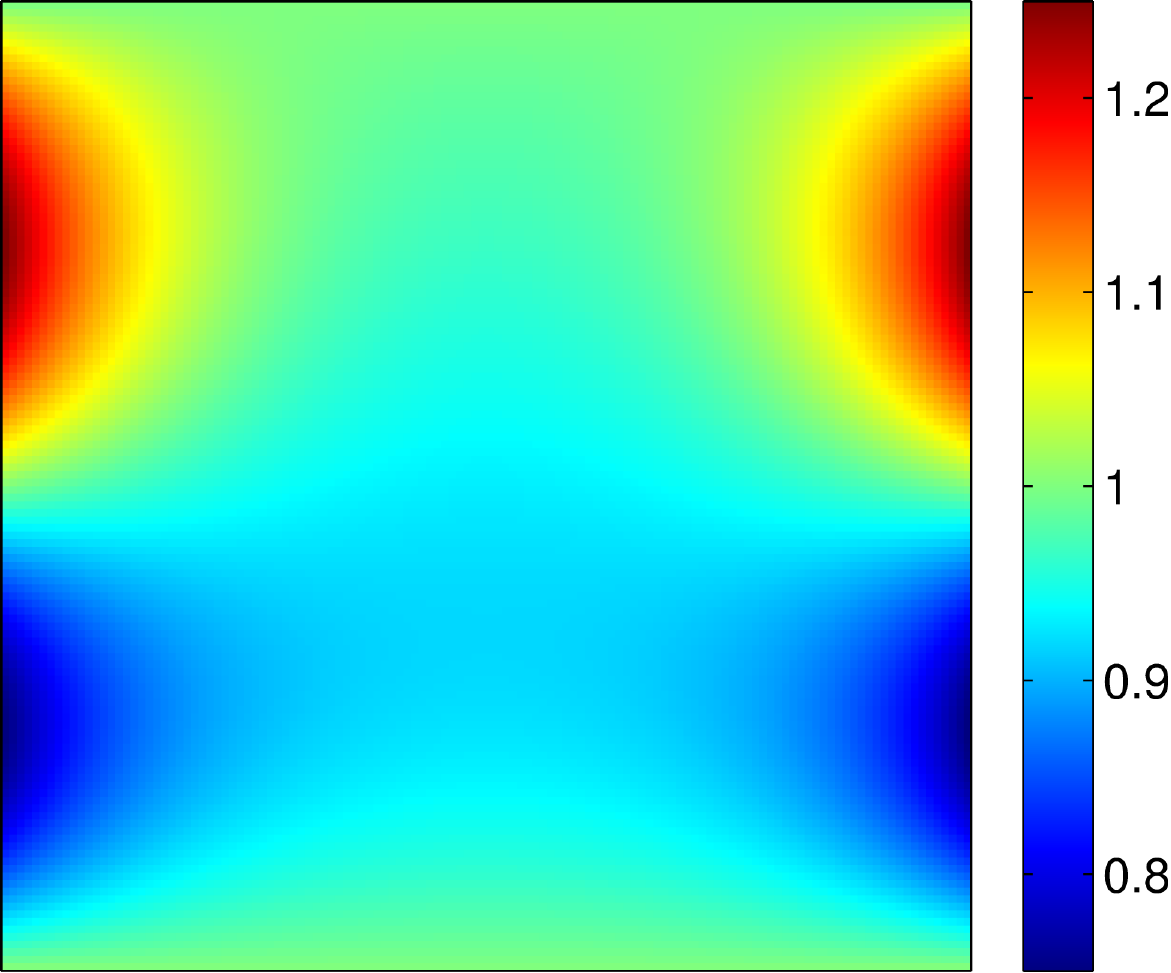}\label{fig:compare_u03}}\quad{}\subfloat[$u_{3}/u_{3}^{0}$]{\includegraphics[clip,width=0.31\columnwidth]{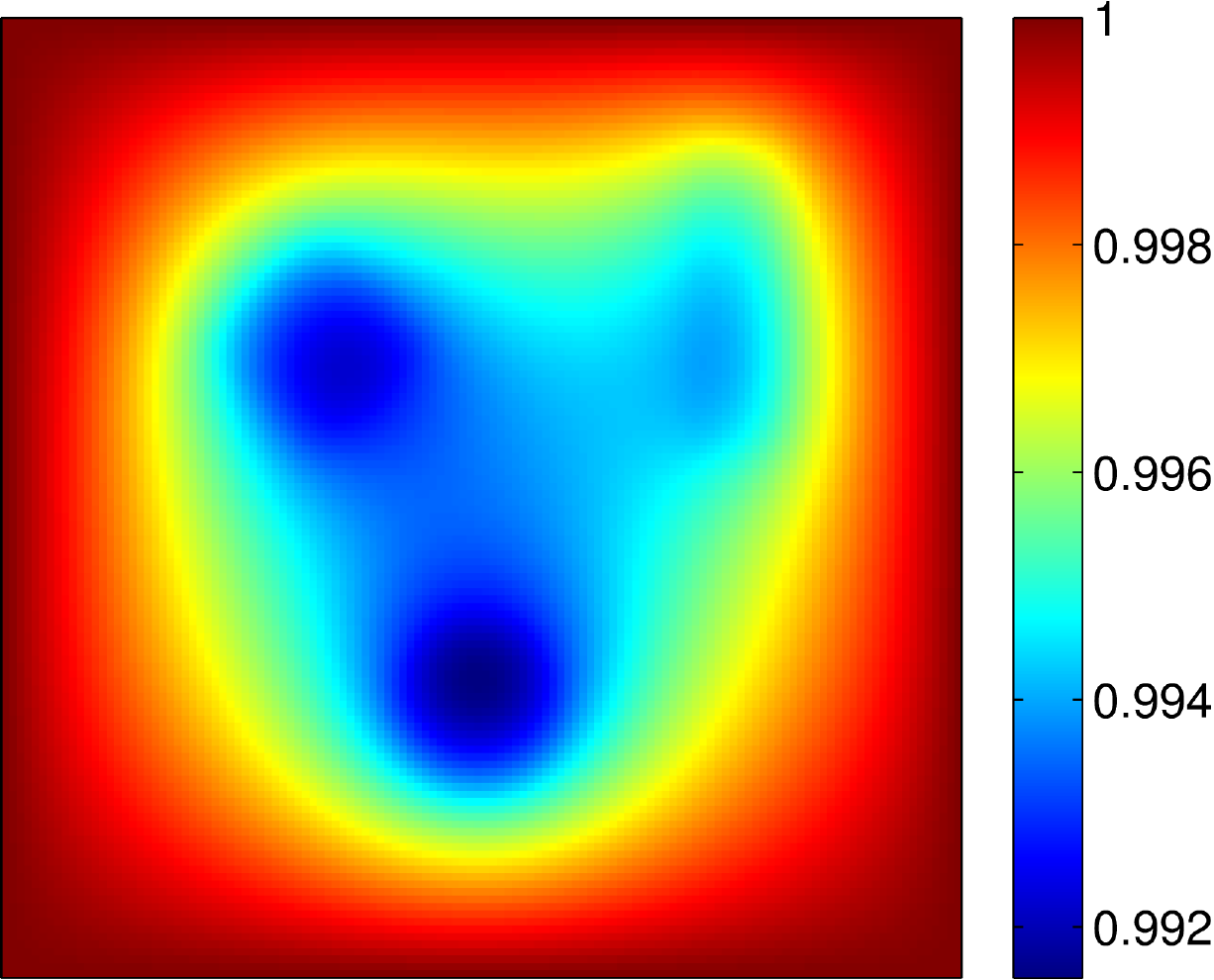}\label{fig:compare_u3/u03}}\caption{The solutions to \ref{eq:u_u0} and their ratios. }
\label{fig:comparison}
\end{figure}

\begin{rem}
\label{rem}The ratios $u_{i}/u_{i}^{0}$ are almost independent of
$i$, provided that $\mu$ is a small variation around a known background
$\mu^{0}$. This is evident from the third column of Figure~\ref{fig:comparison},
and can be proved by arguing as follows. A direct calculation gives
that $v_{i}=u_{i}/u_{i}^{0}$ satisfies
\[
\left\{ \begin{array}{l}
-\div(D\nabla v_{i})+(\mu-\mu^{0})v_{i}=2\frac{\nabla u_{i}^{0}}{u_{i}^{0}}\cdot\left(\frac{\nabla u_{i}}{u_{i}}-\frac{\nabla u_{i}^{0}}{u_{i}^{0}}\right)\quad\text{in }\Omega,\\
v_{i}=1\quad\text{on \ensuremath{\bo}}.
\end{array}\right.
\]
When $\mu$ is close to $\mu^{0}$, the right-hand side of this equation
becomes negligible with respect to the other terms. Thus, $v_{i}$
is substantially independent of $i$.
\end{rem}
Let us now describe the precise reconstruction algorithm based on
these observations. It consists of two initial steps and an iterative
procedure consisting of three more substeps. For simplicity, we shall
discuss only the noise-free case. We suppose that $\mu$ is a small
perturbation around a known coefficient $\mu^{0}$ and that $D$ is
known at one point of the domain.
\begin{enumerate}
\item By using the disjoint sparsity signal separation method applied to
\[
h_{i}=\log(\Gamma\mu)+\log u_{i},\qquad i=1,\dots,N,
\]
a first approximation $u_{i}(0)$ of the solutions $u_{i}$ is reconstructed.
As discussed in Section~\ref{sec:Numerical-implementation}, this can be done
by minimizing
\[
\min_{y\in\R^{m_{f}+Nm_{g}}}\left\Vert y\right\Vert _{0}+\lambda\sum_{i=1}^{N}\bigl\Vert A_{f}y_{f}+A_{g}y_{g}^{i}-h_{i}\bigr\Vert_{2}
\]
with OMP, and then writing $u_{i}(0)=\exp(A_{g}y_{g}^{i})$.
\item By using the computed $u_{i}(0)$ and the PDE
\begin{equation}
-\div(Du_{i}\nabla\frac{u_{j}}{u_{i}})=0\quad\text{in }\Omega\label{eq:D}
\end{equation}
with three suitable measurements, a first approximation $D(0)$ of
the diffusion can be obtained. Indeed, choose three boundary values
$\phi_{1}$, $\phi_{2}$ and $\phi_{3}$ such that
\begin{equation}
\det(\nabla\frac{u_{2}}{u_{1}},\nabla\frac{u_{3}}{u_{1}})>0,\quad\text{in \ensuremath{\Omega}.}\label{eq:alessandrini}
\end{equation}
(This can be easily done in two dimensions \cite{bal-kui-2011}.)
Then the above PDE may be rewritten as
\[
^{t}(\nabla\log D)=-\begin{bmatrix}\div(u_{1}\nabla\frac{u_{2}}{u_{1}})/u_{2} & \div(u_{1}\nabla\frac{u_{3}}{u_{1}})/u_{3}\end{bmatrix}\begin{bmatrix}\nabla\frac{u_{2}}{u_{1}} & \nabla\frac{u_{3}}{u_{1}}\end{bmatrix}^{-1}\quad\text{in \ensuremath{\Omega}},
\]
which can be integrated over $\Omega$ and gives a unique solution
for the diffusion coefficient, since $D$ is known at one point of
the domain. Since the solutions $u_{i}$ are very sensitive to changes
in $D$, we expect this reconstruction to be satisfactory. From the
numerical point of view, an optimal control approach may be applied
to \eqref{eq:D} to find $D$. 
\item We now start the main iterative procedure. Initialize $\mu(0)=\mu^{0}$
and let $u_{i}(0)$ and $D(0)$ be as reconstructed in points (1)
and (2). From $\mu(k)$, $u_{i}(k)$ and $D(k)$, the following iteration
is computed as follows.

\begin{enumerate}
\item Given $D(k)$ and $\mu(k)$, let $u_{i}^{0}(k)$ be the unique solution
to
\[
\left\{ \begin{array}{l}
-\div(D(k)\nabla u_{i}^{0}(k))+\mu(k)u_{i}^{0}(k)=0\quad\text{in }\Omega,\\
u_{i}^{0}(k)=\phi_{i}\quad\text{on \ensuremath{\bo}}.
\end{array}\right.
\]
Since $\phi_{i}$ is known, $u_{i}^{0}(k)$ is a known datum. Therefore
we can measure
\[
h_{i}^{0}=\log(H_{i}/u_{i}^{0}(k))=\log(\Gamma\mu)+\log\frac{u_{i}}{u_{i}^{0}(k)},\qquad i=1,\dots,N.
\]
In view of Remark~\ref{rem}, the quantities $u_{i}/u_{i}^{0}(k)$
are almost independent of $i$. This leads to the minimization of
\begin{multline*}
\qquad\quad\min_{y\in\R^{m_{f}+(N+1)m_{g}}}\left\Vert y\right\Vert _{0}+\lambda_{1}\sum_{i=1}^{N}\bigl\Vert A_{f}y_{f}+A_{g}y_{g}^{i}-h_{i}\bigr\Vert_{2}\\
+\lambda_{2}\sum_{i=1}^{N}\bigl\Vert A_{f}y_{f}+A_{g}y_{g}^{N+1}-h_{i}^{0}\bigr\Vert_{2}
\end{multline*}
with $\lambda_{1}\ll\lambda_{2}$. The second term maintains the incoherence
among the $y_{g}^{i}$'s, on which this disjoint sparsity approach
is based. The third term forces the quantities $u_{i}/u_{i}^{0}(k)$
to be independent of $i$, and numerical evidence shows that this
gives a much better reconstruction than the one performed at point
(1). The multipliers $\lambda_{1}$ and $\lambda_{2}$ may be taken
dependent on $k$. Set $u_{i}(k+1)=u_{i}^{0}(k)\exp(A_{g}y_{g}^{N+1})$.
\item Given $u_{i}(k+1)$, find a better approximation $D(k+1)$ of the
diffusion coefficient by proceeding as in (2).
\item Reconstruct $\mu(k+1)$ via
\begin{equation}
\mu(k+1)=\frac{1}{N}\sum_{i=1}^{N}\frac{\div(D(k+1)\nabla u_{i}(k+1))}{u_{i}(k+1)}\quad\text{in }\Omega.\label{eq:mu(k+1)}
\end{equation}
From the numerical point of view, it may be useful to regularize $u_{i}(k+1)$
and $D(k+1)$ before taking the derivatives. Finally, a TV-regularization
of $\mu(k+1)$ may reduce the accumulated noise.
\end{enumerate}
\end{enumerate}
There is no obvious stopping criterion for this iterative procedure.
However, in the numerical simulations less than five iterations were
sufficient.

In the above algorithm we have assumed for simplicity that the boundary
values $\phi_{i}$ are measurable. However, this is probably not a
necessary conditions, since they may be obtained from point (1) as
$\phi_{i}=u_{i}(0)_{|\bo}$.

\subsubsection{Numerical simulations}

We have tested the algorithm discussed above with the absorption map
$\tilde{\mu}$ considered in $\S$~\ref{subsub:Example-1--} (see
Figure~\ref{fig:mureal}). The same dictionaries considered in $\S$~\ref{sub:Quantitative-photoacoustic-tomog-gamma=00003D1}
are chosen. The light intensities $\tilde{u}_{i}$ are the solutions
of
\[
\left\{ \begin{array}{l}
-\div(\tilde{D}\nabla\tilde{u}_{i})+\tilde{\mu}\tilde{u}_{i}=0\quad\text{in }\Omega,\\
\tilde{u}_{i}=\phi_{i}\quad\text{on \ensuremath{\bo}},
\end{array}\right.
\]
where the diffusion coefficient $\tilde{D}$ is shown in Figure~\ref{fig:Dreal}
and five boundary values are chosen as follows:
\[
\begin{aligned} & \phi_{1}(x)=1,\\
 & \phi_{2}(x)=1-\sin(2\pi x_{1})/8,\\
 & \phi_{3}(x)=1-\sin(2\pi x_{2})/8,\\
 & \phi_{4}(x)=x_{1}/4+7/8,\\
 & \phi_{5}(x)=x_{2}/4+7/8.
\end{aligned}
\]
The internal data take the form
\[
H_{i}=\tilde{\Gamma}\tilde{\mu}\tilde{u}_{i},\qquad i=1,\dots5,
\]
where the Grüneisen parameter is shown in Figure~\ref{fig:gammareal}.
The measurements corresponding to the first three boundary conditions
will be used for the disjoint sparsity signal separation method (with
$N=3$), namely for the steps (1) and (3a) of the above algorithm;
those corresponding to $\phi_{1}$, $\phi_{4}$ and $\phi_{5}$ will
be used in the steps (2) and (3b), in order to satisfy \eqref{eq:alessandrini}.
All measurements are used in the last step (3c).

\begin{figure}
\subfloat[$\tilde{D}$]{\includegraphics[clip,width=0.31\columnwidth]{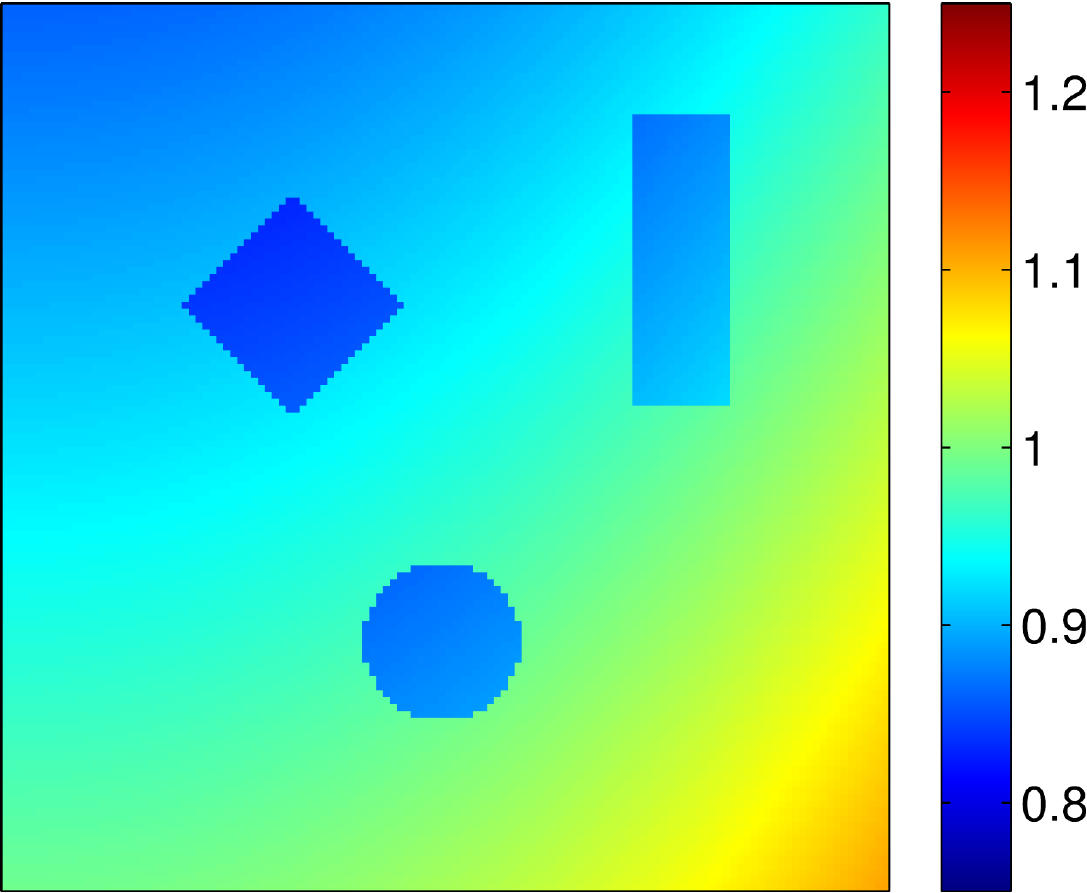}\label{fig:Dreal}}\quad{}
\subfloat[$\tilde{\Gamma}$]{\includegraphics[clip,width=0.31\columnwidth]{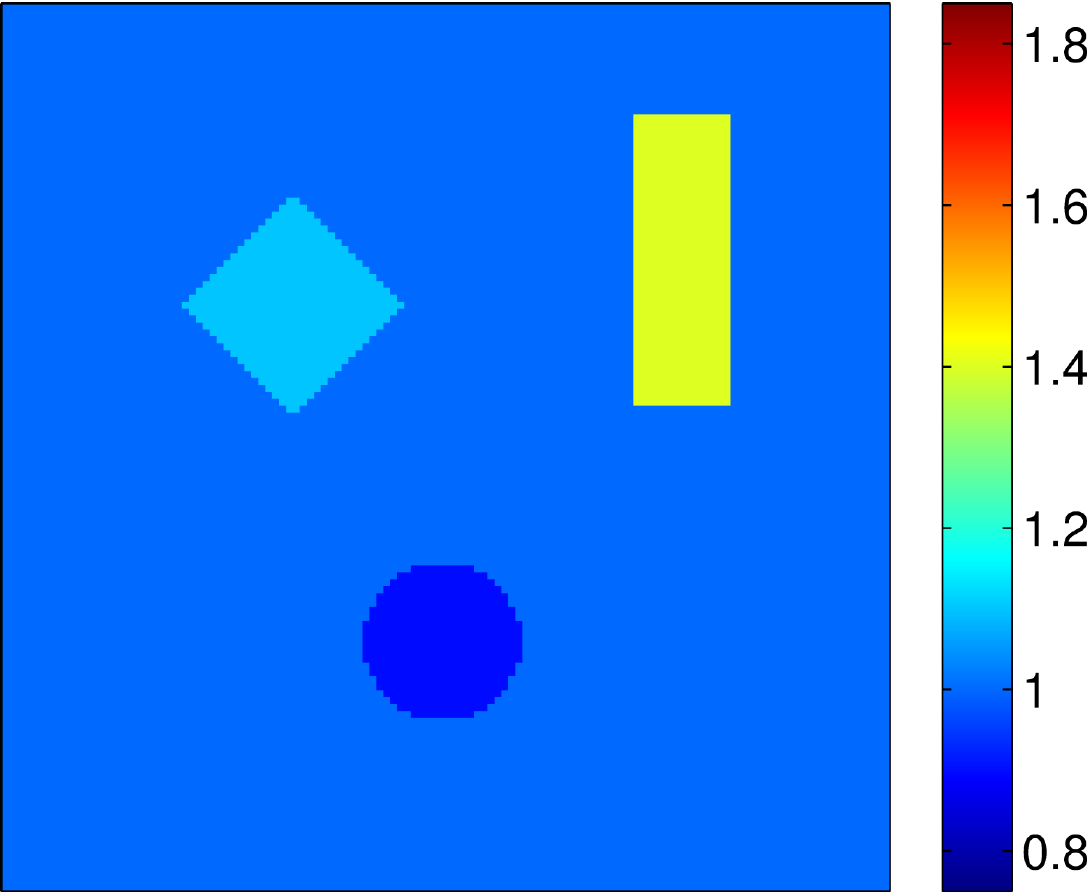}\label{fig:gammareal}}\quad{}\subfloat[$\tilde{\mu}$]{\includegraphics[clip,width=0.31\columnwidth]{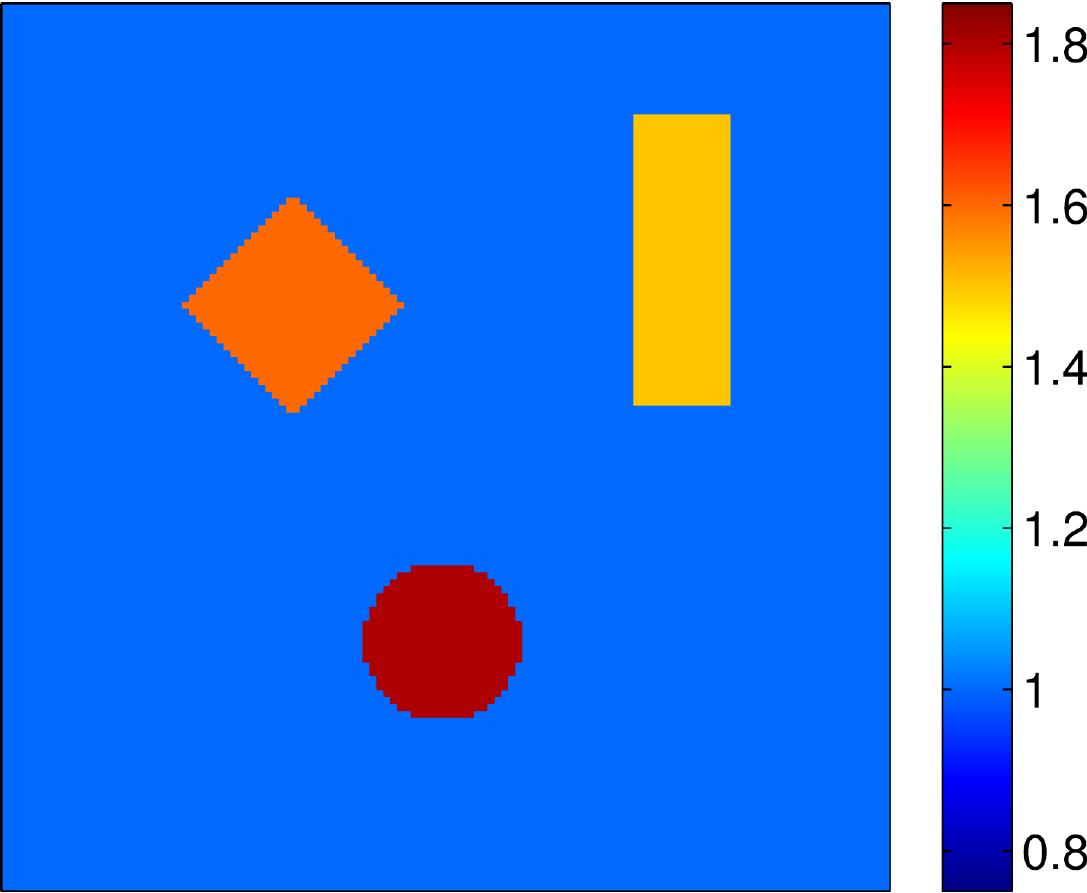}\label{fig:mureal}}

\subfloat[$D$]{\includegraphics[clip,width=0.31\columnwidth]{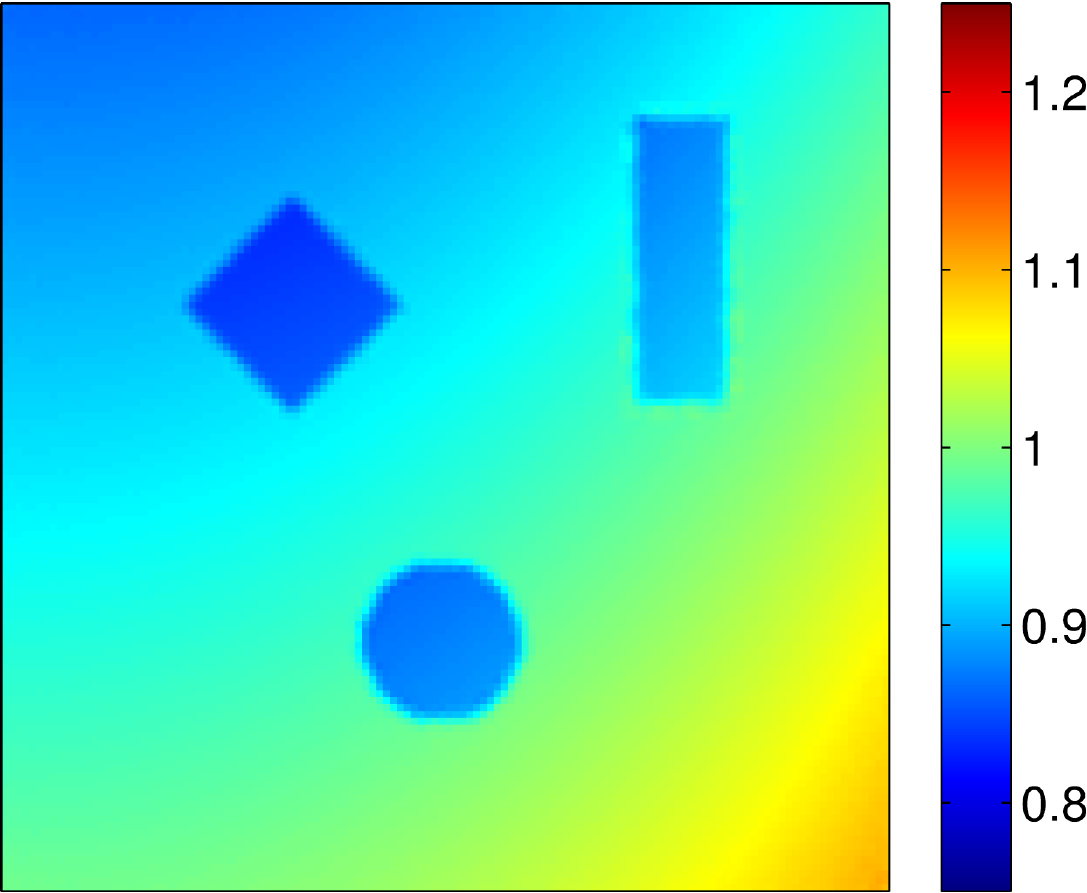}\label{fig:D}}\quad{}
\subfloat[$\mu$]{\includegraphics[clip,width=0.31\columnwidth]{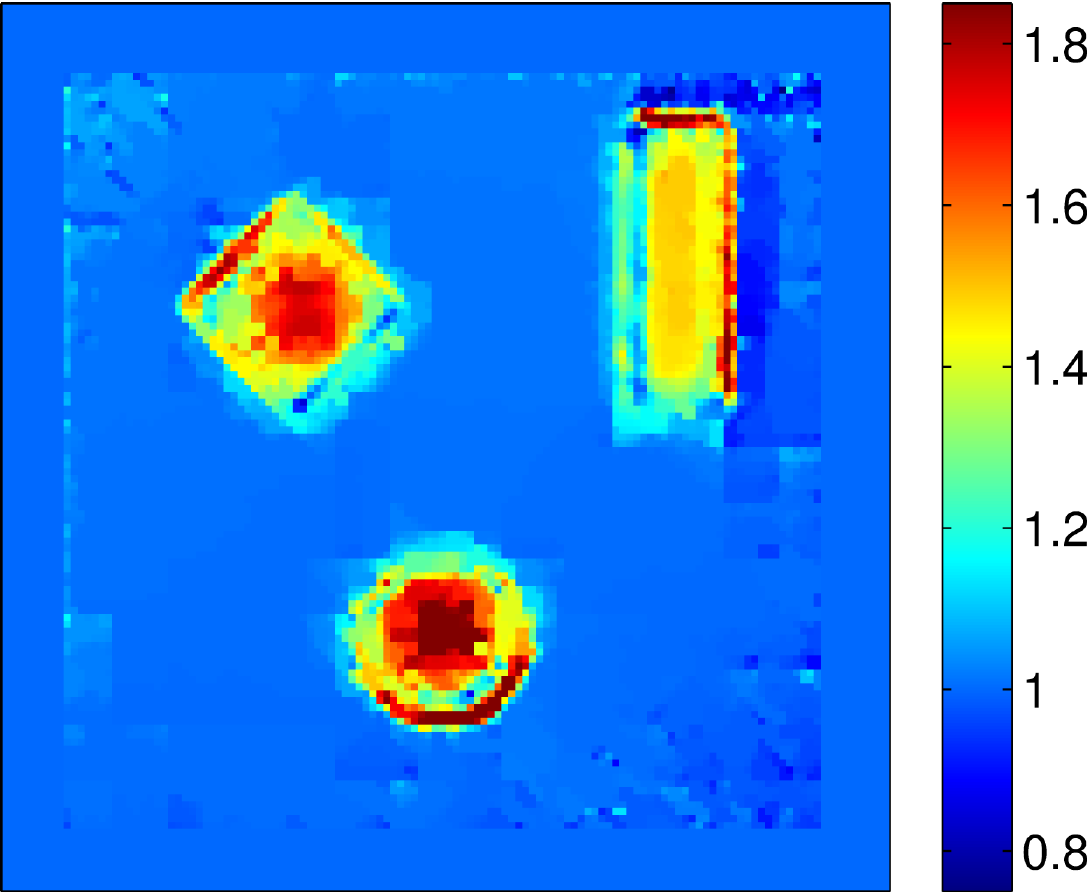}\label{fig:mu_0}}\quad{}\subfloat[$\mu(2)$]{\includegraphics[clip,width=0.31\columnwidth]{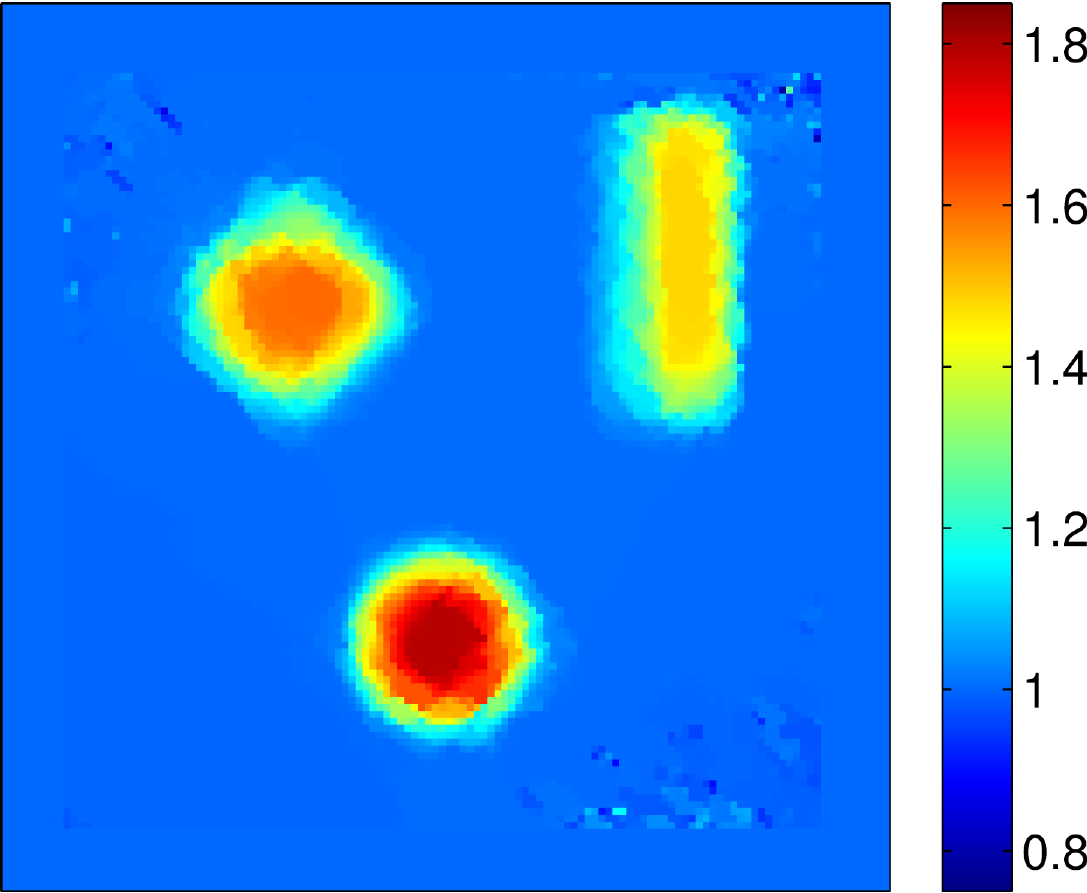}\label{fig:mu_2}}

\caption{The case with variable $\Gamma$.}
\label{fig:gammano1}
\end{figure}
The OMP iterative procedures are stopped after 2000 iterations. If
the absorption $\mu$ were recovered via \eqref{eq:mu(k+1)} immediately
after steps (1) and (2), the reconstruction would not be satisfactory,
as it can be seen in Figure~\ref{fig:mu_0}. This makes steps (3a)
and (3b) necessary: after repeating step (3) twice, the quality of
the reconstruction is sensibly improved (see Figure~\ref{fig:mu_2}).
The corresponding reconstruction of $D$ is shown in Figure~\ref{fig:D}.
As anticipated before, the reconstruction of $D$ from $u_{i}$ is
much more stable than that of $\mu$. Note that the absorption was
supposed to be known near the boundary of the domain, to avoid problems
with the second derivatives in \eqref{eq:mu(k+1)}. 

The case with noise was not studied in this paper, since the robustness
to noise of the disjoint sparsity signal separation algorithm has
already been tested in $\S$~\ref{sub:Quantitative-photoacoustic-tomog-gamma=00003D1}.
The robustness to noise of the other steps of the reconstruction method
discussed above is well-known, and standard regularization method
can be employed. It is worth noticing that since the absorption $\mu$
is found in step (3c) from the reconstructed values of the light
intensities $u_{i}$, the signal to noise ratio of $u_{i}$ has to
be sufficiently large. Unfortunately, as it can be seen from the third
column of Figure~\ref{fig:comparison}, the amplitude of the information
about $\mu$ captured in $u_{i}$ is very small, and so the noise
has to be comparable to this.

\section{\label{sec:Conclusions}Conclusions}

In this work we have studied a method for signal separation based
on the disjoint sparsity of multiple measurements. A theorem giving
unique and stable reconstruction was proved. The result is based on
the incoherence of the measurements. Then, the method was applied
to hybrid imaging problems, and in particular to quantitative photoacoustic
tomography. This technique has been successfully tested on several
numerical simulations, and results to be very robust to noise.

The incoherence between the measurements $g_{i}$ is the main foundation
of the method, and the numerical simulations have shown that such
property holds true with different solutions to the same PDE, which
is the relevant case for hybrid imaging. It would be very interesting
to prove this result in general. Randomly chosen boundary conditions
may give the necessary incoherence.

Robust Principal Component Analysis (rPCA) \cite{2011-candes-rpca} might be used as an alternative to the disjoint separation method for the recovery of $\tilde{f}$ and the $\tilde{g}_i$'s from the knowledge of $h_i=\tilde{f}+\tilde{g}_i$, $i=1,\dots,N$. Writing $\tilde{f}=A_f \tilde{y}_f$ and $\tilde{g}_i=A_g \tilde{y}_g^i$, the known matrix $M:=\,^t\!A_g\begin{bmatrix}
           h_1 & \cdots & h_N 
          \end{bmatrix}$ can be expressed as
\[
M=
\begin{bmatrix}
           ^t\!A_g A_f \tilde{y}_f & \cdots &^t\!A_g A_f \tilde{y}_f 
          \end{bmatrix}
          +
          \begin{bmatrix}
           \tilde{y}_g^1 & \cdots & \tilde{y}_g^N 
          \end{bmatrix},
\]
that is the sum of a rank-one matrix and of a sparse matrix. It would be interesting to see whether the requirements for the application of rPCA are fulfilled. Heuristically, this is indeed the case: the incoherence between $A_f$ and $A_g$ should provide the non-sparsity of the low-rank component, while the incoherence, or possibly the disjoint sparsity, of the $\tilde{g}_i$'s should ensure that the sparse component is not low-rank.

Finally, it would be very interesting to investigate whether the main
ideas behind this method can be applied to other inverse problems
with multiple measurements consisting of two components, of which
only one remains fixed.

\appendix

\section{Uncertainty principles\label{sec:appendix}}

\subsection{Proof of Proposition~\ref{prop:normalized_UP}\label{subsec:appendix}}

Let $A_f$ and $A_g$ be as in Proposition~\ref{prop:normalized_UP}. For $p=1,\dots,n$ and $q\in\R^n$ define
\[
 \xi_p(q)=\max_{1\le\alpha_1<\cdots<\alpha_p\le n} \min_{i=1,\dots,p} |q(\alpha_i)|.
\]
Note that $\xi_p(q)>0$ if and only if $\left\Vert q\right\Vert_0\ge p$: in this sense, the map $\xi$ is a quantitative version of the norm $\left\Vert \cdot\right\Vert_0$. Let $\ceil(z)$ denote the nearest integer greater than or equal to $z$.
\begin{lem}\label{lem:xi}
 Take $\zeta=1,\dots,m_g$ and $p=1,\dots,n$ such that
 \begin{equation*}
  \norm{\,^t\!A_f A_g v}_0\ge p,\qquad v\in\R^{m_g}\setminus\{0\},\; \norm{v}_0\le \zeta.
 \end{equation*}
There exists $C_\xi>0$ depending only on $n$, $A_f$ and $A_g$ such that
 \[
  \xi_p(\,^t\!A_fA_g v)\ge C_\xi,
 \]
 for every $v\in\R^{m_g}$ such that $\left\Vert v \right\Vert_2=1$ and $\left\Vert v \right\Vert_0\le \zeta$.
\end{lem}
\begin{proof}
 By contradiction, there exist $v_l\in\R^{m_g}$ such that $\left\Vert v_l \right\Vert_2=1$, $\left\Vert v_l \right\Vert_0\le \zeta$ and $\xi_p(q_l)\to 0$ as $l\to\infty$, where $q_l=\,^t\!A_fA_g v_l$. Up to a subsequence, we have $v_l\to v$ for some $v\in\R^{m_g}$ such that $\left\Vert v \right\Vert_2=1$ and $\left\Vert v \right\Vert_0\le \zeta$ and $q_l\to\,^t\!A_fA_g v=:q$. By assumption we have  $\left\Vert q \right\Vert_0\ge p$. Thus, there exist $1\le\alpha_1<\cdots<\alpha_p\le n$ and $\epsilon>0$ such that $|q(\alpha_i)|\ge 2\epsilon$ for every $i=1,\dots,p$. Since $q_l\to q$, there exists $l_0$ such that $|q_l(\alpha_i)|\ge \epsilon$ for every $i=1,\dots,p$ and every $l\ge l_0$. Hence $\xi_p(q_l)\ge \epsilon$ for every  $l\ge l_0$, a contradiction.
\end{proof}
We are now in a position to prove Proposition~\ref{prop:normalized_UP}.
\begin{proof}[Proof of Proposition~\ref{prop:normalized_UP}]
 Define $D=(\sqrt{m_g}+\frac{2}{3})(1+C_\xi^{-1})$ and take $q\in\R^n$ such that $\left\Vert A_f q \right\Vert_2 > D$ and  $\left\Vert ^t\!A_g^\perp A_f q\right\Vert _{2}\le \frac{2}{3}$. Write $^t\!A_gA_f q=v+r$ where
 \[
  v(\alpha)=\begin{cases}
               ^t\!A_gA_f q(\alpha) & \text{if $|^t\!A_gA_f q(\alpha)|\ge 1$,}\\
               0 & \text{otherwise.}
              \end{cases}
 \]
 Thus $\left\Vert r \right\Vert_\infty\le 1$, whence $\left\Vert r \right\Vert_2 \le \sqrt{m_g}$. Therefore, by the estimate
 \[
 D<  \left\Vert A_f q \right\Vert_2=(\left\Vert ^t\!A_g A_f q \right\Vert_2^2+\left\Vert ^t\!A_g^\perp  A_f q \right\Vert_2^2)^\frac{1}{2}\le \left\Vert ^t\!A_g A_f q \right\Vert_2 + \frac{2}{3},
 \]
and by construction of $D$ we obtain
\begin{equation}\label{eq:v0}
  \left\Vert v \right\Vert_2\ge  \left\Vert  ^t\!A_g A_f q \right\Vert_2 - \left\Vert  r \right\Vert_2>D-( \sqrt{m_g}+\frac{2}{3})=(\sqrt{m_g}+\frac{2}{3})C_\xi^{-1}>0.
\end{equation}
Set $v':=v/\left\Vert v \right\Vert_2$, $\zeta:=\left\Vert v' \right\Vert_0$ and $p=\ceil(2/M)-\zeta$. By Lemma~\ref{lem:xi}, whose assumptions are satisfied by Proposition~\ref{prop:uncertainty} (using that $A_f\,^t\!A_f=I$),  we have $\xi_p(\,^t\!A_fA_g v')\ge C_\xi$. Thus, there exist $1\le\alpha_1<\cdots<\alpha_p\le n$ such that
\begin{equation}\label{eq:xi}
 |^t\!A_fA_g v'(\alpha_i)|\ge C_\xi,\qquad \text{for every $i=1,\dots,p$.}
\end{equation}
From the identity $A_f q=A_g(\,^t\!A_gA_f q)+A_g^\perp(\,^t\!A_g^\perp A_f q)$, setting $z=\,^t\!A_f(A_g r +A_g^\perp(\,^t\!A_g^\perp A_f q))$ we obtain
\[
 q=(^t\!A_fA_g v')\left\Vert v \right\Vert_2 +z.
\]
In view of Lemma~\ref{lem:properties} parts 1 and 2 we have
\[
 \left\Vert z \right\Vert_\infty\le\left\Vert z \right\Vert_2\le \left\Vert A_g r +A_g^\perp(\,^t\!A_g^\perp A_f q) \right\Vert_2\le \left\Vert r \right\Vert_2 + \left\Vert ^t\!A_g^\perp A_f q \right\Vert_2\le \sqrt{m_g} +\frac{2}{3}.
\]
As a result, by \eqref{eq:xi} and  \eqref{eq:v0} and the expression for $D$ we have
\[
 |q(\alpha_i)|=|\left\Vert v \right\Vert_2 (^t\!A_fA_g v')(\alpha_i)+z(\alpha_i)|\ge \left\Vert v \right\Vert_2 C_\xi-( \sqrt{m_g}+\frac{2}{3})>0
\]
for every $i=1,\dots,p$. Therefore
\[
 \left\Vert q \right\Vert_0 +\left\Vert v \right\Vert_0 \ge p+\zeta\ge 2/M.
\]
Finally, the conclusion follows form the equality $\#\{\alpha:|(^t\!A_{g}A_f q)(\alpha)|\ge 1\}=\left\Vert v \right\Vert_0$.
\end{proof}

\subsection{Proof of Proposition~\ref{prop:new_uncertainty}\label{subsec:appendix2}}

Let us first discuss the orthobasis of 2D Haar wavelets of $\R^{d\times d}$, where $d=2^J$. $A_f$ is constructed via translations and dilations of four types of wavelets, as we now describe. Let $j=1,\dots,J-1$ denote the scale, from the finest to the coarsest and let $k_1,k_2=1,\dots,2^{J-j}$ be the horizontal and vertical translation parameters. We consider four families of atoms $\psi^i_{j,k}$ defined by
\begin{align*}
&\left\{\begin{array}{l}
         \psi^1_{j,k}(2^j(k_1-1)+\alpha_1,2^j(k_2-1)+\alpha_2)=-2^{-j},\\
         \psi^1_{j,k}(2^j(k_1-1)+\alpha_1,2^j(k_2-1)+2^{j-1}+\alpha_2)=2^{-j},
        \end{array}
\right.
\quad
&\begin{array}{l}
         \alpha_1=1,\dots,2^j,\\
         \alpha_2=1,\dots,2^{j-1},
        \end{array}
\\
&\left\{\begin{array}{l}
         \psi^2_{j,k}(2^j(k_1-1)+\alpha_1,2^j(k_2-1)+\alpha_2)=-2^{-j},\\
         \psi^2_{j,k}(2^j(k_1-1)+2^{j-1}+\alpha_1,2^j(k_2-1)+\alpha_2)=2^{-j},
        \end{array}
\right.
\quad
&\begin{array}{l}
         \alpha_1=1,\dots,2^{j-1},\\
         \alpha_2=1,\dots,2^{j},
        \end{array}
        \\
&\left\{\begin{array}{l}
         \psi^3_{j,k}(2^j(k_1-1)+2^{j-1}+\alpha_1,2^j(k_2-1)+\alpha_2)=-2^{-j},\\
         \psi^3_{j,k}(2^j(k_1-1)+\alpha_1,2^j(k_2-1)+2^{j-1}+\alpha_2)=-2^{-j},\\
         \psi^3_{j,k}(2^j(k_1-1)+\alpha_1,2^j(k_2-1)+\alpha_2)=2^{-j},\\
         \psi^3_{j,k}(2^j(k_1-1)\!+2^{j-1}\!+\!\alpha_1,2^j(k_2-1)\!+2^{j-1}\!+\!\alpha_2)=2^{-j},
        \end{array}
\right.\!\!\!\!
&\begin{array}{l}
         \alpha_1=1,\dots,2^{j-1},\\
         \alpha_2=1,\dots,2^{j-1},
        \end{array}
        \\
  &\quad\;\;\psi^4_{j,k}(2^j(k_1-1)+\alpha_1,2^j(k_2-1)+\alpha_2)=2^{-j},
  & \alpha_1,\alpha_2=1,\dots,2^{j},
\end{align*}
and zero elsewhere. The orthonormal basis of Haar wavelets $A_f$ is given by
\[
 \bigcup_{j=1}^{J-2}\bigl\{\psi_{j,k}^i: i=1,2,3,\, k\in\{1,\dots,2^{J-j}\}^2\bigr\}\cup
 \bigl\{\psi_{J-1,k}^i:i=1,\dots,4,\, k\in\{1,2\}^2\bigr\}.
\]
The proof of Proposition~\ref{prop:new_uncertainty} is based on the following result.
\begin{lem}\label{lem:new_uncertainty}
Under the assumptions of Proposition~\ref{prop:new_uncertainty}, for every $v\in\R^{m_g}\setminus\{0\}$ we have
\[
 \norm{\,^t\!A_f A_g v}_0\ge \sum_{j=1}^{J-B-1} (2^{J-j}-2L)^2
\]
\end{lem}
\begin{proof}
 Write $g=A_g v$. By construction of $A_g$, the vector $g$ can be written as a linear combination of low frequency real sinusoids, namely
\[
  g=\sum_{i=1}^4 \sum_{l_1,l_2=0}^L \gamma_l^i \chi_l^i
\]
for some weights $\gamma_l^i\in\R$.  Using standard trigonometric equalities, we can write the above sum in terms of complex sinusoids
\[
 g(\alpha)=\sum_{l_1,l_2=-L}^L \theta_l\, e^{2\pi i l_1 \frac{\alpha_1}{2^J}} \,e^{2\pi i l_2 \frac{\alpha_2}{2^{J}}},\qquad \alpha\in\{1,\dots,2^J\}^2.
\]
for some complex weights $\theta_l\in\C$. Since the constant vector is not in $A_g$ and $v\neq 0$, we have that $\theta_l\neq 0$ for some $l\neq (0,0)$. Without loss of generality, we can assume that there exists $l^*\in\{-L,\dots,L\}^2$ such that
\begin{equation}\label{eq:theta_l_nonzero}
\theta_{l^*}\neq 0,\qquad l_2^*\neq 0. 
\end{equation}
By construction of $A_f$ we have that
\begin{equation}\label{eq:T_j}
 \norm{\,^t\!A_f g}_0\ge \sum_{j=1}^{J-B-1}  \# \{k\in\{1,\dots,2^{J-j}\}^2:(g,\psi_{j,k}^1)_2\neq 0\}=:\sum_{j=1}^{J-B-1}  \# T_{j}
\end{equation}
We now want to find a lower bound for $\#T_j$.

Fix $j=1,\dots,J-B-1$. For simplicity of notation, set $\e_{l_i}(\cdot)=e^{2\pi i l_i \frac{\cdot}{2^J}}$ and $m_{k_i}=2^j(k_i-1)$. For $k\in \{1,\dots,2^{J-j}\}^2$ we have
\begin{multline*}
\begin{split}
 2^j(g,&\psi_{j,k}^1)_2=\sum_{l_1,l_2=-L}^L  \sum_{\alpha_1=1}^{2^j} \sum_{\alpha_2=1}^{2^{j-1}} \theta_l \bigl(\e_{l_1}(m_{k_1}+\alpha_1)\e_{l_2}(m_{k_2}+2^{j-1}+\alpha_2)\\
  &\qquad\qquad\qquad\qquad\qquad\qquad\qquad\qquad-\e_{l_1}(m_{k_1}+\alpha_1)\e_{l_2}(m_{k_2}+\alpha_2)\bigr)\\
  &=\sum_{l_1,l_2=-L}^L  \theta_l \e_{l_1}(m_{k_1})\e_{l_2}(m_{k_2}) \sum_{\alpha_1=1}^{2^j} \sum_{\alpha_2=1}^{2^{j-1}} \e_{l_1}(\alpha_1)\e_{l_2}(\alpha_2)\bigl(\e_{l_2}(2^{j-1})-1 \bigr)\\
  &=\sum_{l_2=-L}^L \zeta_{l_2}(k_1) \e_{l_2}(m_{k_2})
 \end{split}
\end{multline*}
where we have set
\[
 \zeta_{l_2}(k_1)=\sum_{l_1=-L}^L \theta_l \bigl(\e_{l_2}(2^{j-1})-1 \bigr) \Bigl( \sum_{\alpha_2=1}^{2^{j-1}} \e_{l_2}(\alpha_2)\Bigr)    \Bigl(\sum_{\alpha_1=1}^{2^j} \e_{l_1}(\alpha_1)\Bigr) \e_{l_1}(m_{k_1}).
\]
Since $L< 2^B$ with $B\le J-2$ and $j\le J-B-1$, by using standard identities for geometric sums it is easy to show that $ \bigl(\e_{l_2}(2^{j-1})-1 \bigr) \Bigl( \sum_{\alpha_2=1}^{2^{j-1}} \e_{l_2}(\alpha_2)\Bigr)    \Bigl(\sum_{\alpha_1=1}^{2^j} \e_{l_1}(\alpha_1)\Bigr)\neq 0$ for all $l_1$ and all $l_2\neq 0$. As a result, in view of \eqref{eq:theta_l_nonzero} we have that the polynomial in the complex variable $z$
\[
 p_{l_2^*}=\sum_{l_1=-L}^L \theta_{l_1,l_2^*} \bigl(\e_{l_2^*}(2^{j-1})-1 \bigr) \Bigl( \sum_{\alpha_2=1}^{2^{j-1}} \e_{l_2^*}(\alpha_2)\Bigr)    \Bigl(\sum_{\alpha_1=1}^{2^j} \e_{l_1}(\alpha_1)\Bigr) z^{l_1+L}
\]
is non trivial. By the fundamental theorem of algebra, it has at most $2L$ zeros. Therefore, writing $z=e^{2\pi i m_{k_1}2^{-J}}$, we obtain that $\#E_j\ge 2^{J-j}-2L$, where $E_j=\{k_1\in\{1,\dots,2^{J-j}\}:\zeta_{l_2^*}(k_1)\neq 0\}$. Take now $k_1\in E_j$. Arguing in a similar way and writing $z=e^{2\pi i m_{k_2}2^{-J}}$, we have that
\[
 \e_{L}(m_{k_2}) 2^j(g,\psi_{j,k}^1)_2= \sum_{l_2=-L}^L \zeta_{l_2}(k_1) z^{l_2+L}.
\]
Since $\zeta_{l_2^*}(k_1)\neq 0$, this polynomial in $z$ is non trivial, and so has at most $2L$ zeros. In other words, for $k_1\in E_j$ we have that
\[
 \#\{k_2\in\{1,\dots,2^{J-j}\}:(g,\psi_{j,k}^1)_2\neq 0\}\ge 2^{J-j}-2L.
\]
Recalling that $\#E_j\ge 2^{J-j}-2L$, this implies that $\#T_j\ge (2^{J-j}-2L)^2$. Finally, the result immediately follows from \eqref{eq:T_j}.
\end{proof}

We are now ready to prove Proposition~\ref{prop:new_uncertainty}. The proof follows the same argument used for Proposition~\ref{prop:normalized_UP}.
\begin{proof}[Proof of Proposition~\ref{prop:new_uncertainty}]
Define $D=\frac{2}{3}(1+C_\xi^{-1})$ (where $C_\xi$ is given by Lemma~\ref{lem:xi}) and take $q\in\R^n$ with $\norm{A_f q}_2>D$ and $\norm{\,^t\!A_g^\perp A_f q}_2\le 2/3$. Lemma~\ref{lem:properties} part 1 and the identity $A_f q=A_g(\,^t\!A_g A_f q)+A_g^\perp(\,^t\!A_g^\perp A_f q)$ yield
\[
 D<\norm{A_f q}_2=\bigl(\norm{\,^t\!A_g A_f q}_2^2 +\norm{\,^t\!A_g^\perp A_f q}_2^2\bigr)^{\frac{1}{2}}\le \norm{\,^t\!A_g A_f q}_2+ 2/3
\]
whence by construction of $D$ we obtain
\begin{equation}\label{eq:last2}
 \norm{\,^t\!A_g A_f q}_2>2C_\xi^{-1}/3.
\end{equation}
Set $v=\,^t\!A_g A_f q/\norm{\,^t\!A_g A_f q}_2$. By Lemma~\ref{lem:new_uncertainty}, the assumptions of Lemma~\ref{lem:xi} are satisfied with $p=\sum_{j=1}^{J-B-1} (2^{J-j}-2L)^2$ and $\zeta=m_g$; as a result $\xi_p(\,^t\!A_f A_g v)\ge C_\xi$. In other words, there exist $\alpha_1<\dots<\alpha_p$ such that $|(\,^t\!A_f A_g v)(\alpha_i)|\ge C_\xi$. Morevoer,  by Lemma~\ref{lem:properties} parts 1 and 2 we have $\norm{\,^t\!A_f A_g^\perp\,^t\!A_g^\perp A_f q}_2\le2/3$, whence $\norm{\,^t\!A_f A_g^\perp\,^t\!A_g^\perp A_f q}_\infty\le2/3$. As a consequence, by \eqref{eq:last2} and the identity $q=\norm{\,^t\!A_g A_f q}_2\,^t\!A_f A_g v +\,^t\!A_f A_g^\perp\,^t\!A_g^\perp A_f q$ we obtain
\[
 |q(\alpha_i)|\ge \norm{\,^t\!A_g A_f q}_2 |\,^t\!A_f A_g v(\alpha_i)| -|\,^t\!A_f A_g^\perp\,^t\!A_g^\perp A_f q(\alpha_i)|>\frac{2}{3}C_\xi^{-1}C_\xi-\frac{2}{3}=0,
\]
for every $i=1,\dots,p$. In other words, $\norm{q}_0\ge p$, as desired.
\end{proof}

\bibliographystyle{abbrv}
\bibliography{2015-disjoint_sparsity}

\end{document}